\documentclass[11pt]{amsart}
\usepackage[babel]{csquotes}
\usepackage{enumitem}
\usepackage{amsmath,amsthm,amssymb,mathrsfs,amsfonts,verbatim,enumitem,color,leftidx}
\usepackage{mathabx}
\usepackage[left=2.9cm,right=2.9cm,top=3.3cm,bottom=3.4cm,a4paper]{geometry}
\usepackage{tikz}
\usepackage{etoolbox} 
\usepackage{bbm}
\usepackage[all,tips]{xy}
\usepackage{graphicx,ifpdf}
\usepackage{stmaryrd}
\ifpdf
\fi

\usepackage[right,displaymath,mathlines]{lineno}
\usepackage{pgfmath}
\usepackage[colorlinks]{hyperref}
\hypersetup{
linkcolor=blue,          
citecolor=green,      
}

\usepackage{tikz}
\usetikzlibrary{arrows,snakes,backgrounds}

\newtheorem{thm}{Theorem}[section]
\newtheorem{lem}[thm]{Lemma}
\newtheorem{coro}[thm]{Corollary}
\newtheorem{prop}[thm]{Proposition}

\theoremstyle{definition}
\newtheorem{defn}[thm]{Definition}

\newtheorem{remark}[thm]{Remark}

\theoremstyle{remark}

\numberwithin{equation}{section}

\definecolor{esperance}{rgb}{0.0,0.5,0.0}

\newcommand{\R}{\mathbb{R}}

\newcommand{\G}{\Gamma}

\newcommand{\eps}{\epsilon}

\newcommand{\cA}{\mathcal{A}}

\newcommand{\cC}{\mathcal{C}}

\newcommand{\cG}{\mathcal{G}}

\newcommand{\cM}{\mathcal{M}}

\newcommand{\cO}{\mathcal{O}}

\newcommand{\cT}{\mathcal{T}}

\newcommand{\bR}{\mathbb{R}}
\newcommand{\bZ}{\mathbb{Z}}

\newcommand{\ra}{\rightarrow}

\newcommand{\onto}{\xymatrix{\ar@{>>}[r]&}}

\newcommand{\g}{\gamma}
\newcommand{\D}{\Delta}

\newcommand{\fl}{f_\lambda}

\begin{document}

\title{Local limit theorem of Brownian motion on metric trees}
\author{Soonki Hong}
\address{Soonki Hong}
\curraddr{Department of Mathematics\\ Postech \\ Cheongam-ro, Pohang, 37673 \\ Republic of Korea}
\email{soonkihong@postech.ac.kr}

\begin{abstract}
Let  $\mathcal{T}$ be a locally finite tree whose geometric boundary has infinitely many points. Suppose that a non-amenable group $\G$ acts isometrically and geometrically on the tree $\mathcal{T}$. 
 In this paper, we show that if the length spectrum is Diophantine, then there exists a continuous function $C$ on $\mathcal{T}^2$ such that the heat kernel $p(t,x,y)$ of $\mathcal{T}$ satisfies
 $$\lim_{t\rightarrow \infty}t^{3/2}e^{\lambda_0t}p(t,x,y)=C(x,y)$$
  for any $x,y\in \mathcal{T}$. Here, $\lambda_0$ is the bottom of the spectrum of the Laplacian on $\mathcal{T}$. 
 \end{abstract}
 \maketitle

\section{Introduction}\label{sec:1}
Let $\mathcal{T}$ be a locally finite metric tree with the path metric $d$. The geometric boundary $\partial \mathcal{T}$ of the tree $\mathcal{T}$ is the collection of the equivalence classes of geodesic rays in $\mathcal{T}$ up to bounded Hausdorff distance. Suppose that the degree of every vertex is larger than 2 and the geometric boundary $\partial\mathcal{T}$ consists of infinitely many points. Suppose that a discrete group $\G$ acts isometrically and geometrically (i.e. properly and cocompactly) on $\mathcal{T}$. 

To define Brownian motion on $\mathcal{T}$, we choose a Dirichlet form on graphs in \cite{BSSW} (see Section \ref{sec:2.2}). The Dirichlet form allows us to define the Laplacian $\Delta$ and the heat kernel $p(t,x,y)$ on graphs. The heat kernel is the fundamental solution of the heat equation, i.e. $$\Delta p(t,x,y)=\frac{d}{dt}p(t,x,y)\text{ and }\underset{t\ra0+}\lim p(t,x,y)=\delta_x(y),$$  where $\delta_x(y)$ is the Dirac delta mass at $x$. The heat kernel $p(t,x,y)$ is the probability density function of our Brownian motion.

Our main result is the local limit theorem of Brownian motion on $\mathcal{T}$ which describes the asymptotic behavior of the heat kernel as $t\ra\infty$ under some assumptions related to the length spectrum of $\Gamma$. The \textit{translation length} $l_\g$ of an element $\g\in\G$ is defined by
 $$l_\g:=\inf\{d(x,\g x): x\in \mathcal{T}\}.$$  The \textit{length spectrum} $\Lambda_\G$ of $\G$ is the additive subgroup of $\mathbb{R}$ generated by $\{l_\g:\g\in \G\}$.
\begin{defn}\label{def:1.1} The length spectrum $\Lambda_\Gamma$ is \textit{Diophantine} if there exist two elements $\g$ and $\g'$ such that the following property holds. There exist positive numbers $C$ and $\beta$ such that for all $p,q\in \mathbb{Z}$ with $q>0$,
 $$\left|\frac{l_\gamma}{l_{\gamma'}}-\frac{p}{q}\right|>C q^{-\beta}.$$ 
 \end{defn}
 The following theorem is the main theorem of this article.
\begin{thm}\label{thm:1.3}\emph{(Local Limit Theorem)} Let $\mathcal{T}$ be a topologically complete locally finite metric tree. Every degree of a vertex is at least 3. Suppose that the cardinality of $\partial T$ is infinite and a discrete group $\G$ acts isometrically and geometrically on $\mathcal{T}$. Suppose that the length spectrum of $\Gamma$ is a dense subgroup of $\mathbb{R}$ and is Diophantine. There exists a function $C:\mathcal{T}\times\mathcal{T}\rightarrow \mathbb{R}$ such that for any $x,y\in\mathcal{T}$,
\begin{equation}
\lim_{t\rightarrow\infty}t^{3/2}e^{\lambda_0 t}p(t,x,y)=C(x,y).
\end{equation}
\end{thm}

Various authors have shown the local limit theorem of random walk on hyperbolic groups and Brownian motion on Riemannian manifolds (\cite{La}, \cite{GL}, \cite{G} and \cite{LL}) since Bougerol proved the local limit theorem of Brownian motion on symmetric spaces of non-compact type (\cite{Bo}). Using the thermodynamic formalism, the authors of \cite{GL}, \cite{G} and \cite{LL} obtained a continuous function $c(x,y)$ described by 
\begin{equation}\label{eq:1.1}
\lim_{\lambda\ra\lambda_0-} \sqrt{\lambda_0-\lambda}\frac{\partial}{\partial \lambda}G_\lambda(x,y)=c(x,y).
\end{equation}
Utilizing the equation \eqref{eq:1.1} and Hardy-Littlewood Tauberian theorem, the authors proved the local limit theorem. In particular, the proof of Theorem \ref{thm:1.3} is analogous to the method in \cite{LL}. The remaining part of this section is devoted to introducing the key tools for the proof of Theorem \ref{thm:1.3}.

The \textit{$\lambda$-Green function} $G_\lambda:\mathcal{T}\times\mathcal{T}\rightarrow \mathbb{R}$ is defined by 
$$G_\lambda(x,y):=\int_{0}^{\infty}e^{\lambda t}p(t,x,y)dt.$$ The \textit{$\lambda$-Martin kernel} $k_\lambda:\mathcal{T} \times \mathcal{T} \ra \mathbb{R}$ is a map defined by
$$k_\lambda(x_0,x,y)=\frac{G_\lambda(x,y)}{G_\lambda(x_0,y)}.$$
The \textit{$\lambda$-Martin boundary} $\partial_\lambda \mathcal{T}$ is the boundary of the closure of the image of the embedding defined by $\lambda$-Martin kernel $y\mapsto k_\lambda(x_0,\cdot,y)$ on $\mathcal{T}$. The author of this article and Lim proved the convergence of $\lambda$-Green function and characterized $\lambda$-Martin boundary (\cite{HL}). Since the measures to be introduced in this article are related to the $\lambda$-Martin kernel, the characteristic of $\lambda$-Martin kernel plays an important role to prove Theorem \ref{thm:1.3}. 

Following the methods in \cite{PPS} and \cite{BPP}, we construct Gibbs measure $\overline{m}_\lambda$ of a potential function $f_\lambda$ defined by 

\begin{equation}\label{eq:1.2}
f_\lambda(g):=-2\lim_{t\ra0+}\frac{\log k_\lambda(g(0),g(t),g_+)}{t}
\end{equation}
for all geodesic line $g$, where $g_+=\underset{t\rightarrow \infty}\lim g(t)$. By the chain rule, we have
$$f_\lambda(g)=-2\lim_{t\rightarrow 0+}\frac{k_\lambda(g(0),g(t),g_+)-k_\lambda(g(0),g(0),g_+)}{t}.$$ 

To construct Gibbs measure, we verify the existence of Patterson-Sullivan density of $f_\lambda$, which is a family of measures $\mu_{x,\lambda}$ on $\partial\mathcal{T}$ satisfying that for any $x,y\in \mathcal{T}$ and $\gamma\in \Gamma$
 and $\xi \in \partial \mathcal{T},$
$$\g_*\mu_{x,\lambda}=\mu_{\g x,\lambda} \text{ and }\frac{d\mu_{y,\lambda}}{d\mu_{x,\lambda}}(\xi)=k_\lambda^2(x,y,\xi)e^{\delta_\lambda\beta_{\xi}(x,y)},$$
where $\delta_\lambda$ is the critical exponent (see Definition \ref{def:3.1}) and $\beta_\xi(x,y)$ is Busemann cocycle (see Section \ref{sec:2.1}). Patterson constructed the density and applied the density to prove that the critical exponent of a Fuchsian group acting on $\mathbb{H}^2$ coincides with the Hausdorff dimension on the limit set of the Fuchsian group (\cite{P}). Sullivan proved the same result when a discrete group acts convex cocompactly on $\mathbb{H}^{d}$ (\cite{S}). Coornaert generalized the result of Patterson and Sullivan (\cite{C}) to $\delta$-hyperbolic space. Ledrappier constructed Patterson-Sullivan density associated with H\"older continuous functions on the unit tangent bundle of Riemannian manifolds (\cite{Le}). 

Using Hopf parametrization (see Section \ref{sec:2.1}), Patterson-Sullivan density and N\"aim kernel defined by \eqref{3.13}, we obtain a Gibbs measure $m_\lambda$ of $f_\lambda$ on $\mathcal{G}\mathcal{T}$ defined for any $g\in \mathcal{G}\mathcal{T}$ by
$$m_\lambda(g):=\theta_x^\lambda(g_-,g_+)^2e^{2\delta_\lambda(g_-|g_+)_x}d\mu_{x,\lambda(g_-)} d\mu_{x,\lambda}(g_+)dt,$$
where $g_\pm:=\underset{t\rightarrow \infty}\lim g(\pm t)$, $g(t)$ is the closest point to $x$ on $g$ and $(g_-|g_+)_x$ is Gromov product of $g_-$ and $g_+$ based on $x$ (see Section \ref{Sec:2}). Since $m_\lambda$ is $\Gamma$-invariant, we have a probability measure $\overline{m}_\lambda$ on $\Gamma\backslash \mathcal{G}\mathcal{T}$ induced by $m_\lambda$. Although not all Gibbs measures of H\"older continuous functions are equilibrium states (see Definition \ref{def:4.5}), Gibbs measure coincides with the equilibrium state under certain assumptions. Thus the construction of Gibbs measure is important. Roblin constructed Gibbs measure of a constant function and the measure is equal to the entropy maximizing measure, known as the Bowen-Margulis measure (\cite{R}). Following Roblin's construction, Gibbs measures of H\"older continuous functions on the unit tangent bundle of quotient spaces of negatively curved manifolds and trees were built in \cite{PPS} and \cite{BPP}, respectively. In Section \ref{Sec:4}, we introduce the coding of $\Gamma\backslash\mathcal{G}\mathcal{T}$ in \cite{BPP} and show that $\overline{m}_\lambda$ satisfies the following theorem.

\begin{thm}\label{thm:1.4}Let $\mathcal{T}$ be a topologically complete locally finite metric tree. Every degree of a vertex is at least 3. Suppose that the cardinality of $\partial T$ is infinite and a discrete group $\G$ acts isometrically and geometrically on $\mathcal{T}$. The measure $\overline{m}_\lambda$ is the equilibrium state of $f_\lambda$ and the pressure $P_\lambda$ of $f_\lambda$ is equal to $\delta_\lambda$ (see Definition \ref{def:4.5}). 
\end{thm}

Applying the results in \cite{D} and \cite{Me}, the authors of \cite{LL} proved the uniform rapid mixing property of Gibbs measure of $f_\lambda$ on the unit tangent bundle of closed manifolds. In \cite{BPP}, under certain assumption, the rapid mixing property also appears as the application of results in \cite{D} and \cite{Me}. In Section \ref{Sec:5}, we verify that the results \cite{LL} and \cite{BPP} guarantee the uniform rapid mixing property of a Gibbs measure of $f_\lambda$ on $\mathcal{G}\mathcal{T}$ (see Theorem \ref{thm:5.4}). 

This article is constructed as follows. In Section \ref{Sec:2}, we recall the definitions and the properties of the geometry of trees and discuss the result of Brownian motion on graphs in \cite{HL}. In Section \ref{Sec:3}, we construct Patterson-Sullivan density and Gibbs measure of $f_\lambda$. In Section \ref{Sec:4}, we show that Gibbs measure $\overline{m}_\lambda$ is the equilibrium state of $f_\lambda$. In Section \ref{Sec:5}, we verify that Gibbs measure has uniform rapid mixing property under the assumption in Theorem \ref{thm:1.3}. Applying the uniform rapid mixing, we obtain formulae to show Theorem \ref{thm:1.3}. In Section \ref{Sec:6}, using the results in section \ref{Sec:5}, we show the Theorem \ref{thm:1.3}.
\subsection*{Acknowledgments} We would like to thank S. Lim for his encouragement and helpful suggestions. This paper is supported by Basic Science Research Institute Fund under NRF grant number 2021R1A6A1A10042944 and National Research Foundation of Korea(NRF) grant funded by the Korea government(MSIT). (2020R1A2C1A01005446).

\section{Preliminaries}\label{Sec:2}
\subsection{The geometry of trees}\label{sec:2.1}
In this section, we summarize the geometry of trees. We refer to \cite{BH} and \cite{BPP} for details. The assumption is the same as in the beginning of Section \ref{sec:1}. Let $V$ be the set of vertices of $\mathcal{T}$ and let $E$ be the set of edges of $\mathcal{T}$. Let $l_e$ be the length of an edge $e$. Denote by $l_m$ and $l_M$ the minimal and the maximal edge length of $\mathcal{T}$, respectively.  A geodesic segment, a ray and a line is the isometry from a close interval $[0,l]$, a ray $[0,+\infty)$ and a line $(-\infty,\infty)$ to $\mathcal{T}$, respectively. Denote by $[x,y]$ the geodesic segment from $x$ to $y$. Since the quotient space $\G\backslash \mathcal{T}$ is compact and locally finite, $l_m$ and $l_M$ are finite and positive. Denote by $d_M$ the maximal degree of vertices. Under the assumption, $d_M$ is finite. Let $S(x,r)$ and $B(x,r)$ be the sphere and the ball of radius $r$ centered at $x$, respectively. Denote by $Diam(A)$ the diameter of a connected subset $A$ of $\mathcal{T}$. Fix a connected fundamental domain $T_0$ of $\G$ in $\mathcal{T}$. 

Denote $\overline{\mathcal{T}}:=\mathcal{T}\cup\partial \mathcal{T}$. Since the action of $\G$ is cocompact, the limit set $\Lambda \G$ of $\G$ (i.e. the accumulation points in $\partial \mathcal{T}$ of $\G$-orbit in $\mathcal{T}$) coincides with $\partial \mathcal{T}$. The \textit{shadow} $\cO_x(y)$ of a point $y\in \mathcal{T}$ seen from $x$ is the set of end points $\xi$ in $\partial \mathcal{T}$ of the geodesic rays from $x$ 
 passing through $y$. Busemann cocycle $\beta:\partial \mathcal{T}\times\mathcal{T}\times \mathcal{T} \ra \bR$ is defined by
 $$\beta_{\xi}(x,y):=\displaystyle\lim_{z\ra+\xi} \{d(x,z)-d(y,z)\}$$
for all $(\xi,x,y)\in \partial\mathcal{T} \times\mathcal{T}\times \mathcal{T}$. Gromov product with a base point $x$ is defined by 
 $$(y|z)_x:=\frac{1}{2}\{d(x,y)+d(x,z)-d(y,z)\}$$
 for all $y,z\in \mathcal{T}$. Gromov product on $\partial{\mathcal{T}}$ with a base point $x$ is also defined as follows: for any $\xi,\zeta\in \overline{\mathcal{T}}$,
 $$(\xi|\zeta)_x:=\lim_{y\ra\xi,z\ra\zeta}\frac{1}{2}\{d(x,y)+d(x,z)-d(y,z)\}.$$
 For any $x,y\in \mathcal{T}$ and $\xi,\zeta\in\partial \mathcal{T}$, 
 \begin{equation}\label{eq:2.1}
2(\xi|\zeta)_x=2(\xi|\zeta)_y+\beta_{\xi}(x,y)+\beta_{\zeta}(x,y).
\end{equation}
 The visual distance $d_x$ between two points $\xi$ and $\zeta$ in $\partial \mathcal{T}$ with a base point $x$ is defined by
 $$d_x(\xi,\zeta):=e^{-(\xi|\zeta)_x}.$$
Let $\cG \mathcal{T}$ be the space of geodesic lines in $\mathcal{T}$. Denote $\cG_x\mathcal{T}:=\{g\in \cG\mathcal{T}:g(0)=x\}$. The distance $d_{\cG\mathcal{T}}$ between $g$ and $g'$ is defined by
$$d_{\cG\mathcal{T}}(g,g'):=\int_{-\infty}^\infty d(g(t),g'(t))e^{-2|t|}dt.$$
The geodesic flow $\phi_t$ on $\cG\mathcal{T}$ is defined by
 $$\phi_tg(s):=g(s+t)$$
 for all $g\in\cG\mathcal{T}$. Define a map $\pi:\cG\mathcal{T}\ra\mathcal{T}$ as follows: for all $g\in \cG\mathcal{T}$,
$$\pi(g):=g(0).$$
 Denote $\partial^2\mathcal{T}:=\{(\xi,\zeta)\in\partial\mathcal{T}\times\partial\mathcal{T}:\xi\neq\zeta\}$.  The space $\cG\mathcal{T}$ is identified with $\partial^2 \mathcal{T}\times \bR$ via Hopf parametrization with the base point $x$:
 $$g\mapsto (g_-,g_+,t),$$
  where $g_{\pm}:=\displaystyle\lim_{s\ra\pm\infty}g(s)$ and $g(t)$ is the closest point to $x$ on $g$. Denote $f\asymp_c g$ if $\frac{1}{c}g\leq f \leq cg.$ The following lemma will be used to show H\"older continuity of $f_\lambda$ (see Proposition \ref{prop:2.11}).
\begin{lem}\label{lem:2.1} \emph{(\cite{BPP} Lemma 3.4)} There exist constants $\epsilon$ and $C$ satisfying the following property. For all geodesic lines $g,g'$ with $d_{\cG\mathcal{T}}(g,g')<\varepsilon$, there exist constants $a<0<b$ and $s$ such that $g([a,b])=g(\mathbb{R})\cap g'(\mathbb{R})$ and $g(t)=g'(t+s)$ for all $t\in [a,b]$, and
\begin{displaymath}
d_{\cG\mathcal{T}}(g,g')\asymp_C d_{\pi(g)}(g_-,g_-')^2+d(\pi(g), \pi (g'))+d_{\pi(g)}(g_+,g_+')^2.
\end{displaymath}
\begin{figure}[h]
\begin{center}
\begin{tikzpicture}[scale=1]
  \draw [thick](-3,0) -- (3,0);   \draw [thick](-4.5,-0.2) -- (-3,0);\draw[thick](-4.5,0.2) -- (-3,0);\draw[thick] (4.5,-0.2) -- (3,0);\draw [thick](4.5,0.2) -- (3,0);
 \node at (-2.7,-0.3) {$=g(a)$};  \node at (-2.5,0.3) {$g'(a+s)$};   \node at (2.7,0.3) {$g'(b+s)$};\node at (0,-0.3) {$g(0)$};  \node at (0.4,0.3) {$g'(0)$};\node at (2.6,-0.3) {$=g(b)$};\node at (-4.8,-0.3) {$g'_-$};\node at (-4.8,0.3) {$g_-$};\node at (4.8,-0.3) {$g'_+$};\node at (4.8,0.3) {$g_+$}; \fill (-3,0)    circle (2pt);\fill (-4.5,0.2)    circle (2pt); \fill (3,0)    circle (2pt);\fill (4.5,0.2)    circle (2pt); \fill (0,0)    circle (2pt);\fill (4.5,-0.2)    circle (2pt);\fill (-4.5,-0.2)    circle (2pt);\fill (0.4,0) circle (2pt);
\end{tikzpicture}
\end{center}
\caption{Two geodesic lines $g_1$ and $g_2$ with $d_{\cG\mathcal{T}}(g,g')<\varepsilon$}\label{figure1}
\end{figure}
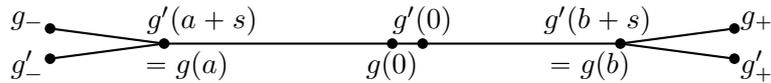
\end{lem}
\subsection{Laplacian and heat kernel on trees}\label{sec:2.2} In this section, using a Dirichlet forms on $\mathcal{T}$, we recall the definition of the Laplacian of $\mathcal{T}$. See \cite{BSSW} and \cite{HL} for detail.

Choose a $\Gamma$-invariant orientation of edges of $\mathcal{T}$ and denote by $i(e)$ and $t(e)$ the initial vertex and the terminal vertex of an edge $e$, respectively. Let $e^o$ be the open edge of $e$. Let $e_s$ be a point on $e$ satisfying $d(i(e),e_s)=s$. Since every edge $e$ is isometric to a closed interval $[0,l_e]$, the restriction $f|_e$ of a function $f$ on $e$ is often considered as a function on $[0,l_e]$. Thus it is possible that $f|_e$ has the derivative at a point in $e^o$. 

 The space $C^\infty(\mathcal{T})$ is the space of continuous functions $f$ such that for any $e\in E$ and $k\in \mathbb{N}$, the restriction $f|_e$ has a continuous $k$-th derivative on $e^o$ and $\underset{e^o}\sup|(f|_e)^{k}|<\infty.$  For all $f\in C^\infty(\mathcal{T})$, the derivatives of $f|_e$ at the vertices exist. However, the derivative $(f|_e)^{(k)}(v)$ depends on each edge $e$ containing $v$.
 The space $C_c^\infty(\mathcal{T})$ is the subspace of compactly supported functions in $C^\infty(\mathcal{T})$.
 
 The integral of a function $f$ is defined by
 $$\int_{\mathcal{T}} fd\mu:=\sum_{e\in E} \int_{0}^{l_e} f(e_s)ds,$$
 where $ds$ is the Lebesgue measure on the real line $\mathbb{R}$. Since the orientation of edges is $\Gamma$-invariant, for any integrable function $f$,
 \begin{equation}
 \nonumber\int_{\mathcal{T}} fd\mu=\sum_{\gamma\in\Gamma}\int_{T_0}f(\gamma x)d\mu(x),
 \end{equation}
 where $T_0$ is a fundamental domain.
 
 Let $O$ be an open set in $\mathcal{T}$. \textit{Sobolev space} $W^1({O})$ is the space of $L^2$-functions $f$ with the following properties:
 \begin{itemize}
 \item Every restriction $(f|_e)$ has the first weak derivative on $e^o$ and $\displaystyle\int_{e^o\cap O} |f|_e'(s)|^2 ds<\infty$.
 \item $\|f'\|^2_{L^2({O})}=\displaystyle\int_{O} |f'|^2 d\mu<\infty.$
 \end{itemize} 
 Denote by $W_0^1({O})$ the closure of $C_c^\infty({O})$ in $W^1({O})$ with respect to the norm defined by $\|f\|_{W^1({O})}=(\|f\|_{L^2({O})}+\|f'\|_{L^2({O})}\|)^{1/2}$. The space $W_{loc}^1(O)$ is the space of functions $f$ such that for any compact set $K\subset O$, there exists a function $f_1\in W_0^1(O)$ satisfying $f|_K=f_1|_K$.
 
 The space $W_0^1(\mathcal{T})$ is the domain of our \textit{Dirichlet form} $\mathcal{E}$ on $\mathcal{T}$ and Dirichlet form $\mathcal{E}$ is defined by 
 $$\mathcal{E}(f,g)=\int_{\mathcal{T}} f'g'd\mu$$
 for all $f,g\in W_0^1(\mathcal{T}).$ 
  
 The domain $Dom(\Delta)$ of the \textit{Laplacian} $\Delta$ of $\mathcal{T}$ is the space of function $f\in W^1_0(\mathcal{T})$ with the following property: There exists a constant $C_f$ such that for all $h\in W_0^1(\mathcal{T}),$
 $\mathcal{E}(f,g)\leq C_f\|h\|_{L^2(\mathcal{T})}.$ By Riesz representation theorem, for any $f\in Dom(\Delta)$, there exists a $L^2$-function $u$ such that for any $h\in W_0^1(\mathcal{T}),$ 
 $$\mathcal{E}(f,g)=-(u,h).$$
 Then we define Laplacian $\Delta f$ of $f$ to be $u$. Under the assumption in Section \ref{sec:1}, the heat kernel $p(t,x,y)$ exists and is a positive continuous function. The heat kernel satisfies
 \begin{equation}\label{eq:2201283}
 \int_{\mathcal{T}} p(t,x,y)d\mu(y)=1 \text{ and } \int_{\mathcal{T}} p(t,x,y)p(s,y,z)d\mu(y)=p(t+s,y,z)
 \end{equation}
 (see \cite{BSSW} Theorem 4.3 and Theorem 5.12). The map $y\mapsto p(t,x,y)$ is in $C^\infty(\mathcal{T})$ for fixed $t>0$ and $x\in \mathcal{T}$(see \cite{BSSW} Theorem 5.23). The next lemma plays a crucial rule to show local limit theorem and follows from the results in \cite{St1,St2, St3} as an application of Moser iteration techniques.
 \begin{lem} \label{bssw}\emph{(\cite{BSSW} Theorem 4.2)} Let $\mathcal{T}$ be a topologically complete locally finite metric tree. For any compact set $K$, point $x\in K$ and compact interval $[a,b]$, there exist constants $\alpha$ and $C$ such that for any $y\in \mathcal{T}$,
 $$\sup\left\{\frac{|p(t,x_1,y)-p(t,x_2,y)|}{d(x_1,x_2)^\alpha}:t\in [a,b],x_1,x_2\in K\right\}\leq Cp(2b,x,y).$$
 \end{lem}
 \subsection{Properties of Green function}\label{sec:2.2} In this section, we recall the properties of Green functions and show that $f_\lambda$ in \eqref{eq:1.2} is H\"older continuous. 

The assumption is the same as in Section \ref{sec:1} holds. By Theorem 8.5 in \cite{SW}, the bottom of $L^2$-spectrum $\lambda_0$ of $-\D$ is positive. By Corollary 3.12 and Theorem 3.15 in \cite{HL}, for any $\lambda\in (-\infty,\lambda_0]$, $\lambda$-Green function is finite. A function $f\in W_{loc}^1(O)$ on a connected open subset $O$ of $\mathcal{T}$ is $\lambda$-\textit{harmonic} if $f$ is the weak solution of the equation $-\D f=\lambda f$, i.e for any compactly supported function $g \in W^1(O),$  
$$\mathcal{E}(f,g)=\lambda(f,g).$$

For any $r,l>0$ and $\lambda\in [0,\lambda_0]$, there exists a constant $D_{r,l}$ such that for any positive $\lambda$-harmonic function $f$ on $B(x,r+l)$ and $y,z\in B(x,r)$,
\begin{equation}\label{eq:2.2}{f(y)/}{f(z)}<e^{D_{r,l}}\end{equation}
(\cite{HL} Corollary 3.4). The inequality \eqref{eq:2.2} is called Harnack inequality.
\begin{lem}\label{lem:2.2}Let $f$ be a positive $\lambda$-harmonic function on $B(x,r_1+r_2)$. For any $\lambda\in [0,\lambda_0]$, $r_1,r_2\in(0, l_m/4]$ and $y,z\in  B(x,r_1)$, we have
\begin{equation}\label{eq:2.3}
\frac{f(y)}{f(z)}\leq e^{4d_M\sqrt{r_1/r_2}}
\end{equation}
\end{lem}
\begin{proof} In the proof of Corollary 3.4 in \cite{HL}, we have $D_{r,l}=\underset{x}\max\sqrt{2|S(x,r+l)|\mu(B(x,r))/l},$ where $|S(x,r+l)|$ is the number of points in $S(x,r+l).$ 
 Since $r_1+r_2\leq l_m/2$, $$|S(x,r_1+r_2)|\leq d_M\text{ and }\mu(B(x,r_1))\leq d_Mr_1.$$ Thus we have \eqref{eq:2.3}.
\end{proof}
\begin{prop}\label{prop:2.3} \emph{(\cite{HL} Corollary 3.16)}\label{prop:2.3} For any $\lambda\in(-\infty,\lambda_0]$ and $x\in \mathcal{T}$, the function $y\mapsto G_\lambda(x,y)$ is a positive $\lambda$-harmonic function on $\mathcal{T}\backslash\{x\}$ and is contained in $C^\infty(\mathcal{T}\backslash\{x\})$.
\end{prop}
Using the above proposition, we show the following lemma. The lemma will be used to prove Theorem \ref{thm:6.3}.
\begin{lem}\label{lem:2.4}For any $\lambda\in[0,\lambda_0]$, distinct points $x,y\in \mathcal{T}$, and connected compact set $K$,
$$\int_K G_\lambda (x,z)d\mu(z)<\infty.$$ 
For any $R>0$, there exist constant $C_R$ such that for any $\lambda\in[0,\lambda_0]$ and $x\in \mathcal{T}$,
$$\int_{B(x,R)} G_\lambda (x,z)d\mu(z)<C_R.$$ 
\end{lem}
\begin{proof} For any compact set $K$, there exists a ball $B(x,r)$ containing $K$. Thus it is enough to show the existence of the constant $C_R$. By Proposition \ref{prop:2.3}, $\int_K G_\lambda(x,z)d\mu(z)<\infty$ if $x\notin K$. For any $\lambda\in[0,\lambda_0]$, $\g\in \G$ and $x,y\in \mathcal{T}$, $G_\lambda(x,y)$ is equal to $G_\lambda(\g x, \g y)$ and $G_\lambda(x,y)\leq G_{\lambda_0}(x,y)$. Hence we will consider the case when $\lambda=\lambda_0$ and $x$ is in the closure of the fundamental domain ${T_0}$. Suppose that there exist $R$ and a sequence $\{x_n\}$ in $\overline{T_0}$ such that for any $n$, $\int_{B(x_n,R)} G_{\lambda_0}(x_n,y)d\mu(y)\geq n$. Choose a compact set $K$ disjoint to $B(x,2R+Diam(T_0))$. By Harnack inequality \eqref{eq:2.2}, for any $x\in T_0$, we have
\begin{equation}\label{eq:2201281}
 \int_K G_{\lambda_0}(x,z)d\mu(z)\overset{ \eqref{eq:2.2} }\geq C\int_K G_{\lambda_0}(x_n,z)d\mu(z),
 \end{equation}
 where $C=e^{-D_{Diam(T_0),R}}$.
 By the definition of $\lambda$-Green function, we have the equation of \eqref{eq:2201282}. Changing the range of the integral, we obtain the inequality of \eqref{eq:2201282}.
\begin{equation}\label{eq:2201282}
\int_K G_{\lambda_0}(x_n,z)d\mu(z)=\int_K \int_0^\infty e^{\lambda_0 t} p(t,x_n,z)dtd\mu(z)\geq \int_K \int_0^\infty e^{\lambda_0 (t+1)} p(t+1,x_n,z)dt d\mu(z).
\end{equation}
 The continuity of the heat kernel guarantees that the number  $$m:=\min\{p(1,y',z'):(y',z')\in K\times \overline{B(x,R+Diam(T_0))}\}$$
 is bounded and positive. By the second equation in \eqref{eq:2201283}, we have the first equation of \eqref{eq:2201284}. Changing the range of the integral, we have the first inequality of \eqref{eq:2201284}. The remaining part of \eqref{eq:2201284} follows from the choice of $m$ and $\{x_n\}$.
\begin{equation}\label{eq:2201284} 
\begin{split}
&\int_K \int_0^\infty e^{\lambda_0 (t+1)} p(t+1,x_n,z)dt d\mu(z)\\
&\overset{\eqref{eq:2201283}}=\int_K\int_0^\infty\int_{\mathcal{T}}e^{\lambda_0 (t+1)}p(t,x_n,z)p(1,z,y)d\mu(z)dtd\mu(y)\\
&\geq  \int_K\int_0^\infty\int_{B(x_n,R)} e^{\lambda_0(t+1)}p(t,x_n,z)p(1,z,y)d\mu(z)dtd\mu(y)\\
&\geq me^{\lambda_0}\mu(K)\int_{B(x_n,R)}\int_0^\infty e^{\lambda_0 t}p(t,x_n,z)d\mu(z)\\
&=  me^{\lambda_0}\mu(K)\int_{B(x_n,R)} G_{\lambda_0}(x_n,z)d\mu(z)\geq me^{\lambda_0}\mu(K) n.
\end{split}
\end{equation} 
The ineqaulities \eqref{eq:2201281}, \eqref{eq:2201282} and \eqref{eq:2201284} show that any point $x\in T_0$, $\int_K G_{\lambda_0}(x,z)d\mu(z)$ diverges, which is a contradiction. Hence for any $R\geq 0,$ there exists $C_R$ such that for any $\lambda\in [0,\lambda_0]$ and $x\in \mathcal{T}$, $$\int_{B(x,R)}G_\lambda(x,y)d\mu(y)\leq C_R.$$ 
This completes the proof.
\end{proof}
As in \cite{Ac},\cite{GL}, \cite{G} and \cite{LL}, we also verified that the uniform Ancona inequality holds. 
\begin{thm} \emph{(\cite{HL} Theorem 1.1)}\label{thm:2.6} Let $\mathcal{T}$ be a topologically complete locally finite metric tree. Suppose that the cardinality of $\partial T$ is infinite and a discrete group $\G$ acts isometrically and geometrically on $\mathcal{T}$. There exists a constant $C>0$ such that for any $\lambda \in [0,\lambda_0]$ and three points $x,y,z$ with $y\in [x,z]$, $d(x,y)\geq 1$ and $d(y,z)\geq 1$,  
\begin{equation} \label{eq:2.4}
\frac{1}{C}G_\lambda(x,y)G_\lambda(y,z)\leq G_\lambda(x,z)\leq{C}G_\lambda(x,y)G_\lambda(y,z).
\end{equation}
\end{thm}
Using Ancona inequality, we obtain the following corollary.
\begin{coro}\label{coro2.7}\label{coro:2.7} \emph{(\cite{HL} Theorem 4.15)} Let $\mathcal{T}$ be a topologically complete locally finite metric tree. Suppose that the cardinality of $\partial T$ is infinite and a discrete group $\G$ acts isometrically and geometrically on $\mathcal{T}$. There exist constants $C>0$ and $\rho\in(0,1)$ such that for any $\lambda \in [0,\lambda_0]$ and geodesic segments $[x,y]$ and $[x',y']$ of which the intersection is the geodesic segment of length $n\geq 1$,   
\begin{equation}\label{eq:2.5}
\left|\frac{G_\lambda(x,y)/G_\lambda(x',y)}{G_\lambda(x,y')/G_\lambda(x',y')}-1 \right| \leq C\rho^n
\end{equation}
\begin{figure}[h]
\begin{center}
\begin{tikzpicture}[scale=1]
  \draw [thick](-3,0) -- (3,0);   \draw [thick](-4.5,-0.2) -- (-3,0);\draw[thick](-4.5,0.2) -- (-3,0);\draw[thick] (4.5,-0.2) -- (3,0);\draw [thick](4.5,0.2) -- (3,0);
\node at (-4.8,-0.3) {$x'$};\node at (-4.8,0.3) {$x$};\node at (4.8,-0.3) {$y'$};\node at (4.8,0.3) {$y$}; \fill (-4.5,0.2)    circle (2pt); \fill (4.5,0.2)    circle (2pt); \fill (4.5,-0.2)    circle (2pt);\fill (-4.5,-0.2)    circle (2pt); \node at (0,0.3) {$n$};\draw[snake=brace,thick] (-3,0.05) -- (3,0.05);
\end{tikzpicture}
\end{center}
\caption{Strong Ancona inequality}\label{figure2}
\end{figure}
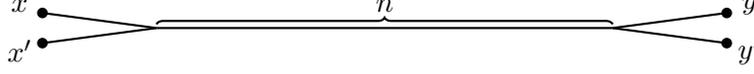
\end{coro}

 By Corollary \ref{coro:2.7}, for any $x,y\in \mathcal{T}$, $\xi \in \mathcal{T}$ and $\lambda\in [0,\lambda_0]$, the limit
$$k_\lambda(x,y,\xi):=\lim_{t\rightarrow\infty}k_\lambda(x,y,g(t))$$
exists, where $g$ is the geodesic ray from $x$ to $\xi\in\partial \mathcal{T}$. For any fixed $x$, the map define by $\xi\mapsto k_\lambda(x,\cdot,\xi)$ is the bijection from $\partial \mathcal{T}$ to $\partial_\lambda \mathcal{T}$ (\cite{HL} Theorem 4.17).
 \begin{coro}\label{coro:2.8}For any $\lambda\in [0,\lambda_0]$ and $(x,\xi)\in \mathcal{T}\times\partial \mathcal{T}$, the function $k_{\lambda,x,\xi}(y):=k_\lambda(x,y,\xi)$ is positive $\lambda$-harmonic on $\mathcal{T}$ and is in $C^\infty(\mathcal{T})$. 
\end{coro}
\begin{proof}Let $g$ be a geodesic ray from $x$ to $\xi$. Fix $T>0$ and a compact set $K$ with $g(t)\notin N_1(K)$ for any $t\geq T$, where $N_1(K)$ is the set of points $y$ with $d(y,K)\leq 1$. Harnack inequality \eqref{eq:2.2} shows that the family of the functions $k_{\lambda,x,g(t)}(y):= k_\lambda(x,y,g(t))$ on $K$ is equicontinuous and uniformly bounded. By Arzela-Ascoli's theorem (see \cite{Mun} Theorem 47.1), the function $ k_{\lambda,x,g(t)}(y)$ converges uniformly to $k_{\lambda,x,\xi}(y)$ on compact sets. Choose a function $f$ on $\mathbb{R}_{\geq 0}$ with $$f(0)=1,\quad f|_{[1,\infty)}\equiv 0\text{ and }|f'|\leq 2.$$ For a given compact set $K'$, define a smooth function function $f_{K'}$ by $f_{K'}(y)=f(d(K',y)).$ Obviously, $f_{K'}$ is a function in $C^\infty(\mathcal{T})$ whose support is $N_1(K')$.
Leibniz rule of Dirichlet forms (see page 171 in \cite{LSV}) shows the first equality below. The second equality follows from the definition of a $\lambda$-harmonic function. 
\begin{equation}
\begin{split}
\nonumber &\mathcal{E}(f_{K'}(k_{\lambda,x,g(t_1)}-k_{\lambda,x,g(t_2)}),f_{K'}(k_{\lambda,x,g(t_1)}-k_{\lambda,x,g(t_2)}))\\
&= \mathcal{E}(f_{K'}^2(k_{\lambda,x,g(t_1)}-k_{\lambda,x,g(t_2)}),k_{\lambda,x,g(t_1)}-k_{\lambda,x,g(t_2)})+\|f_{K'}'(k_{\lambda,x,g(t_1)}-k_{\lambda,x,g(t_2)})\|^2_{L^2(\mathcal{T})}\nonumber\\
&=\lambda\|f_{K'}(k_{\lambda,x,g(t_1)}-k_{\lambda,x,g(t_2)})\|^2_{L^2(\mathcal{T})}
+\|f_{K'}'(k_{\lambda,x,g(t_1)}-k_{\lambda,x,g(t_2)})\|^2_{L^2(\mathcal{T})}\nonumber\\
&\leq(\lambda_0+4)\|(k_{\lambda,x,g(t_1)}-k_{\lambda,x,g(t_2)})1_{N_1(K')}\|^2_{L^2(\mathcal{T})}.\nonumber
\end{split}
\end{equation}
The above inequality and \eqref{eq:2.5} implies that $\underset{t_1,t_2\rightarrow \infty}\lim\|f_{K'}(k_{\lambda,x,g(t_1)}-k_{\lambda,x,g(t_2)})\|_{W^1(\cT)}=0.$
Since $\mathcal{E}$ is a closed form, $f_{K'}k_{\lambda,x,g(t)}$ converges to a function in $W^1_0(\mathcal{T})$ which is supported by $N_1(K')$ and is equal to $k_{\lambda,x,\xi}$ on $K'$. Thus $k_{\lambda,x,\xi}$ is in $W_{loc}^1(\mathcal{T})$. Since $k_{\lambda,x,g(t)}$ is positive $\lambda$-harmonic,  $k_{\lambda,x,\xi}$ is harmonic. Corollary 3.2 in \cite{HL} shows that $k_{\lambda,x,\xi}$ is in $C^\infty(\mathcal{T})$. 
\end{proof}
By Corollary \ref{coro:2.8}, the $\G$-invariant function $f_\lambda$ on $\cG\mathcal{T}$ is well-defined for all $\lambda\in [0,\lambda_0]$. The remaining part of this section is devoted to show H\"older continuity of $f_\lambda$.

For any $x\in \mathcal{T}$, the probability measure $\mathbb{P}_x$ on the space $\Omega_x$ of continuous paths $\omega(t)$ starting at $x$ is defined for any $0<t_1<\cdots<t_n$ and Borel sets $B_n$ in $\mathcal{T}$, by
\begin{eqnarray}\label{eq:2.6}
\displaystyle &&\mathbb{P}_x[\omega({t_1)}\in B_1,\cdots, \omega({t_n})\in B_n]\\&=&\nonumber\displaystyle \int_{B_1\times \cdots \times B_n} p(t_1,x,y_1)p(s_1,y_1,y_2)\cdots p(s_n,y_{n-1},y_n)d\mu^n(y_1,\cdots,y_n),\end{eqnarray}
 where $\mu^n=\mu\times\cdots\times \mu$ and $s_i=t_i-t_{i-1}$. Since the Dirichlet form $\mathcal{E}$ is strong local and regular (\cite{BSSW} Theorem 3.28), the process induced by \eqref{eq:2.6} is the strong Markov process. (\cite{FOT} Theorem 7.3.1)

Denote $\mathbb{E}_x(f):=\int_{\Omega_x} fd\mathbb{P}_x.$ Then $\lambda$-Green function has the following property: For any $\lambda\in [0,\lambda_0]$ and measurable function $f$ on a Borel set $B\subset \mathcal{T}$,
\begin{equation}\label{eq:2.7}
\int_B G_\lambda(x,y)f(y)d\mu(y)=\mathbb{E}_x\left[\int_0^\infty e^{\lambda t}1_B(\omega(t))f(\omega(t))dt\right].
\end{equation}

 Let $t_y(\omega)$ be the exit time of the connected component of $\mathcal{T}\backslash\{y\}$ containing the starting point $\omega(0)$, i.e. $t_y(\omega):=\inf\{s\geq0:\omega({s})=y\}.$ For convenience, denote $$t_y:=t_y(\omega),1_{t_y<\infty}:=1_{\{\omega:t_y(\omega)<\infty\}}(\omega)\text{ and }e^{\lambda t_y(\omega)}:=e^{\lambda t_y}.$$ By \eqref{eq:2.7}, following the proof of Proposition 8.5 in \cite{LL}, for any distinct points $x,y,z$ on the same geodesic in order, we have
\begin{equation}\label{eq:2.8}
G_\lambda(x,z)=\mathbb{E}_{x}[1_{t_{y}<\infty}e^{\lambda t_y}G_\lambda(y,z)]=G_\lambda(y,z)\mathbb{E}_{x}[1_{t_y<\infty}e^{\lambda t_y}].
\end{equation}
The equation \eqref{eq:2.8} enables us to show the following lemma.
\begin{lem}\label{lem:2.9}
\begin{enumerate} 
\item For any $\lambda\in [0,\lambda_0]$ and distinct points $x,y$, the limit
$$A_\lambda(x,y)=\lim_{\epsilon\rightarrow0+}\frac{\mathbb{E}_{x_\epsilon}[1_{t_y<\infty}e^{\lambda t_y}]-\mathbb{E}_x[1_{t_y<\infty}e^{\lambda t_y}]}{\epsilon}$$ exists, where $x_\epsilon$ is the point on $[x,y]$ with $d(x,x_\epsilon)=\epsilon.$ Furthermore, there exist a constant $C$ such that $|A_\lambda(x,y)|\leq C$ for any $\lambda\in [0,\lambda_0]$ and $x,y$ with $d(x,y)=2l_M$.
\item For any $\lambda\in [0,\lambda_0]$ and $\epsilon>0$, there exists a constant $\delta$ satisfying the following property: 
Let $y$ be a point and let $e$ be an edge with $l_M\leq d(e,y)\leq 2l_M.$ For any $\lambda'\in [0,\lambda_0]$ with $|\lambda-\lambda'|<\delta$ and $x\in e$, we have $$| A_{\lambda}(x,y)|_e-A_{\lambda'}(x,y)|_e|<\epsilon.$$
\end{enumerate}
\end{lem}
\begin{proof}$(1)$ By Proposition \ref{prop:2.3}, the restriction $G_\lambda(x,y)|_e$ of the function $x\mapsto G_\lambda(x,y)$ on an edge $e$ has the derivative $\frac{\partial}{\partial x}G_\lambda(x,y)|_e$ at any interior point of $e\backslash\{y\}$. For any edge $e$ with $y\notin e$, $\frac{\partial}{\partial x}G_\lambda(x,y)|_e$ is bounded on $e^o$ and the derivatives at the vertices of $e$ exist. 

Let $x,y,z$ be distinct points on the same geodesic in order. For sufficiently small $\epsilon>0$, there exists an edge $e$ such that $[x,x_{\epsilon}]\subset e$ and $e\cap [x,y]\neq \phi$. Suppose that $z$ is a vertex with $0<d(y,z)\leq l_M$. By \eqref{eq:2.8}, we have
$$\underset{\epsilon\rightarrow0+}\lim\frac{\mathbb{E}_{x_\epsilon}[1_{t_y<\infty}e^{\lambda t_y}]-\mathbb{E}_x[1_{t_y<\infty}e^{\lambda t_y}]}{\epsilon}=\underset{\epsilon\rightarrow0+}\lim\frac{G_\lambda(x_\epsilon,z)-G_\lambda(x, z)}{\epsilon G_\lambda(y,z)}=\frac{\partial}{\partial x}\frac{G_\lambda(x,z)|_e}{G_\lambda(y,z)}.$$ 
Thus $A_\lambda(x,y)$ exists for all $x\neq y$.

Since the $\G$-action is cocompact and $G_\lambda(y,z)=G_\lambda(\g y, \g z)$, for all $y,z\in\mathcal{T}$ and $\g\in \G$, Lemma 4.4 in \cite{HL} shows that there exists a positive constant $c$ such that for all $y,z$ with $d(y,z)\leq l_M$ and $\lambda\in[0,\lambda_0]$, $G_\lambda(y,z)\geq c$. The family of functions $x\mapsto \lambda G_\lambda(x,z)|_e$ converges uniformly to $x\mapsto\lambda' G_{\lambda'}(x,z)|_e$ on $e$ with $z\notin e$ as $\lambda$ goes to $\lambda'$. Obviously, $\frac{\partial^2}{\partial x^2}  G_{\lambda}(x,z)|_e=\lambda G_\lambda(x,z)|_e$ for all $\lambda\in[0,\lambda_0]$ and $\frac{\partial}{\partial x}G_{\lambda}(x,z)|_e$ also converges uniformly to
$\frac{\partial}{\partial x}G_{\lambda'}(x,z)|_e$ on $e$ with $z\notin e$ as $\lambda$ goes to $\lambda'$. Thus for any $e$ and a vertex $z$ with $z\notin e$, the map $\lambda\mapsto \frac{\partial}{\partial x}G_\lambda(x,z)|_e$ is continuous with respect to $\|\cdot\|_\infty$-norm. 
The number of pairs $(e,z)\in E\times V$ with $l_M < d(e,z) \leq3l_M$ is finite up to $\G$-action. Thus we have
$$m:=\max \left\{\underset{x\in e}\sup\left|\frac{\partial}{\partial x}G_\lambda(x,z)|_e\right|:  e\in E, z\in V ,l_M\leq d(e,z)\leq 3l_M,\lambda\in[0,\lambda_0]\right\}<\infty.$$ 
Hence $|A_\lambda(x,y)|\leq \frac{m}{c}$ when $d(x,y)=2l_M$.\\
$(2)$ Let $e'$ be an edge containing $y$. Then $l_M\leq d(e,e')\leq 2l_M$. Let $z$ be a vertex on a geodesic containing edges $e$ and $e'$ in order. Suppose that $l_M\leq d(e',z)\leq 2l_M$ and $e'$ is closer to $z$ than $e$. The function $y\mapsto G_\lambda(y,z)|_{e'}$ on $e'$ is bounded below and above.  The function $y\mapsto G_{\lambda}(y,z)|_{e'}$ converges uniformly to $y\mapsto G_{\lambda'}(y,z)|_{e'}$ as $\lambda'$ goes to $\lambda$. Thus we conclude that $x\mapsto A_{\lambda'}(x,y)|_e$ on $e$ converges uniformly to $x\mapsto A_\lambda(x,y)|_e$ as $\lambda'$ goes to $\lambda$. The number of $(e,e',z)\in E\times E\times V$ satisfying $l_M\leq d(e,e')\leq 2l_M$ and $l_M\leq d(e',z)\leq 2l_M$ is finite up to $\Gamma$-action. Thus there exists the number $\delta$ satisfying Lemma \ref{lem:2.9} $(2)$.
\end{proof}
\begin{defn} Let $(X,d)$ be a metric space and let $\alpha\in (0,1].$ A function $f:X\rightarrow\mathbb{R}$ is $\alpha$-\textit{H\"older continuous} if there exists a constant $\epsilon$ such that
$$|f|_{\alpha}:=\sup\left\{\frac{|f(x)-f(y)|}{d(x,y)^\alpha}:x\neq y\text{ and }d(x,y)<\epsilon\right\}<\infty .$$
\end{defn} 
Denote $\|f\|_\alpha:=\|f\|_\infty+|f|_\alpha$. The space $C^\alpha_b(X)$ of bounded $\alpha$-H\"older continuous functions is a Banach space with respect to $\|\cdot\|_\alpha$-norm. 
\begin{prop}\label{prop:2.11} There exists a constant $\alpha\in(0,1)$ such that $f_\lambda$ is $\Gamma$-invariant and  $f_\lambda\in C^\alpha_b( \cG\mathcal{T})$ for all $\lambda\in[0,\lambda_0]$. 
For any $\alpha'\in(0,\alpha),$ the map $\lambda\mapsto f_\lambda$ is a continuous map from $[0,\lambda_0]$ to the space of $\alpha'$-H\"older continuous functions with respect to $\|\cdot\|_{\alpha'}$-norm.
\end{prop}
\begin{proof}The second equality of \eqref{eq:2.9} follows from the equation \eqref{eq:2.8}. Lemma \ref{lem:2.9} shows the last equality of \eqref{eq:2.9}. \begin{equation}\label{eq:2.9}
\begin{split}
{f_\lambda (g)}&=\lim_{\epsilon\ra0+}\lim_{s\ra\infty}\frac{-2k_\lambda(g(0),g(\epsilon),g(s))+2k_\lambda(g(0),g(0),g(s))}{\epsilon}\\&\overset{\eqref{eq:2.8}}=\lim_{\epsilon\ra0+}\lim_{s\ra\infty}\frac{-\mathbb{E}_{g(\epsilon)}[1_{t_{g(2l_M)}<\infty}e^{\lambda t_{g(2l_M)}}]+\mathbb{E}_{g(0)}[1_{t_{g(2l_M)}<\infty}e^{\lambda t_{g(2l_M)}}]}{\epsilon}2k_\lambda(g(0),g(2l_M),g(s))\\
 &=\lim_{\epsilon\ra0+}\frac{-\mathbb{E}_{g(\epsilon)}[1_{t_{g(2l_M)}<\infty}e^{\lambda t_{g(2l_M)}}]+\mathbb{E}_{g(0)}[1_{t_{g(2l_M)}<\infty}e^{\lambda t_{g(2l_M)}}]}{\epsilon}2k_\lambda(g(0),g(2l_M),g_+) \\ &\overset{\substack{\text{Lemma}\\\ref{lem:2.9}}}=-2A_\lambda(g(0),g(2l_M))k_\lambda(g(0),g(2l_M),g_+).
 \end{split}
 \end{equation}
 By Harnack inequality \eqref{eq:2.2}, for any geodesic line $g$ and $\lambda\in[0,\lambda_0]$,
  \begin{eqnarray}\label{eq:2.10}
 k_\lambda(g(0),g(2l_M),g_+)\leq e^{D_{2l_M,l_M}}k_\lambda (g(0),g(0),g_+)= e^{D_{2l_M,l_M}}.
 \end{eqnarray}
  Denote $C_1:=\sup\{\rho^{-2l_M}|f_\lambda(g)|:g\in \cG\mathcal{T},\lambda\in[0,\lambda_0]\}$, where $\rho$ is the constant in Corollary \ref{coro2.7}. 
By Lemma \ref{lem:2.9}, \eqref{eq:2.9} and \eqref{eq:2.10}, $C_1$ is finite. 

Let $\alpha :=\min\{-\frac{1}{2}\log \rho,\frac{1}{2}\}.$ We claim that there exist $\epsilon, C>0$ such that for all $g_1,g_2\in \cG\mathcal{T}$ with $d_{\mathcal{G}\mathcal{T}}(g_1,g_2)<\epsilon$,
$$|f_\lambda(g_1)-f_\lambda(g_2)|\leq Cd_{\mathcal{G}\mathcal{T}}(g_1,g_2)^\alpha.$$
 By Lemma \ref{lem:2.2} and Corollary \ref{coro:2.8}, for any geodesic line $g$ and $s \in [0,l_m/4]$, 
\begin{equation}\label{eq:2.11}
 e^{-8d_M\sqrt{s/l_m}}
\leq{k_\lambda(g(0),g(s),g_+)}\leq e^{8d_M\sqrt{s/l_m}}.
\end{equation}
For any $(x,\xi)\in \mathcal{T}\times \partial \mathcal{T}$, $\frac{\partial^2}{\partial y^2} k_\lambda(x,y,\xi)=\lambda k_\lambda(x,y,\xi)$ (see Corollary \ref{coro:2.8}). Applying the mean value theorem on an edge $e$, for any $y_1,y_2\in e$, there exists $y'\in [y_1,y_2]$ satisfying
\begin{equation}\label{eq:2.12}
\left|\frac{\partial}{\partial z}\middle|_{z=y_1}k_\lambda(x,z,\xi)|_e-\frac{\partial}{\partial z}\middle|_{z=y_2}k_\lambda(x,z,\xi)|_e\right|=|\lambda k_\lambda(x,y',\xi) d(y_1,y_2)|.
\end{equation}
For any geodesic line $g$ and $s\in[0,l_m/4]$, we have 
\begin{equation}\label{eq::2.17}
\begin{split}
& |f_\lambda(\phi_s g)-f_\lambda(g)|\\ 
&=|f_\lambda(\phi_s g)-k_\lambda(g(0),g(s),g_+)f_\lambda(\phi_s g)+k_\lambda(g(0),g(s),g_+)f_\lambda(\phi_s g)-f_\lambda(g)|\\
&\leq|f_\lambda(\phi_s g)-k_\lambda(g(0),g(s),g_+)f_\lambda(\phi_s g)|+|k_\lambda(g(0),g(s),g_+)f_\lambda(\phi_s g)-f_\lambda(g)|\\
&\leq e^{8d_M}\sqrt{s}|f_\lambda(\phi_s g)|+\max\{\lambda k_\lambda(g(0),g(s'),g_+):s'\in[0,l_m/4]\}s.
\end{split}
\end{equation}
Using \eqref{eq:2.11} and the fact that $\max\{|e^{8d_M\sqrt{s/l_m}}-1|,|1-e^{-8d_M\sqrt{s/l_m}}|\}\leq e^{8d_M}\sqrt{s}$ for all $s\in [0,l_m/4]$, we obtain the first term of the last inequality of \eqref{eq::2.17}. The second term of the last inequality of \eqref{eq::2.17} follows from \eqref{eq:2.12}.
  Harnack inequality \eqref{eq:2.2} allows us to choose a constant $C_2$ satisfying 
  $$C_2\geq C_1e^{8d_M}+\max\{\lambda k_\lambda(g(0),g(s_1),g_+):s_1\in[0,l_m/4]\}.$$
  Hence, for any $s\leq \min\{1,l_m/4\},$ we have
  \begin{equation}\label{eq:2.13}|f_\lambda(\phi_s g)-f_\lambda(g)|\leq C_2d(\pi (g),\pi(\phi_s g))^{\frac{1}{2}}.\end{equation}

 Let $g_1$ and $g_2$ be the geodesic lines with $\pi(g_1)=\pi(g_2)$ and $g_1|_{[0,T]}=g_2|_{[0,T]}$ for some $T>2l_M+2$. The equation \eqref{eq:2.9} shows the first and the second line of \eqref{eq:2.14}. By Corollary \ref{coro2.7}, there exist constants $C_3>0$ and $\rho\in(0,1)$ satisfying the inequality of \eqref{eq:2.14}.
\begin{eqnarray}\label{eq:2.14}
 &&|f_\lambda(g_1)-f_\lambda(g_2)|\\
\nonumber&\overset{\eqref{eq:2.9}}=& |2A(g_1(0),g_1(2l_M))| |k_\lambda(g_1(0),g_1(2l_M),g_{1+})-k_\lambda(g_2(0),g_2(2l_M),g_{2+})|\\
\nonumber&\overset{\substack{\text{Coro}\\\ref{coro2.7}}}\leq& C_3\rho^{T-2l_M}|2A_\lambda(g_1(0),g_1(2l_M))k_\lambda(g_1(0),g_1(2l_M),g_{1+})|\\\nonumber&\overset{\eqref{eq:2.9}}=&C_3\rho^{T-2l_M}|f_\lambda(g_1)|\leq C_1C_3d_{g_1(0)}(g_{1+},g_{2+})^{2\alpha}.
\end{eqnarray}
 
Choose $\varepsilon$ satisfying $\varepsilon<\min\{l_m/4,1,e^{-2l_M-2}\}$. Let $C_4$ be a constant in Lemma \ref{lem:2.1}. Then \eqref{eq:2.13} and \eqref{eq:2.14} shows that for any $g_1,g_2 \in \mathcal{G}\mathcal{T}$ with $d_{\cG\mathcal{T}}(g_1,g_2)<\varepsilon$ and $\pi(g_1)=\pi(\phi_s g_2),$
\begin{eqnarray}
\nonumber |f_\lambda(g_1)-f_\lambda(g_2)|&=& |f_\lambda(g_1)-f_\lambda(\phi_s g_2)+ f_\lambda(\phi_s g_2)-f_\lambda(g_2)|\\
 \nonumber&\leq&  |f_\lambda(g_1)-f_\lambda(\phi_s g_2)|+ |f_\lambda(\phi_s g_2)-f_\lambda(g_2)|\\
 \nonumber&\leq& C_1C_3d_{g_1(0)}(g_{1+},g_{2+})^{2\alpha}+C_2d(\pi (g_2),\pi(\phi_s g_2)))^\alpha\\
 \nonumber&\leq&(C_1C_3+C_2)C_4d_{\cG\mathcal{T}}(g_1,g_2)^\alpha.
 \end{eqnarray}
  Lemma \ref{lem:2.1} shows the last inequality above. Hence, $f_\lambda$ is H\"older continuous.
 
To complete the proof, it remains to show that the map $\lambda\mapsto f_\lambda$ is a continuous map from $[0,\lambda_0]$ to $C_b^{\alpha'}(\cG\mathcal{T})$. We claim that for any sufficiently close $x,y\in \mathcal{T}$, the function $\xi \mapsto k_\lambda(x,y,\xi)$ converges uniformly to $\xi \mapsto k_{\lambda'}(x,y,\xi)$ as $\lambda$ goes to $\lambda'$.
By \eqref{eq:2.9}, for any $g\in \mathcal{G}\mathcal{T},$
\begin{equation}\nonumber
\begin{split}
|f_\lambda(g)-f_{\lambda'}(g)|\leq& |A_\lambda(g(0),g(2l_M))-A_{\lambda'}(g(0),g(2l_M))|k_\lambda(g(0),g(2l_M),g_+)\\
&+|A_{\lambda'}(g(0),g(2l_M))\|k_\lambda(g(0),g(2l_M),g_+)-k_{\lambda'}(g(0),g(2l_M),g_+)|
\end{split}
\end{equation}
  If the claim is true, Lemma \ref{lem:2.9} shows that the above inequality shows that $$\lim_{\lambda\rightarrow\lambda'}\|f_\lambda-f_{\lambda'}\|_\infty=0.$$
 Since $|f|_{\alpha'}\leq 2\|f\|_\infty^{1-\alpha'/\alpha}|f|_{\alpha}^{\alpha'/\alpha}$, the map $\lambda\mapsto f_\lambda$ is continuous with respect to $\|\cdot\|_{\alpha'}$-norm.
 
 The remaining part is to prove the claim. By Corollary \ref{coro2.7} and Harnack inequality \eqref{eq:2.2}, for any $\epsilon/3>0$, there exists a constant $R>2$ such that for any $g\in \cG\cT$ with $g(0)=x$, $\lambda\in [0,\lambda_0]$ and $t\geq R$ and for any point $y$ with $d(x,y)\leq 2l_M$,
  \begin{equation}
 \nonumber \left| k_\lambda(x,y,g(R))-k_\lambda(x,y,g(t))\right|\overset{\substack{\text{Coro}\\\ref{coro2.7}}}< \frac{\epsilon k_\lambda(x,y,g(R))}{3e^{D_{l_M,1}}}\overset{\eqref{eq:2.2}}\leq \frac{\epsilon}{3}.
  \end{equation}
  Since the action of $\G$ is cocompact, for any $\lambda$ sufficiently close to $\lambda'$, point $z$ with $d(x,z)=R,$ and points $x$ and $y$ with $d(x,y)\leq 2l_M$,
  $$|k_\lambda(x,y,z)-k_{\lambda'}(x,y,z)|<\epsilon/3.$$
  For any $\lambda$ sufficiently close to $\lambda'$, $g\in \cG_x \mathcal{T}$, we have 
\begin{equation}
\begin{split}
&\nonumber|k_\lambda(x,y,g_+)-k_{\lambda'}(x,y,g_+)|\\
&\leq |k_\lambda(x,y,g_+)-k_{\lambda}(x,y,g(R))|+|k_\lambda(x,y,g(R))-k_{\lambda'}(x,y,g(R))|\\
&+|k_{\lambda'}(x,y,g(R))-k_{\lambda'}(x,y,g_+)|< \epsilon,
\end{split}
\end{equation}
where $g$ is the geodesic ray from $x$ to $\xi.$ The above inequality shows that $\xi\mapsto k_\lambda(x,y,\xi)$ converges uniformly to $\xi \mapsto k_{\lambda'}(x,y,\xi)$ as $\lambda$ goes to $\lambda'$.
  \end{proof}
\section{Gibbs measure related to $\lambda$-Green functions}\label{Sec:3} In the section, we construct Patterson-Sullivan density and Gibbs measure of $f_\lambda$. 
\subsection{Poincar\'e series and critical exponent of $f_\lambda$}\label{subsec:3.1} In the section, we define Poincar\'e series and the critical exponent of $f_\lambda$ which is necessary for constructing Patterson-Sullivan density.  

Let $f$ be a H\"older continuous function on the unit tangent bundle $T^1 X$ of a manifold or a tree $X$. The integral $f$  from $x$ to $y$ of is define by
\begin{equation}\label{eq:3.1}\int_x^y f:=\int_{0}^{d(x,y)} f(v_{\phi_t g}) dt,\end{equation}
 where $g$ is a geodesic segment from $x$ to $y$ and $v_{\phi_t g}$ is the unit tangent vector to $g$ at time $t$. In \cite{Le}, \cite{PPS} and \cite{BPP}, the equation \eqref{eq:3.1} enable us to define the critical exponent of $f$ and Poincar\'e series of $f$. Since the function $f_\lambda$ is defined on $\cG\mathcal{T}$, the equation \eqref{eq:3.1} for $f_\lambda$ is not valid. Hence, $G_\lambda^2(x,y)$ will be used instead of $e^{\int_x^y f_\lambda}$.
\begin{defn}\label{def:3.1} Let $x,y$ be points in $\mathcal{T}$.
\begin{enumerate}
\item The \textit{Poincar\'e series} $P_{x,y,\lambda}(s)$ of $f_\lambda$ is defined by
$$P_{x,y,\lambda}(s):=\sum_{\g\in\G_{x,y}}G_\lambda^2(x,\g y)e^{-sd(x,\g y)},$$
where $\Gamma_{x,y}=\{\g\in\G:x\neq \g y\}$.
\item The \textit{critical exponent} $\delta_{\lambda}$ of $f_\lambda$ is defined by
$$\delta_{\lambda}:= \limsup_{n\ra +\infty}\frac{1}{n} \log \sum_{\substack{\g\in \G_{x,y}\\n-1\leq d(x,\g y)<n}} G_\lambda^2 (x,\gamma y).$$ 
\end{enumerate}
\end{defn}
The \textit{critical exponent} $\delta_\G$ of the discrete group $\G$ is defined by
$$\delta_{\G}:= \limsup_{n\ra +\infty}\frac{1}{n} \log N(\G,n),$$
where $N(\G,n)$ is the cardinality of the set $\{\g\in\G:n-1\leq d(x,\g x)<n\}$. As in the proof of Proposition 2 in \cite{LW}, the critical exponent $\delta_\G$ coincides with the volume entropy of $\mathcal{T}$. Using the formula for the volume entropy of trees in \cite{Lim}, one obtains that the critical exponent $\delta_{\G}$ is finite.

By \eqref{eq:2.2}, the critical exponent $\delta_{\lambda}$ is independent of the choice of $x$ and $y$. Consider the finiteness of the critical exponent $\delta_\lambda$. By Proposition 4.10 in \cite{HL}, $\delta_{\lambda}\leq 0$. Since $\G$ acts cocompactly, for all $x,y\in \mathcal{T}$ with $1\leq d(x,y)\leq 2$, $\frac{1}{a}\leq G_\lambda(x,y)$ for some $a>0$. Let $\{x_n\}$ be the sequence on the geodesic segment $[x,y]$ satisfying $d(x,x_n)=n$. By Ancona inequality, for any $x,y$ with $n-1<d(x,y)\leq n$, we have 
$$G_\lambda(x,y)\geq \frac{1}{C^{n-1}}G_\lambda(x,x_1)G_\lambda(x_1,x_2)\cdots G_\lambda(x_{n-2},y)\geq \frac{1}{(aC)^{n-2}}$$
and
\begin{equation}\label{eq:3.2}
\sum_{\substack{\g\in \G\\(n-1)\leq d(x,\g y)<n}} G_\lambda^2 (x,\gamma y)\geq \frac{N(\G,n)}{(aC)^{2n-4}}.
\end{equation}
The equation \eqref{eq:3.2} implies that $\delta_\lambda$ is bounded below by $\delta_\G-2\log aC$. Thus the critical exponent $\delta_\lambda$ is finite for all $\lambda\in [0,\lambda_0]$. Poincar\'e series $P_{x,y,\lambda}(s)$ converges when $s>\delta_\lambda$, and diverges when $s<\delta_\lambda$. Poincar\'e series $P_{x,y,\lambda}(s)$ is of \textit{divergence type  (convergence type, resp.)} if the $P_{x,y,\lambda}(\delta_\lambda)$ diverges (converges, resp.).
\subsection{Patterson-Sullivan density of $f_\lambda$}\label{subsec:2.3} Applying the method in \cite{LW} and \cite{Le}, we construct Patterson-Sullivan density of $f_\lambda$.
\begin{defn} A \textit{Patterson-Sullivan} density of dimension $s$ for $f_\lambda$ is a family $\{\mu_
{x,\lambda}\}_{x\in \mathcal{T}}$ of measures on the geometric boundary $\partial \mathcal{T}$ such that for all $x,y\in \mathcal{T}$, $\g\in \G$, and $\xi \in \partial \mathcal{T},$
\begin{equation}\label{eq:3.3}
\g_*\mu_{x,\lambda}=\mu_{\g x,\lambda}
\end{equation}
and
\begin{equation}\label{eq:3.4}
\frac{d\mu_{y,\lambda}}{d\mu_{x,\lambda}}(\xi)=k_\lambda^2(x,y,\xi)e^{s\beta_{\xi}(x,y)}.
\end{equation}
\end{defn}
 
 \begin{prop}\label{prop:3.3}For any $\lambda\in [0,\lambda_0]$, there exists a Patterson-Sullivan density of dimension $\delta_\lambda$ for $f_\lambda$. The support of Patterson-Sullivan density is $\partial \mathcal{T}$. \end{prop}
\begin{proof}There exists a non-decreasing function $h:\R_+\ra\R_+\cup\{+\infty\}$ satisfying the following.
 \begin{itemize}
 \item For any $\eps>0$, there exists $r_\eps$ satisfying $h(t+r)\leq e^{\eps t}h(r)$ for all $t\geq 0$ and $r\geq r_\eps$ 
 \item and $$P_{x,y,\lambda,h}(\delta_\lambda):=\sum_{{\g \in\G_{x,y}}}G_\lambda^2(x,\g y)h(d(x,\g y))e^{-sd(x,\g y)}$$ diverges when $s\leq\delta_\lambda$ (see \cite{Q} Lemma 4.9).
\end{itemize}
  If Poincar\'e series $P_{x,y,\lambda}(s)$ is of divergence type, then the constant function $h\equiv1$ satisfies the above conditions.

 For any $s>\delta_\lambda$, define a family $\{\mu_{s,x,\lambda}\}_{x\in\mathcal{T}}$ of measures  on $\overline{\mathcal{T}}$ as follows:
\begin{displaymath}
\mu_{s,x,\lambda}:=\displaystyle\frac{\sum_{\g \in \G_{x,y}}G_\lambda^2(x,\g y)h(d(x,\g y))e^{-sd(x,\g y)}\delta_{\g y}}{\sum_{\g \in \G_{y,y}}G_\lambda^2(y,\g y)h(d(y,\g y))e^{-sd(y,\g y)}},
\end{displaymath}
where $\delta_{\g y}$ is the Dirac delta mass at $\g y$. Let $\{s_n\}$ be a sequence converging to $\delta_\lambda$. Since $\overline{\mathcal{T}}$ is compact, there exists a subsequence $\{\mu_{s_{n_k},y,\lambda}\}$ of probability measures converging weakly to a measure probability $\mu_{y,\lambda}$ on $\overline{\mathcal{T}}$. Since $P_{y,y,\lambda,h}(\delta_\lambda)$ diverges, the support of $\mu_{y,\lambda}$ is in $\Lambda_\G$.
  
 Let $f$ be a continuous function satisfying $f\equiv 0$ on $B(x,2d(x,y))$. For all $s_{n_k}$, we have
\begin{eqnarray}\label{eq:3.5}
&&\sum_{\g \in \G_{x,y}}G_\lambda^2(x,\g y)h(d(x,\g y))e^{-s_{n_k}d(x,\g y)}f(\g y)\\
&=&\sum_{\g \in \G_{x,y}}\frac{G_\lambda^2(x,\g y)}{G_\lambda^2(y,\g y)}G_\lambda^2(y,\g y)\frac{h(d(x,\g y))}{h(d(y,\g y))}h(d(y,\g y))\frac{e^{-s_{n_k}d(x,\g y)}}{e^{-s_{n_k}d(y,\g y)}}e^{-s_{n_k}d(y,\g y)}f(\g y)\nonumber
\end{eqnarray}
For any $\g\in \Gamma_{y,y}$, $f(\g y)=f(y)=0$. Thus by \eqref{eq:3.5}, 
\begin{equation}\label{eq:3.6}
\mu_{s_k,x,\lambda}(f)=\int_{\overline{\mathcal{T}}}k_\lambda^2(y,x,z)\frac{h(d(x,z))}{h(d(y,z))}\frac{e^{-s_{n_k}d(x,z)}}{e^{-s_{n_k}d(y,z)}}f(z)d\mu_{s_{n_k},y,\lambda}(z).
\end{equation}
For any points $z$ with $d(y,z)\geq r_{s_{n_k}-\delta_\lambda}$, 
$$h(d(x,z))\leq h(d(x,y)+d(y,z))\leq e^{(s_{n_k}-\delta_\lambda)d(x,y)}h(d(y,z)).$$
As $k$ goes to infinity, the left hand side of \eqref{eq:3.6} converges to a weak limit $\mu_{x,\lambda}$ satisfying $$d\mu_{x,\lambda}(\xi)=k_\lambda^2(y,x,\xi)e^{\delta_\lambda\beta_\xi(y,x)}d\mu_{y,\lambda}(\xi)$$
for any point $\xi$ in the support of $\mu_{y,\lambda}$. The equation \eqref{eq:3.4} follows from the above equality.
Fix $\g\in\G$. Choose a sufficiently large number $R$ such that $\g x\in B(x,2d(x,y)+R)$. For any continuous function $f$ satisfying $f\equiv 0$ on $B(x,2d(x,y)+R)$, we obtain
\begin{equation}\label{eq:3.7}
\begin{split}
P_{y,y,\lambda,h}(s_{n_k})\mu_{s_{n_k},\g x,\lambda}(f)&=\sum_{\g' \in\G_{\gamma x,y}}G_\lambda^2(\g x,\g' y)h(d(\g x, \g' y))e^{-s_{n_k}d(\g x,\g' y)}f(\g' y)\\
&=\sum_{\g\g' \in \G_{\gamma x,y}}G_\lambda^2(\g x,\g\g' y)h(d(\g x, \g\g' y))e^{-s_{n_k}d(\g x,\g\g' y)}f(\g\g' y)\\
&=\sum_{\g' \in \G_{x,y}}G_\lambda^2(x,\g' y)h(d(x, \g' y))e^{-s_{n_k}d(x, \g' y)}f(\g\g'y)\\
&=P_{y,y,\lambda,h}(s_{n_k})\int_{\overline{\mathcal{T}}}f(\g z)\mu_{s_{n_k},x,\lambda}(z).
\end{split}
\end{equation}
The equation \eqref{eq:3.7} implies that $\mu_{s_{n_k},\g x,\lambda}(f)=\g_*\mu_{s_{n_k},x,\lambda}(f)$ and $\mu_{\g x,\lambda}(f)=\g_*\mu_{x,\lambda}(f)$.
 By \eqref{eq:3.3}, the support of $\mu_{x,\lambda}$ is $\G$-invariant. Since the smallest nonempty $\G$-invariant subset of $\partial \mathcal{T}$ is $\Lambda_\G$ (\cite{Q}, Proposition 4.7), the support of $\mu_{x,\lambda}$ is  $\Lambda_\G$. Since the $\G$-action is cocompact, $\Lambda_\G=\partial \mathcal{T}$.
\end{proof}

 The cocompact action of $\G$ and Harnack inequality imply $\int_{T_0}\mu_{x,\lambda}(\partial \mathcal{T})d\mu(x)<\infty$. After this, we assume that for any $\lambda\in [0,\lambda_0]$, $\int_{T_0}\mu_{y,\lambda}(\partial\mathcal{T})d\mu(y)=1$. 
Similar to Corollary 3.1 in \cite{LL}, Proposition \ref{prop:2.11} shows the following lemma.
\begin{lem}\label{2203311} For any $x\in \mathcal{T}$, the map $\lambda\mapsto \mu_{x,\lambda}$ is continuous.
\end{lem}
The following lemma is analogous to Lemma 3.10 in \cite{PPS} and Lemma 4.11 in \cite{BPP} called Mohsen's shadow lemma (\cite{M}).
\begin{lem}\label{lem:3.4}There exists a constant $C>1$ such that 
 for any $\lambda\in [0,\lambda_0]$, $x,y\in \overline{T_0}$ and $\g\in \G$ with $d(x,\g y)>1$, \begin{equation}\label{eq:3.8}
\frac{1}{C}e^{-\delta_\lambda d(x,\g y)}G_\lambda^2 (x,\g y)\leq \mu_{x,\lambda}(\cO_x(\g y))\leq Ce^{-\delta_\lambda d(x,\g y)}G_\lambda^2 (x,\g y).
\end{equation}
\end{lem}
\begin{proof} 
As in the proof of Lemma 4.13 in \cite{BPP}, using \eqref{eq:3.3}, we obtain a constant $C$ such that for any $\lambda\in [0,\lambda_0]$, $x,y\in \overline{T_0}$ with $d(x,\gamma y)>1$,
\begin{equation}\label{eq:3.9}
\frac{1}{C}\leq\mu_{\g y,\lambda}(\cO_x(\g y))\leq C.
\end{equation}
 
By \eqref{eq:3.4}, we have
\begin{equation}\label{eq:3.10}
\mu_{x, \lambda}(\cO_x(\g y))=\int_{\cO_x(\g y)} k_\lambda^2 (\g y, x,\xi)e^{\delta_\lambda\beta_\xi(\g y,x)}d\mu_{\g y,\lambda}(\xi).
\end{equation}
For all $\xi \in \cO_x(\g y)$, the geodesic ray $g$ from $x$ to $\xi$ contains the geodesic segment $[x,\g y]$. By Ancona inequality \eqref{eq:2.4}, for all $t>d(x,\g y)+1$ and geodesic line $g$ with $g_+\in \mathcal{O}_x(\gamma y),$
\begin{equation}\label{eq:3.11}
\frac{1}{C'}G_\lambda(x,\g y)\leq\frac{G_\lambda(x,g(t))}{G_\lambda(\g y ,g(t))}\leq\frac{C'G_\lambda(x,\g y)G_\lambda(\g y,g(t))}{G_\lambda(\g y, g(t))}={C'}{G_\lambda(x,\g y)}.
\end{equation}
The inequality \eqref{eq:3.11} shows $G_\lambda(x,\g y)/C'\leq k_\lambda(\g y,x,\xi)\leq C'G_\lambda(x,\g y)$. Since $\beta_\xi(\g y,x)$ is equal to $-d(x,\g y)$ for all $\xi \in \mathcal{O}_x(\g y)$, 
\begin{equation}\label{eq:3.12}
\begin{split}
\frac{1}{C'}e^{-\delta_\lambda d(x,\g y)}\mu_{\g y}(\mathcal{O}_x(\g y))G_\lambda^2(x,&\g y)\\
&\leq\mu_{x, \lambda}(\cO_x(\g y))\leq C'e^{-\delta_\lambda d(x,\g y)}\mu_{\g y}(\mathcal{O}_x(\g y))G_\lambda^2 (x,\g y).
\end{split}
\end{equation}
By \eqref{eq:3.9} and \eqref{eq:3.12}, we have the inequality \eqref{eq:3.8}.
\end{proof}

\subsection{Gibbs measure of $\fl$}\label{subsec:2.4} In this section, we construct Gibbs measure of $f_\lambda$.
\begin{prop} \emph{(\cite{K}, Section 4, p.390, Remark 2)}
If Ancona inequality \eqref{eq:2.5} holds, for all $x\in \mathcal{T}$ and $\xi,\zeta\in \partial\mathcal{T}$, the limit
\begin{equation}\label{3.13}
\theta_x^\lambda(\xi,\zeta):=\lim_{y\ra\xi,z\ra\zeta}\frac{G_\lambda(y,z)}{G_\lambda(y,x)G_\lambda(x,z)}
\end{equation}
exists.
\end{prop}
The function $\theta_x^\lambda:\partial^2\mathcal{T}\ra\bR$ is called \textit{N\"aim kernel}. Similar to the proof of Proposition \ref{prop:2.11}, applying \eqref{eq:2.5}, we have the following corollary.
\begin{coro}\label{coro:3.6} For any $\epsilon>0$, there exists a constant $\delta$ such that for any $\lambda\in[\lambda_0-\delta,\lambda_0]$, $x\in \mathcal{T}$ and $\xi,\zeta\in \partial \mathcal{T}$ with $(\xi|\zeta)_x\leq Diam(T_0)$,
$$|\theta_x^\lambda(\xi,\zeta)-\theta_x^{\lambda_0}(\xi,\zeta)|<\epsilon.$$
\end{coro}
Using Hopf parametrization and {N\"aim} kernel, we obtain a $\phi_t$-invariant measure as follows.
\begin{defn}
The \textit{Gibbs measure} of $f_\lambda$ on $\cG\mathcal{T}$ is defined by
$$m_\lambda(g):=\theta_x^\lambda(g_-,g_+)^2e^{2\delta_\lambda(g_-|g_+)_x}\mu_{x,\lambda}(g_-)\mu_{x,\lambda}(g_+)dt$$
for all $g\in \cG\mathcal{T}$, where $(g_-,g_+,t)$ is Hopf parametrization of $g$ based on $x$ and $dt$ is the Lebesgue measure on $\bR.$
\end{defn}
By \eqref{eq:2.1} and \eqref{eq:3.4}, the measure $m_\lambda$ does not depend on the choice of $x$. By \eqref{eq:3.3}, $m_\lambda$ is $\G$-invariant. Since $dt$ is invariant under the translations $t\mapsto t+s$, the measure $m_\lambda$ is $\phi_t$-invariant. Choose a fundamental domain $(\cG\mathcal{T})_0$ of $\G$ in $\cG\mathcal{T}$ contained in the set of geodesic lines $g$ such that $\pi (g)\in \overline{T_0}$. The compactness of $\G$-action guarantees that $m_\lambda((\cG\mathcal{T})_0)$ is finite. Thus there exists a probability measure $\overline{m}_\lambda$ on $\G\backslash \cG\mathcal{T}$ induced by $m_\lambda$. Similar to Corollay 3.11 in \cite{LL}, by Proposition \ref{prop:2.11} and \ref{2203311}, we have the following lemma.
\begin{lem}\label{2203312} The map $\lambda\mapsto m_\lambda((\cG\mathcal{T})_0)$ is continuous on $[0,\lambda_0]$.
\end{lem}
\section{Variactional principle of Gibbs measure of $f_\lambda$}\label{Sec:4}
One way to prove the variational principle (see Theorem \ref{thm:1.4}) is to show weak Gibbs property (see Definition \ref{def:4.4}) for a potential function $F_\lambda$ on the cross section $\G\backslash Y$ of $\G\backslash\cG\mathcal{T}$ as in \cite{BPP}, which will be defined shortly in Section \ref{subsec:3.4}. Let $\tau$ and $T$ be the first return time and the first return map of $\G\backslash Y$, respectively (see Definition \ref{def:4.1}). In Section \ref{sec:3.5}, we introduce a homeomorphism ${\Theta}$ from $(\G\backslash Y,T)$ to a transitive Markov shift $(\Sigma,\sigma)$ satisfying ${\Theta}\circ T=\sigma\circ{\Theta}$ and we prove the variational principle for the potential $F_\lambda\circ \Theta^{-1}$ using weak Gibbs property and the homeomorphism $\Theta$. In Section \ref{sec:3.6}, using the Abramov's formula for the suspension $(\Sigma,\sigma)_{\tau\circ\Theta^{-1}}$, we show the variational principle for $f_\lambda$. 
 
\subsection{Cross section of $\cG\mathcal{T}$}\label{subsec:3.4} In this section, we introduce the cross section $Y$ of $\cG\mathcal{T}$ and show the weak Gibbs property of a measure $\overline{\nu}_\lambda$ related to $F_\lambda$ (see \eqref{eq::4.2}) on $\G\backslash Y$. We refer to Chapter 5.4 of \cite{BPP} for details. 

\begin{defn}\label{def:4.1}
A closed subset $C$ of a metric space $X$ with a flow $\{\psi_t\}_{t\in \mathbb{R}}$ is the \textit{cross section} to the flow $\{\psi_t\}$ on $X$ if for any $x\in X$, $\psi_t x\in C$ for some positive number $t$ and if there exist a continuous function $\tau:C\ra \R$, which is called  the \textit{first return time} of $C$, such that $\psi_{\tau(x)}x\in C$ and $\psi_t x\notin C$ for all $x\in C$ and $t\in(0,\tau(x))$. The map $T$ defined by $Tx:=\psi_{\tau(x)}x$ for any $x\in C$ is called the \textit{first return map} of the cross section $C$.
\end{defn}

Let $Y$ be the set of geodesic lines $g$ satisfying $\pi(g)\in V$. For any $g\in \cG\mathcal{T}$, there exists a constant $t\in [0,l_M]$ such that $\phi_t  g \in Y$. Define the first return time $\tau:Y\ra\mathbb{R}$ by $\tau(g)=\inf\{t>0: \phi_t g\in Y\}$ for all $g\in Y$. Note that $\tau(g)$ is the length of the edge containing $[g(0),g(l_m)]$. Obviously, $\tau(g)\in[l_m,l_M]$ for all $g\in Y$. The set $Y$ is the cross section of $\cG\mathcal{T}$. The distance $d_Y$ on $Y$ defined by
\begin{equation}\label{eq:3.14}d_Y(g,g'):=e^{-\min\{|i|:\pi(T^ig)\neq \pi(T^ig'),i\in \mathbb{Z}\}}.\end{equation}
for all $g,g'\in Y$. We also define a distance on $\G\backslash Y$ by
$$d_{\G\backslash Y}(\G g, \G g')=\min\{d_Y(g_1,g_2):g_1\in \G g, g_2\in \G g'\}.$$
The map $\tau$ is locally constant and thus is $\tau$ is H\"older continuous.

 Let $T$ be the first return map of $Y$. The $\G$-invariant potential $F_\lambda$ on $Y$ is defined by 
 \begin{equation}\label{eq::4.2}
 F_\lambda(g):=\int_0^{\tau(g)} f_\lambda(\phi_t g)-\delta_\lambda dt=2\log k_\lambda(\pi(Tg),\pi(g),g_+)-\delta_\lambda d(\pi(g),\pi(Tg)).
 \end{equation}
By Corollary \ref{coro:2.7}, the function $F_\lambda$ is also H\"older continuous (see \cite{GL} Section 6.1 and \cite{G} Lemma 3.1). The function $F_\lambda$ is considered as a function on $\Gamma\backslash Y.$

Fix a vertex $x$. The modified Hopf parametrization of $Y$ is defined by $g\mapsto(g_-,g_+,n)$, where $\pi (T^ng)$ is the closest point to a base point $x\in V$. As in \cite{BPP}, a measure $\nu_\lambda$ is defined for any $g\in Y$, by 
$$\nu_{\lambda}(g):=\theta_{x}^\lambda(g_-,g_+)^2\mu_{x,\lambda}(g_-)\mu_{x,\lambda}(g_+)dn,$$
where $dn$ is the counting measure on $\mathbb{Z}$. The measure $\nu_{\lambda}$ doesn't depend on the choice of $x$. The measure $\nu_{\lambda}$ is also $\G$-invariant and $T$-invariant. 

For any two distinct integers $p<q$, the cylinder set $C(g,p,q)$ of a geodesic line $g\in Y$ is the set of geodesic lines $g'$ satisfying $\pi (T^ig)=\pi (T^ig')$ for all integers $i \in [p,q]$.  Then the measure $\nu_{\lambda}$ satisfies the following lemma. 

\begin{lem}\label{Gibbs} There exists a constant  $C$ such that for any $\lambda\in [0,\lambda_0]$ and $g'\in C(g,p,q),$
\begin{equation}\label{eq:3.15}\frac{1}{C}\leq \frac{\nu_{\lambda}(C(g,p,q))}{e^{S_{p}^{q}F_\lambda(g')}}\leq C,\end{equation}
where $S_{p}^{q}F_\lambda(g'):=\displaystyle\sum_{n=p}^{q} F_\lambda(T^n g').$
\end{lem}
\begin{proof}
 By Ancona inequality \eqref{eq:2.4}, there exists a constant $C_1$ such that for any $\lambda\in [0,\lambda_0]$, $g\in \cG\mathcal{T}$ and point $x$ on the geodesic line $g$, 
\begin{equation}\label{eq:3.16}
\frac{1}{C_1}\leq \theta_{x}^\lambda(g_-,g_+)\leq C_1.
\end{equation}
\begin{figure}[h]
\begin{center}
\begin{tikzpicture}[scale=1]
  \draw [thick](-3,0) -- (3,0);   \draw [thick](-5,-0.5) -- (-3,0);\draw[thick](-5,0.5) -- (-3,0);\draw[thick] (5,-0.5) -- (3,0);\draw [thick](5,0.5) -- (3,0);\draw plot[smooth,densely dotted] coordinates{(-5,0.5) (-5.2,0) (-5,-0.5) };\draw plot[smooth,densely dotted] coordinates{(5,0.5) (5.2,0) (5,-0.5) };
  \node at (-3,-0.3) {$x_p$};\node at (-2.1,-0.3) {$x_{p+1}$};\node at (-1.3,-0.3) {$x_{p+2}$};\node at (2.2,-0.3) {$x_{q-1}$};\node at (3,-0.3) {$x_q$}; \fill (-3,-0)   circle (2pt); \fill (-2.2,0)    circle (2pt); \fill(-1.4,-0)  circle (2pt);\fill (3,-0)    circle (2pt); \fill (2.2,-0)    circle (2pt); \node at (-6.1,0) {\footnotesize{$\mathcal{O}_{x_q}(x_{p})$}};\node at (6.1,0) {\footnotesize{$\mathcal{O}_{x_p}(x_{q})$}};
\end{tikzpicture}
\end{center}
\caption{$C(g,p,q)$}\label{figure3}
\end{figure}

Denote $x_i:=\pi (T^ig)$. Since $$C(g,p,q)=\{g'\in Y:\pi (T^pg')=x_p,g'_-\in \cO_{x_q}(x_p),g'_+\in \cO_{x_p}(x_q)\}$$ (see Figure \ref{figure3}), by definition of $\nu_\lambda$, $$\nu_{\lambda}(C(g,p,q))=\displaystyle\int_{\mathcal{O}_{x_p}(x_q)}\int_{\mathcal{O}_{x_q}(x_p)}\theta_{x_p}^\lambda(g'_-,g'_+)^2\mu_{x_p,\lambda}(g'_-)\mu_{x_p,\lambda}(g'_+).$$ 
and this shows that 
\begin{equation}\label{eq:3.17}
\begin{split}
\frac{1}{C_1}\mu_{x_p,\lambda}(\cO_{x_q}(x_p))\mu_{x_p,\lambda}(&\cO_{x_p}(x_q))\\
&\leq \nu_{\lambda}(C(g,p,q)) \leq {C_1}\mu_{x_p,\lambda}(\cO_{x_q}(x_p))\mu_{x_p,\lambda}(\cO_{x_p}(x_q)),
\end{split}
\end{equation}
where $C_1$ is the constant in the inequality \eqref{eq:3.16}. Lemma \ref{lem:3.4} and the inequalities \eqref{eq:3.9}
and \eqref{eq:3.17} show that there exists a contant $C_2$ such that
\begin{equation}\label{eq:3.18}
\frac{1}{C_2}e^{-\delta_\lambda d(x_p,x_q)}{G_\lambda^2 (x_p,x_q)}\leq \nu_{\lambda}(C(g,p,q))\leq {C_2e^{-\delta_\lambda d(x_p,x_q)}}{G_\lambda^2 (x_p,x_q)}
\end{equation}
By Ancona inequality, there exists a constant $C_3$ such that for any $g'\in C(g,p,q)$,
\begin{equation}\label{eq:3.19}
\frac{1}{C_3}k_\lambda(x_q,x_p,g'_+)\leq G_\lambda (x_p,x_q)\leq C_3k_\lambda(x_q,x_p,g'_+).
\end{equation}
Since $e^{S_{p}^{q}F_\lambda(g')}=k_\lambda^2(x_q,x_p,g'_+)e^{-\delta_\lambda d(x_p,x_q)}$ for any $g' \in C(g,p,q)$, the inequality \eqref{eq:3.15} follows from the inequalites \eqref{eq:3.18} and \eqref{eq:3.19} when $C=C_2C_3^2$.
\end{proof}

The measure $\overline{\nu}_{\lambda}$ on $\G\backslash Y$ induced by $\nu_{\lambda}$ satisfies for any two integers $p<q$ and any geodesic line $g$,
\begin{equation}\label{eq:3.20}
\overline{\nu}_{\lambda}(\G C(g,p,q)):=\frac{\nu_\lambda(C(g,p,q))}{|\G_{[x_p,x_q]}|},
\end{equation}
where $x_i=(\pi T^ig)$ and $|\G_{[x_p,x_q]}|$ is the cardinality of the stabilizer of the geodesic segment $[x_p,x_q]$. 

Consider the supension $Y_\tau:=(\G\backslash Y)\times \bR/(\G g,s+\tau (g))\sim (\G Tg,s)$ of $Y$.  The suspension flow $T_t$ on $Y_\tau$ is defined by $T_t[\G g,s]=[\G g,s+t]$ for all $[\G g,s]\in Y_\tau$. For any $\G g\in \G\backslash Y$, the vertical distance between two points $[\G g,s]$ and $[\G g,t]$ is defined by
$$d_V([\G g,s],[\G g,t])=|s-t|.$$
The horizontal distance between two points $[\G g,t]$ and $[\G g',t]$ is defined by
$$d_H([\G g,t],[\G g',t])=d_{\G\backslash Y}(\G g,\G g').$$ Bowen-Walters distance $d_{BW}$ between $[\G g,t]$ and $[\G g',s]$ is defined by the smallest path length of a chain of the vertical paths and horizontal paths from $[\G g,t]$ to $[\G g',s]$. Then the map $\Phi$ defined by $\Phi[\G g,t]=\phi_t\G g$ is a bilipschitz homeomorphism from $Y_\tau$ to $\G\backslash \cG\mathcal{T}$ satisfying $\Phi\circ T_t =\phi_t\circ\Phi$ (\cite{BPP}, Theorem 5.8).

The measure $\overline{\nu}_{\lambda}$ also satisfies 
\begin{equation}\label{eq:3.21}
d(\Phi^{-1})_*\overline{m}_{\lambda}([\G g,s])=\frac{1}{\overline{\nu}_\lambda(\Gamma\backslash Y)\int_{\G\backslash Y} \tau d\overline{\nu}_{\lambda}}dsd\overline{\nu}_{\lambda}(\G g),
\end{equation}
where $ds$ the Lebesgue measure on $\bR$ (\cite{BPP}, Theorem 5.8).
\subsection{Coding of $\G\backslash \cG\mathcal{T}$}\label{sec:3.5} In this section, we introduce a transitive Markov shift $(\Sigma,\sigma)$  homeomorphic to $(\G\backslash Y,T)$ (see \cite{BPP}). 

Denote by $L_{\G,\mathbb{Z}}$ the subgroup of $\mathbb{Z}$ generated by $\{n\in \mathbb{Z}:\G T^n g = \G g \text{ for some } g\in Y\}$. If the quotient space $\G\backslash \mathcal{T}$ is compact, then $L_{\G,\mathbb{Z}}$ is $\mathbb{Z}$ or $2\mathbb{Z}$ (\cite{BPP}, Theorem 4.17).  Since the proof of the variational principle when $L_{\G,\mathbb{Z}}=2\mathbb{Z}$ is similar to the proof when $L_{\G,\mathbb{Z}}=\mathbb{Z}$, in this article, we deal with the case when $L_{\G,\mathbb{Z}}=\mathbb{Z}$ (see \cite{BPP}). \begin{defn}
 Let $\cA=\{a_i\}$ be a countable discrete alphabet and $(A_{i,j})_{i,j\in\cA}\in\{0,1\}$ be a transition matrix. 
 \begin{enumerate}
 \item A \textit{(topological) Markov shift} $(\Sigma,\sigma)$ of $(\cA,A)$ consists of a subspace $\Sigma$ of $\cA^{\bZ}$ defined by
 $\Sigma:=\{x=(x_n)_{n\in\bZ}: A_{x_n,x_{n+1}}=1\}$
 and a shift map $\sigma$ defined by $\sigma(x)_{n}:=x_{n+1}.$ 
 \item A \textit{word} $w=a_1\cdots a_n$ of the length $n$ is a finite sequence of alphabets. A word $w$ is \textit{admissible} if $A_{a_i,a_{i+1}}=1$ for all $i$. For all $p<q\in \bZ,$ the \textit{cylinder set} $[w,p,q]$ of a word $w=a_1\cdots a_{q-p+1}$ is the subset of $\Sigma$ defined by
 $$[w,p,q]:=\{x\in\Sigma: x_{p+i-1}=a_i, \forall i\in [1,q-p+1]\cap \bZ\}.$$
 \item A Markov shift $(\Sigma,\sigma)$ is \textit{transitive} if for $x,y\in \cA$, there exists an admissible word $w=a_1\cdots a_n$ with $x=a_1$ and $y=a_n$. 
  \end{enumerate}
\end{defn} 
The distance $d_\Sigma$ between $x$ and $x'$ in $\Sigma$ is  
\begin{equation}\label{eq:3.22}
d_\Sigma(x,x'):=e^{-\min\{|n|: x_n\neq x_n'\}}.
\end{equation}
Let $\text{Per}_n$ be the set of elements $x$ in $\Sigma$ satisfying $\sigma^n x= x$. A continuous function $f$ on $\Sigma$ is called \textit{potential}. Denote $S_{n_1}^{n_2}f(x):=\Sigma_{i=n_1}^{n_2}f(\sigma^{i}x)$.
\begin{defn}\label{def:4.4} A 
$\sigma$-invariant Borel probability measure $m$ on $\Sigma$ is called \textit{weak Gibbs measure} of $f$ if there exists a constant $p(m)$ satisfying the following property: For any compact set $K$, there exists a constant $C_K$ satisfying for all $x\in \text{Per}_n \cap K$ and $n\geq 1,$ 
\begin{equation}\label{eq:3.23}
\frac{1}{C_K}\leq \frac{m(\{x: x_i=y_i ,\,\forall i\in [0,n-1]\cap\bZ\})}{e^{S_0^nf(x)-p(m)n}}\leq C_K.
\end{equation} 
\end{defn}
 Denote $f^-:=\max\{-f,0\}$. The definition of the pressure for $f$ and the equilibrium state for $f$ is the following.
 \begin{defn}\label{def:4.5} Let $X$ be a locally compact space. Let $T$ ($\phi_t$, resp.) be a homeomorphism (an one-parameter group of homeomorphisms, resp.) on $X$. Let $f$ be a continuous function. Denote by $\cM_f$ the space of $T$ ($\phi_t$, resp.)-invariant probability measures $m$ in $X$ with $\int_X f^-dm<+\infty$. The \textit{pressure} $P_f$ for $f$ is 
$$P_f:=\sup\left\{h_m+\int_{X}f dm: m\in \cM_f \right\},$$
where $h_m$ is the entropy of $T$ ($\phi_1$, resp.) with respect to $m$.  The \textit{equilibrium state} of $f$ is a measure $m$ in $\cM_f$ satisfying $P_f=h_m+\int_X fdm.$ 
 \end{defn} 
 
Denote $\text{var}_n(f):=\sup\{|f(x)-f(y)|:x_k=y_k,\, \forall |k|\leq n\}.$ Buzzi's theorem in Appendix A of \cite{BPP} is as follows.
\begin{thm}\label{buzzi}\emph{(\cite{BPP} Corollary A.5)}  Suppose that a potential $f$ on a transitive Markov shift $(\Sigma,\sigma)$ satisfies $\underset{n=1}{\overset{\infty}\sum} n \emph{var}_n f<+\infty$. A measure $m\in \cM_f$ is a weak Gibbs measure of $f$ if and only if $m$ is the equilibrium state of $f$ and the pressure coincides with $p(m)$ in \eqref{eq:3.23}.
\end{thm}

 To recall the coding of $\G\backslash \mathcal{T}$ in \cite{BPP}, we first define the alphabet of the coding. Denote by $\overline{e}$ the opposite direction of an oriented edge $e$  in $\Gamma \backslash \mathcal{T}$ or $\mathcal{T}$. Choose a lift $\widetilde{e}$ of an oriented edge $e$ in $\G\backslash \mathcal{T}$ with $\overline{\widetilde{e}}=\widetilde{\overline{e}}$. Denote  by  $\G_e$ the stabilizer of the edge $\widetilde{e}$ in $\G$. Let $\widetilde{v}$ be a lift of a vertex $v\in e$ in $\mathcal{T}$. Denote by $\G_v$  the stabilizer of $\widetilde{v}$ in $\G$. Choose an element $\g_{e,v}$  in $\G$ satisfying $\g_{e,v} \widetilde{v}\in\widetilde{e}$.  Then an injective homomorphism $\rho_{e,v}$ from $\G_e$ to $\G_v$ is defined for all $\g \in \G_e$, by $\rho_{e,v}(\g)=\g_{e,v}^{-1} \g \g_{e,v}$. 
 
An element $(e^-,[h],e^+)$ of an alphabet $\mathcal{A}$ is constructed as follows: 
\begin{itemize}
\item the oriented edges $e^-$ and $e^+$ in $\G\backslash\mathcal{T}$ intersects at $t(e^-)=i(e^+)$.
\item $[h] \in \rho_{e^-,t(e^-)}(\G_{e^-})\backslash\G_{t(e^-)}/\rho_{e^+,i({e^+})} (\G_{{e^+}})$.
\item $[h]$ is a nontrivial element of $\rho_{e^-,t(e^-)}(\G_{e^-})\backslash\G_{t(e^-)}/\rho_{e^+,i({e^+})} (\G_{{e^+}})$ if $e^- = \overline{e^+}$.
\end{itemize}
Since the $\G$-action is cocompact and proper, $\mathcal{A}$ is finite.

If the transition matrix $A_{(e^-,h,e^+),(e'^-,h',e'^+)}$ is defined by 
 $$A_{(e^-,h,e^+),(e'^-,h',e'^+)}:=\begin{cases}1 & \text{ if } e^+=e'^-\\0 & \text{otherwise,}\end{cases}$$ 
then we obtain a subspace $\Sigma=\{(e_i^-,h_i,e^+_i)_{i\in \mathbb{Z}}:A_{(e_i^-,h_i,e^+_i),(e_{i+1}^-,h_i,e^+_{i+1})}=1\}$ of $\mathcal{A}^{\mathbb{Z}}$ and a shift map $\sigma:\Sigma\ra\Sigma$ defined by $ \sigma(x)_i=x_{i+1}$ for all $x\in \Sigma$. 
\begin{figure}[h]
\begin{center}
\begin{tikzpicture}[scale=1]
  \draw [thick](-4,0.4) -- (-3,1)--(-2,0.5)--(-1,1)--(0,0.6)--(1,0.8)--(2,0.5)--(3,0.7)--(4,0.4);   \draw [thick](-1.5,-0.1)--(-0.5,-0.5);\draw[thick] (0.5,-0.5)--(1.5,-0.3);
  \draw [->,thick,dashed,smooth] (0,-1.5)-- (-0.43,-0.58);   \draw [->,,thick,dashed,smooth] (0,-1.5)-- (0.43,-0.58);  \draw [->,,thick,dashed,smooth] (-0.5,0.76)--(-1.03,-0.23) ; 
  \draw [->,,thick,dashed,smooth]  (0.52,0.68)--(1,-0.28); 
  \node at (-4.4,0.4) {$\cdots$};  \node at (-1,0.37) {$\gamma_i$};\node at (1.1,0.27) {$\gamma_{i+1}$};\node at (4.4,0.4) {$\cdots$};\node at (-3.3,0.48) {$f_{i-3}$};\node at (-2.59,0.62) {$f_{i-2}$};\node at (-1.63,1) {$f_{i-1}$};\node at (-0.4,1.1) {$f_i$};\node at (0.54,1) {$f_{i+1}$};  \node at (1.6,0.9) {$f_{i+2}$}; \node at (2.58,0.38) {$f_{i+3}$}; \node at (3.45,0.34) {$f_{i+4}$}; \node at (-1.07,-0.49) {$\widetilde{e_i}$};\node at (1.08,-0.65) {$\widetilde{e_{i+1}}$};\node at  (0,-2)  {$\widetilde{v_{i+1}}$}; \node at (-0.8,-1) {$\gamma_{e_i,v_{i+1}}$}; \node at (1,-1.1) {$\gamma_{e_{i+1},v_{i+1}}$};
 \fill (-4,0.4) circle (2pt); \fill(-3,1)  circle (2pt);\fill (-2,0.5) circle (2pt); \fill  (-1,1) circle (2pt); \fill  (0,0.6) circle (2pt); \fill  (1,0.8) circle (2pt); \fill (2,0.5) circle (2pt); \fill (3,0.7) circle (2pt); \fill (4,0.4) circle (2pt);
  \fill  (0,-1.5) circle (2pt);\fill  (-1.5,-0.1) circle (2pt);\fill  (-0.5,-0.5) circle (2pt);\fill (0.5,-0.5) circle (2pt); \fill (1.5,-0.3) circle (2pt);
\end{tikzpicture}
\end{center}
\caption{The construction of $\Theta$}\label{figure4}
\end{figure}
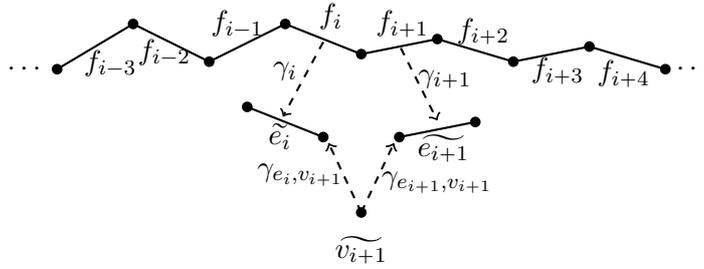

Let $g$ be a geodesic line in $Y$. Denote $\pi (T^i g):=w_i$ and $f_i:=[\pi (T^i g), \pi (T^{i+1} g)]$. Let $\widetilde{v_i}$ and $\widetilde{e_i}$ be the lifts of $v_i:=\G w_i$ and $e_i:=\G f_i$ which are chosen to define the groups $\G_{v_i}$ and $\G_{e_i}$, respectively. Choose $\g_i\in \G$ such that $\widetilde{e_i}=\g_i f_i$. The element 
$$h_{i+1}:=\g_{e_i,v_{i+1}}^{-1}\g_i\g_{i+1}^{-1}\g_{e_{i+1},v_{i+1}}$$
 is an element of $\G_{v_{i+1}}$. If $\widetilde{e_i}=\g'_i{f_i}$, then $\g'_i\g_i^{-1}\in\G_{e_i}$ and
$$\g_{e_i,v_{i+1}}^{-1}\g'_i\g'^{-1}_{i+1} \g_{e_{i+1},v_{i+1}}=\rho_{e_i,v_{i+1}}(\g'_i \g^{-1}_i)h_{i+1}\rho_{e_{i+1},v_{i+1}}(\g_{i+1}\g'^{-1}_{i+1}).
$$
Thus the element $[h_{i+1}]$ in $\rho_{e_i,v_{i+1}}(\G_{e_i})\backslash\G_{v_{i+1}}/\rho_{e_{i+1},v_{i+1}} (\G_{{e_{i+1}}})$ is independent of the choice of $\g_i$. Similarly, $[h_i]$ does not depend on the choice of $g'\in \G g$. 
 
 A map $\Theta$ from $\G\backslash Y$ to $\Sigma$ is defined by $\Theta(\G g)_i:=(e_{i-1},[h_i],e_i).$
The map $\Theta$ satisfies the following theorem.
\begin{lem}\emph{(\cite{BPP}, Theorem 5.1)}\label{scoding1} The map $\Theta$ is a bilipschitz homeomorphism from $(\G\backslash Y,T)$ to $(\Sigma,\sigma)$ satisfying $\sigma \Theta=\Theta T$. Furthermore, the Markov shift $(\Sigma,\sigma)$ is transitive if every degree of vertex is larger than 2 and $\Lambda_\G=\mathbb{Z}$.
\end{lem}

Let  $[w]$ be the cylinder set of an admissible word $w=a_1\cdots a_{q-p+1}$ for $p<q$ and let $a_i=(e_i^-,[h_i],e_i^+)$. The admissible word $w$ corresponds to an edge path $e_{1}^-\cdots e_{q-p+1}^-e_{q-p+1}^+$ in $\G\backslash \mathcal{T}$. Choose a lifting $[x_{p-1},x_{q+1}]$ in $\mathcal{T}$ of the path $e_{1}^-\cdots e_{q-p+1}^-e_{q-p+1}^+$. It follows from \eqref{eq:3.20} that
 \begin{equation}\label{eq:3.24}
\frac{ \Theta_*\overline{\nu}_{\lambda}([w,p,q])}{\overline{\nu}_{\lambda}(\G\backslash Y)}=\frac{\overline{\nu}_{\lambda}(\G C(g,p-1,q+1))}{\overline{\nu}_{\lambda}(\G\backslash Y)}=\frac{{\nu}_{\lambda}(C(g,p-1,q+1))}{|\G_{[x_{p-1},x_{q+1}]}|\overline{\nu}_{\lambda}(\G\backslash Y)},  \end{equation}
 for any $g\in Y$ with $\Theta(\G g)\in [w]$. Using \eqref{eq:3.24} and Lemma \ref{Gibbs}, we have the following lemma. 
 \begin{lem}\label{lem:3.14} The measure $\Theta_*\frac{\overline{\nu}_{\lambda}}{\overline{\nu}_{\lambda}(\G\backslash Y)}$ satisfies the weak Gibbs property, i.e. for any admissible word $w=a_1\cdots a_{p-q+1}$, there exists a constant $C_w$ such that for all $x\in [w,p,q]$,
 \begin{equation}\label{eq:4.12}
 \frac{1}{C_w}\leq \frac{\Theta_*\overline{\nu}_{\lambda}([w,p,q])/\overline{\nu}_{\lambda}(\G\backslash Y)}{e^{S_p^qF_\lambda\circ\Theta^{-1}(x)}}\leq C_w.
 \end{equation} 
  \end{lem}
By Harnack inequality \eqref{eq:2.2}, $\sup |F_\lambda\circ \Theta^{-1}|=\sup |F_\lambda|<\infty$. The space $\cM_{F_\lambda\circ \Theta^{-1}}$ is the space of $\sigma$-invariant probability measures on $\Sigma$.   Since $\Theta$ is bilipschitz, $F_\lambda\circ\Theta^{-1}$ is also H\"older continuous and satisfies $\underset{1\geq n}\sum n\text{var}_nF_\lambda\circ\Theta^{-1}<\infty$. Lemma \ref{lem:3.14} and Buzzi's theorem (Theorem \ref{buzzi}) show that $\Theta_*\frac{\overline{\nu}_{\lambda}}{\overline{\nu}_{\lambda}(\G\backslash Y)}$ is the equilibrium state of $F_\lambda\circ\Theta^{-1}$ and the pressure is zero. The map $\Theta_*$ is a bijection from the space of $T$-invariant measure probability measures on $\G\backslash Y$ to the space of $\sigma$-invariant probability measures on $\Sigma$. Thus we have the following lemma.
 
 \begin{lem}\label{vari} For any $\lambda\in [0,\lambda_0]$, the measure $\frac{\overline{\nu}_\lambda}{\overline{\nu}_{\lambda}(\G\backslash Y)}$ is the equilibrium state of $F_\lambda$ and the pressure of $F_\lambda$ is zero.
\end{lem}

 \subsection{Variational principle}\label{sec:3.6} In this section, we show that the measure $\overline{m}_\lambda$ is the equilibrium state for $f_\lambda$.
 
  The roof function $r$ on $\Sigma$ is the length of the edge $e_0^+$ of $x$ in $\Sigma$ and coincides with $\tau\circ \Theta^{-1}(x)$. Let $\Sigma_r:=(\Sigma\times \bR)/(x,r(s))\sim(\sigma(x),0)$ be the suspension of $\Sigma$. The suspension flow $\sigma_r^t$ on $\Sigma_r$ is defined by $\sigma_r^t ([x,s]):=[x,s+t].$ Then the map $\overline{\Theta}:[\G g,s]\mapsto[\Theta(\G g),s])$ is a homeomorphism from $Y_\tau$ to $\Sigma_r$ satisfying $\overline{\Theta}T_t=\sigma_r^t \overline{\Theta}$. By Theorem 5.9 in \cite{BPP}, the $\overline{\Theta}\circ\Phi^{-1}$ is a bilipschitz homeomorphism from $\G\backslash\cG\mathcal{T}$ to $\Sigma_r$ satisfying  $\overline{\Theta}\circ\Phi^{-1}\phi_t=\sigma_r^t \overline{\Theta}\circ\Phi^{-1}$.
  
  The map $({\overline{\Theta}}\circ\Phi^{-1})_* $ is a bijection from the space $\cM_\phi$ of $\phi_t$-invariant probability measures on $\G\backslash\cG\mathcal{T}$ to the space $\cM_{\sigma_r}$ of $\sigma_r^t$-invariant measures $\Sigma_r$. Thus the pressure of $f_\lambda-\delta_\lambda$ is the same as the pressure of $f_\lambda\circ\Phi\circ{\overline{\Theta}}^{-1}-\delta_\lambda$. 
 
Since $l_m\leq r(x)\leq l_M$ for all $x\in \Sigma$, the space of $\sigma$-invariant probability measures $m$ on $\Sigma$ with $\int_\Sigma r dm<+\infty$ coincides with the space $\cM_\sigma$ of $\sigma$-invariant probability measures. By the result of \cite{AK}, a bijective map $S$ from $\cM_\sigma$ to $\cM_{\sigma_r}$ is  defined by 
$$dS(m):= \frac{ds\times dm}{\int_\sigma r dm}.$$         
 for all $m\in \cM_\sigma$.
  For all $m\in \cM_\sigma$, we have
  \begin{equation}\label{potential}
  \begin{split}
  &\int_\Sigma\int_0^{r(x)} f_\lambda\circ\Phi\circ{\overline{\Theta}}^{-1}([x,s])-\delta_\lambda dsdm(x)\\
  &=\int_\Sigma\int_0^{r(x)}f_\lambda\circ\Phi([\Theta^{-1}(x),s])-\delta_\lambda dsdm(x)\\
&= \int_\Sigma\int_0^{r(x)}f_\lambda(\phi_s\Theta^{-1}(x))-\delta_\lambda dsdm(x)=\int_\Sigma F_\lambda(\Theta^{-1}(x))dm(x).
 \end{split}
 \end{equation}
By Abramov's formula in \cite{Ab}, for all $m\in \cM_\sigma$, 
\begin{equation}\label{Abramov}
h_{S(m)}(\sigma_r^1)=\frac{h_m(\sigma)}{\int_\Sigma r dm}.
\end{equation}
Using \eqref{potential} and \eqref{Abramov},
we obtain 
\begin{equation}\label{pressure}
P_{ f_\lambda\circ\Phi\circ\overline{\Theta}^{-1}-\delta_\lambda}(S(m))=\frac{P_{F_\lambda\circ\Theta^{-1}}(m)}{\int_\Sigma r dm}.	
\end{equation}
\begin{proof}[Proof of Theorem \ref{thm:1.4}] The equation \eqref{pressure} and Lemma \ref{vari} show that the equilibrium state of $ f_\lambda\circ\Phi\circ\overline{\Theta}^{-1}-\delta_\lambda$ is $S(\Theta_*\frac{\overline{\nu}_{\lambda}}{\overline{\nu}_{\lambda}(\G\backslash Y)})$. By \eqref{eq:3.21}, the measure $(\Phi\circ\overline{\Theta}^{-1})_*S(\Theta_*\frac{\overline{\nu}_{\lambda}}{\overline{\nu}_{\lambda}(\G\backslash Y)})$ is equal to $\overline{m}_\lambda$. This implies Theorem \ref{thm:1.4}.
\end{proof}
Since $\delta_\lambda$ is non-positive by Proposition 4.10 in \cite{HL}, we have the following corollary.
\begin{coro}\label{2203313} The map $\lambda\mapsto P_\lambda$ is continuous on $[0,\lambda_0]$ and $P_\lambda$ is non-positive.
\end{coro}
After this, we will use $P_\lambda$ instead of $\delta_\lambda$.
\section{Uniform rapid mixing}\label{Sec:5}
 \subsection{Uniform rapid mixing property of Gibbs measure of $f_\lambda$} Applying the results in \cite{BPP}, \cite{D} and \cite{LL}, we prove the uniform mixing property of Gibbs measure of $f_\lambda$ under assumptions in Theorem \ref{thm:1.3}. 

Let $\{\psi_t\}$ be a flow on a metric space $X$. The space $C^{k,\alpha}_b(X)$ is the space of functions $f$ satisfying the following properties:  
\begin{itemize}
\item For any $x\in X$, the function $t\mapsto f(\psi_t x)$ is a $C^k$-function.
\item For any integer $i\in[0,k]$, the function $\partial_t^if(x):=\frac{\partial^k}{\partial t^k}|_{t=0}f(\psi_t x)$ is in $C^{\alpha}_b(X)$. 
\end{itemize}
The space $C^{k,\alpha}_b(X)$ is a Banach space with respect to the norm $\|f\|_{k,\alpha}=\max_{0\leq i\leq k} \|\partial_t^i f\|_{\alpha}.$ For any $f_1,f_2\in C_b^{k,\alpha}$, the correlation function of $f_1$ and $f_2$ at time $t$ under the flow $\{\psi_t\}$ for a probability measure $m$ is defined by
\begin{equation}
\rho_{f_1,f_2,m}(t):=\left|\int_{X} f_1(x)f_2(\psi_t x)dm(x)-\int_X f_1dm\int_X f_2dm\right|.
\end{equation}
\begin{defn}A $\psi_t$-invariant probability measure $m$ on a metric space $X$ is \textit{mixing} if for any $f_1,f_2$ in $C_b^{k,\alpha},$ $\underset{t\rightarrow\infty}\lim \rho_{f_1,f_2,m}(t)=0.$
  
A flow $\{\psi_t\}$ on a metric space $X$ has \textit{superpolynomial decay} of $\alpha$-correltaion for a $\psi_t$-invariant probability measure $m$ if for any $n$, there exist $C$ and $k$ such that for any $f_1, f_2$ in $C^{k,\alpha}_b(X),$
$$\rho_{f_1,f_2,m}(t)\leq C(1+|t|)^{-n}\|f_1\|_{k,\alpha}\|f_2\|_{k,\alpha}.$$
The flow $\{\psi_t\}$ is \textit{rapid mixing} with respect to $m$ if  $\{\psi_t\}$ has superpolynomial decay of $\alpha$-correlation for some $\alpha$.
 \end{defn}
 
In Section \ref{sec:1}, we introduced the Diophatine assumption for the rapid mixing of Gibbs measures of continuous functions on the unit tangent bundle $T^1\mathcal{T}$ (see Definition \ref{def:1.1}). The following lemma demonstrates the mixing property of the geodesic flow on $\mathcal{T}$
 \begin{lem} \emph{(\cite{BPP} Theorem 9.7)}\label{lem:5.2} Let $\mathcal{T}$ be a topologically complete locally finite metric tree. A discrete group $\Gamma$ acts isometrically and geometrically on $\Gamma$. Suppose that the normalized Gibbs measure $\overline{m}_f$ of a potential $f$ on $\G\backslash T^1\mathcal{T}$ exists and is mixing for the geodesic flow $\phi_t$ and $0<l_m\leq l_M<\infty.$ Suppose that $\Lambda_\G$ is Diophantine. Then the geodesic flow on $\G\backslash\cG\mathcal{T}$ is rapid mixing with respect to ${\overline{m_f}}$.

   \end{lem}
   
 The cocompact action of $\Gamma$ guarantees that Gibbs measure on $\Gamma\backslash \cT$ for the potential $f\equiv0$ is finite. If the length spectrum $\Lambda_\Gamma$ dense and Gibbs measure of $f\equiv 0$ is finite, the normalized Gibbs measure, which is called Bowen-Margulis measure, is mixing (\cite{BPP} Theorem 4.17). Under the assumption of Theorem \ref{thm:1.3}, the Bowen-Margulis measure satisfies the assumption of Lemma \ref{lem:5.2} and has the rapid mixing property. Using this fact and the result in \cite{LL}, we will obtain the uniform rapid mixing of $\overline{m}_\lambda$. To do this, we introduce the following definition.
\begin{defn}\label{def:5.3} A flow $\psi_t$ on a metric space $X$ is \textit{topologically power mixing} if there exist $t_0,\beta>0$ such that for any $r>0$, $t\geq\max\{r^{-\beta},t_0\}$, and $x_1,x_2\in X$
$$\psi_t B(x_1,r)\cap B(x_2,r)\neq \phi,$$
where $B(x,r)$ is the ball of radius $r$ centered at $x$.
\end{defn}
\begin{thm}\label{thm:5.4} \emph{(\cite{LL}, Theorem 7.2, Uniform rapid mixing)} Let $(\Sigma,\sigma)$ be a Markov shift constructed from a finite alphabet $\cA$ and a transition matrix $A$. Suppose that the flow $\{\sigma_r^t\}$ on $\Sigma_r$ is topologically power mixing. There exists $\alpha_0>0$ satisfying the following property. Let $f_0\in C_b^{\alpha_0}(\Sigma_r)$. There exist $\epsilon>0$, $\alpha>0$ and $c_1,c_2>0$ such that for any $f$ with $\|f-f_0\|_{\alpha_0}<\epsilon$ and any $f_1,f_2,f_3\in C_b^{\alpha}(\Sigma_r)$, we obtain for any $s,t>0$,
\begin{equation}
\begin{split}
&\left|\frac{1}{t}\int_0^t \int_{\Sigma_r}f_1(f_2\circ\sigma_r^s)(f_3\circ\sigma_r^t)d\overline{m_f}ds-\int_{\Sigma_r}f_1d\overline{m_f}\int_{\Sigma_r}f_2d\overline{m_f}\int_{\Sigma_r}
f_3d\overline{m_f}\right|\\
&\leq c_1[(1+s)^{-c_2}+(1+t)^{-c_2}] \|f_1\|_\alpha\|f_2\|_\alpha\|f_3\|_\alpha.\nonumber
\end{split}
\end{equation}
Here $m_f$ is the Gibbs measure of $f$.
\end{thm}
 The assumption of Theorem 7.2 \cite{LL} is that the flow is Anosov. An Anosov flow is bilipschitz homeomorphic to the suspension of a finite Markov shift (see \cite{PP} Chapter 9 and Appendix III). Theorem 7.2 in \cite{LL} was proved for the suspension. Thus the proof of Theorem \ref{thm:5.4} is valid. Theorem 1 in \cite{D} shows that the suspension flow $\{\sigma_r^t\}$ on the suspension $\Sigma_r$ is topologically power mixing if a Gibbs measure of a potential has the rapid mixing property. Since $\overline{\Theta}\circ\Phi^{-1}$ 
 is bilipschitz, Theorem \ref{thm:5.4} implies that the geodesic flow has the uniform rapid mixing property with respect to $\overline{m}_\lambda$. As in the proof of Proposition 4.2 in \cite{LL}, using Proposition \ref{prop:2.11} and Theorem \ref{thm:5.4}, we have the following corollaries.

 \begin{coro}\label{coro:5.5}\emph{(\cite{LL} Proposition 4.1)} Suppose that $\mathcal{T}$ and $\Gamma$ satisfy the assumption of Theorem \ref{thm:1.3}. There exists a positive constant $\alpha$ satisfying the following property. For any $\epsilon>0$ and $\alpha$-H\"older continuous functions $f_1,f_2$ on $\Gamma\backslash\cG\mathcal{T}$ with $\int f_i d\overline{m}_\lambda>0$, there exist constants $T_1$ and $a_1$ such that for all $t>T_1$ and $\lambda\in[\lambda_0-a_1,\lambda_0]$,
 $$ \int_{\cG\mathcal{T}} f_1 f_2\circ \phi_t d\overline{m}_\lambda\asymp_{1+\epsilon} \int_{\cG\mathcal{T}} f_1 d\overline{m}_\lambda \int_{\cG\mathcal{T}} f_2 d\overline{m}_\lambda.$$
 The constant $T_1$ depends on $\|f_1\|_\alpha$, $\|f_2\|_\alpha,$ $\inf \int f_1 d\overline{m}_\lambda,$ and $\inf \int f_2 d\overline{m}_\lambda$ whereas $T_1$ is independent of $\lambda\in[\lambda_0-a_1,\lambda_0].$ 
 \end{coro}
  \begin{coro}\label{coro:5.6}\emph{(\cite{LL} Proposition 4.2)} Suppose that $\mathcal{T}$ and $\Gamma$ satisfy the assumption of Theorem \ref{thm:1.3}. There exists a positive constant $\alpha'$ satisfying the following property. For any $\epsilon>0$ and $\alpha'$-H\"older continuous functions $f_1,f_2,f_3$ on $\Gamma\backslash\cG\mathcal{T}$ with $\int f_i d\overline{m}_\lambda>0$, there exist constants $T_2$ and $a_2$ such that for all $t>T_2$ and $\lambda\in[\lambda_0-a_2,\lambda_0]$,
 $$ \frac{1}{t}\int_0^t\int_{\cG\mathcal{T}} f_1 f_2\circ\phi_s f_3\circ \phi_t d\overline{m}_\lambda ds\asymp_{1+\epsilon} \int_{\cG\mathcal{T}} f_1 d\overline{m}_\lambda\int_{\cG\mathcal{T}} f_2 d\overline{m}_\lambda \int_{\cG\mathcal{T}} f_3 d\overline{m}_\lambda.$$
 The constant $T_2$ depends on $\|f_1\|_\alpha$, $\|f_2\|_\alpha,$ $\|f_3\|_\alpha,$ $\inf \int f_1 d\overline{m}_\lambda,$ $\inf \int f_2 d\overline{m}_\lambda$ and $\inf \int f_3 d\overline{m}_\lambda,$ whereas $T_2$ is independent of $\lambda\in[\lambda_0-a_2,\lambda_0].$ 
 \end{coro}
  \subsection{Change of the description of Gibbs measure}In this section, we redescribe Gibbs measure of $f_\lambda$ by changing the parametrization. 
  
 Let $\cG^+ \mathcal{T}$ and $\cG^-\mathcal{T}$ be the set of geodesic rays defined on $[0,+\infty)$ and $(-\infty,0]$, respectively. Denote $\cG^\pm_x\mathcal{T}:=\{w\in \cG^\pm \mathcal{T}: w(0)=x\}$. The geodesic ray $w$ is considered as a path defined on $\mathbb{R}$ when we define $w(\mp t):=w(0)$ for any $w\in \cG^\pm \mathcal{T}$ and $t\in [0,\infty)$. The distance between $w$ and $w'$ in $\cG^\pm\mathcal{T}$ is defined by
 $$d_{\cG\mathcal{T}}(w,w'):=\int_{-\infty}^\infty d(w(t),w'(t))e^{-2|t|}dt.$$
 
 Using the end point homeomorphism $w\mapsto w_\pm:=\underset{t\rightarrow\pm\infty}\lim w(t)$ for any $w\in\cG^{\pm}_x\mathcal{T}$, a measure $\mu_{\cG^{\pm}_x\mathcal{T}}^\lambda$ on $\cG^{\pm}_x\mathcal{T}$ is defined by 
 $$\mu_{\cG^{\pm}_x\mathcal{T}}^\lambda(w):=\mu_{x,\lambda}(w_\pm).$$
 
The stable leaf $W^{0+}(g)$ (the unstable leaf $W^{0-}(g)$, resp.) of a geodesic line $g$ is the set of geodesic lines $g'$ such that $g_+=g'_+$ ($g_-=g_-'$, resp.). The strong stable leaf $W^+(g)$ and the strong unstable leaf $W^-(g)$ of $g$ are defined by
$$W^{\pm}(g):=\{g'\in W^{0\pm}(g): \lim_{t\rightarrow \infty} d(g(\pm t),g'(\pm t))=0\}.$$
Similarly, we obtain the leaves $W^{\pm}(w)$ and $W^{0\pm}(w)$ of $w\in\cG^\pm\mathcal{T}$. 

Fix $w\in \mathcal{G}^\pm\mathcal{T}$. As an application of the end point homeomorphism $g\mapsto g_\mp$ from $W^\pm(w)$ to $\partial \mathcal{T}\backslash\{w_\pm\}$, a measure $\mu_{W^{\pm}(w)}^\lambda$ on $W^\pm(w)$ is defined for any $g\in W^\pm(w)$, by
$$\mu_{W^{\pm}(w)}^\lambda(g):=k^2_\lambda(x_0,g(0),g_{\mp})e^{P_{\lambda}\beta_{g_{\mp}}(x_0,g(0))}d\mu_{x_0,\lambda}(g_\mp).$$
The measure $\mu_{W^{\pm}(w)}^\lambda$ dose not depend on the choice of $x_0$. 

The homeomorphism from $W^\pm(w)\times \mathbb{R}$ to $W^{0\pm}(w)$ defined by
$$(g,s)\mapsto g':=\phi_s g.$$
Using the above homeomorphism, we obtain a measure $\nu^{\lambda}_{W^{0\pm}(w)}$ on $W^{0\pm}(w)$ defined for any $g=\phi_s g'\in W^{0\pm}(w)$, by
 $$\nu^{\lambda}_{W^{0\pm}(w)}(g):=\theta^\lambda_{g'(0)}(g_\mp,w_\pm)^2k^2_\lambda(w(0),g'(0),w_\pm)ds d\mu_{W^{\pm}(w)}^\lambda(g').$$
\begin{lem}\label{lem:5.7} For any $\lambda\in [0,\lambda_0]$, $x\in \mathcal{T}$ and $w\in \cG_x\mathcal{T}$, the measures $\mu_{\cG^{\pm}_x\mathcal{T}}^\lambda,$ $\mu_{W^{\pm}(w)}^\lambda$ and $\nu^{\lambda}_{W^{0\pm}(w)}$ are $\G$-equivariant, i.e. 

$$\mu^\lambda_{\mathcal{G}_x^\pm \mathcal{T}}(g)=\mu^\lambda_{\mathcal{G}_{\gamma x}^\pm \mathcal{T}}(\gamma g),\,\mu_{W^{\pm}(w)}^\lambda(g)=\mu_{W^{\pm}(\gamma w)}^\lambda(\gamma g)\text{ and }\nu^{\lambda}_{W^{0\pm}(w)}(g)=\nu^{\lambda}_{W^{0\pm}(\g w)}(\g g)$$
for any $\gamma\in \Gamma$ and $g\in W^{\pm}(w).$
\end{lem}
\begin{proof}By \eqref{eq:3.3}, $\mu_{\cG^{\pm}_x\mathcal{T}}^\lambda,$ $\mu_{W^{\pm}(w)}^\lambda$ are $\G$-equivariant. Since $\mu_{W^{\pm}(w)}^\lambda$ is $\G$-equivariant, by definition, we have
\begin{eqnarray}
\nonumber \nu^{\lambda}_{W^{0\pm}(w)}(g)&=&\theta^\lambda_{g'(0)}(g_\mp,w_\pm)^2k^2_\lambda(w(0),g'(0),w_\pm)ds d\mu_{W^{\pm}(w)}^\lambda(g')\\
\nonumber&=&\theta^\lambda_{\g g'(0)}(\g g_\mp,\g w_\pm)^2k^2_\lambda(\g w(0),\g g'(0),\g w_\pm)ds d\mu_{W^{\pm}(\g w)}^\lambda(\g g')\\
&=&\nu^{\lambda}_{W^{0\pm}(\g w)}(\g g).\nonumber
\end{eqnarray}
This shows that $\mu_{W^{\pm}(w)}^\lambda$ is $\G$-equivariant.
\end{proof}
 
 \begin{lem}\label{decomposition} For any integrable function function $f$ on $\cG\mathcal{T}$ and $\lambda\in [0,\lambda_0]$, 
 $$\int_{\cG\mathcal{T}} f d m_\lambda=\int_{\cG_x^{\pm} \mathcal{T}}\int_{W^{0\pm}(w)}f(g)\nu^{\lambda}_{W^{0\pm}(w)}(g)\mu_{\cG_x^\pm\mathcal{T}}^\lambda(w)$$
 \end{lem}
 \begin{proof}  
 For any $w\in \cG_x^\pm \mathcal{T}$ and $g'\in W^\pm(w)$,
 \begin{eqnarray}\label{gromov}
 \nonumber2(g'_- |g'_+)_{g'(0)}&=&2(g_+' |g_-')_x+\beta_{g'_+}(g'(0),x)+\beta_{g'_-}(g'(0),x)\\
 \nonumber&=&2(g'_+ |g'_-)_x+\beta_{g'_\mp}(g'(0),x)+\beta_{w_\pm}(w(0),x)
 \\&=&2(g'_\mp |w'_\pm)_x-\beta_{g'_\mp}(x,g'(0)).\nonumber
 \end{eqnarray}
 Since $(g'_- |g'_+)_{g'(0)}=0$, we have $2(g'_\mp |w'_\pm)_x=\beta_{g'_\mp}(x,g'(0)).$
 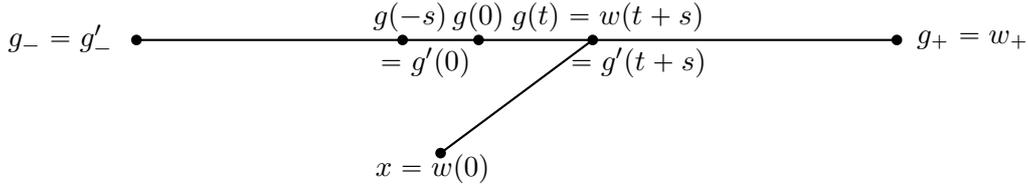
\begin{figure}[h]
 \begin{center}
\begin{tikzpicture}[scale=1]
  \draw [thick](-5,1) -- (5,1);   \draw [thick](-1,-0.5) -- (1,1);
  \node at (-6,1) {$g_-=g'_-$};\node at (-0.5,1.3) {$g(0)$};\node at (-1.4,1.3) {${g(-s)}$};\node at (-1.2,0.7) {$={g'(0)}$}; \node at (1.2,1.3) {$g(t)=w(t+s)$};\node at (1.6,0.7) {$=g'(t+s)$};\node at (6,1) {$g_+=w_+$}; \node at (-1.1,-0.7) {$x=w(0)$};\fill (-5,1)   circle (2pt); \fill (-0.5,1)    circle (2pt); \fill(-1.5,1)  circle (2pt);\fill (1,1)    circle (2pt); \fill (5,1)    circle (2pt);\fill (-1,-0.5)    circle (2pt);
\end{tikzpicture}
\end{center}
\caption{The map $(g_-,g_+,t)\mapsto(w,g',s)\in \mathcal{G}_x^+\mathcal{T}\times W^+(w)\times \mathbb{R}$}\label{figure5}
\end{figure}

 Let $(g_-,g_+,t)$ be Hopf parametrization of $g\in \cG\mathcal{T}$ with the base point $x$. Using the endpoint homeomorphism $w\mapsto w_\pm$ from $\cG_x^\pm\mathcal{T}$ to $\partial \mathcal{T}$ and the end point homeomorphism $g'\mapsto g'_\mp$ from $W^\pm(w)$ to $\partial \mathcal{T}\backslash\{w_\pm\}$, we have a triple $(w,g',s)\in \cG^\pm_x\mathcal{T}\times W^\pm(w)\times \mathbb{R}$ where $w_\pm=g_\pm'=g_\pm$, $g'_\mp=g_\mp$ and $\phi_s g'=g.$ Note that $(g_-|g_+)_x=t+s$ (see Figure \ref{figure5}). 
 
 The second equality below follows from the definition of $\mu_{\cG^{\pm}_x\mathcal{T}}^\lambda$ and $\mu_{W^{\pm}(w)}^\lambda$.
 \begin{displaymath}\begin{array}{rcl}
&&\nonumber \displaystyle\int_{\cG\mathcal{T}} f d m_\lambda\\
&=&\displaystyle\int_{g_\pm\in\partial\mathcal{T}}\int_{g_\mp\in\partial\mathcal{T}\backslash\{g_\pm\}}\int_{\mathbb{R}}f(g)\theta_x^\lambda(g_-,g_+)^2e^{2P_{\lambda}(g_-|g_+)_x}dtd\mu_{x,\lambda}(g_-)\mu_{x,\lambda}(g_+)\\ &=&\displaystyle\int_{w\in\cG_x^\pm\mathcal{T}}\int_{g'\in W^\pm(w)}\int_{\mathbb{R}}f(\phi_s g')k_\lambda^2(g'(0),x,g'_\mp)\theta_{x}^\lambda(g'_\mp,w_\pm)^2 dsd\mu_{W^\pm(w)}^\lambda(g')d\mu_{\cG_x^{\pm}\mathcal{T}}^\lambda(w).
\end{array}\end{displaymath}
 
Since $\cG\mathcal{T}=\bigcup_{w\in\cG_x^\pm\mathcal{T}} W^{0\pm}(w)$, $\theta_x^\lambda(\xi,\zeta)=\theta_y^\lambda(\xi,\zeta)k_\lambda(x,y,\xi)k_\lambda(x,y,\zeta)$ and $w(0)=x$ for all $w\in \cG_x^\pm\mathcal{T}$, we have
\begin{displaymath}
\begin{array}{rcl}
&&\displaystyle\int_{\cG\mathcal{T}} f d m_\lambda\\&=&\displaystyle\int_{w\in\cG_x^\pm\mathcal{T}}\int_{g'\in W^\pm(w)}\int_{\mathbb{R}}f(\phi_s g')k_\lambda^2(x,g'(0),w_\pm)\theta^\lambda_{g'(0)}(g'_\mp,w_\pm)^2 dsd\mu_{W^\pm(w)}^\lambda(g')d\mu_{\cG_x^{\pm}\mathcal{T}}^\lambda(w)\\
&=&\displaystyle\int_{w\in\cG^\pm_x\mathcal{T}}\int_{g\in W^{0\pm}(w)}f(g)d\nu_{W^{0\pm}(w)}^\pm(g)d\mu_{\cG_x^{\pm}\mathcal{T}}^\lambda(w).
 \end{array}
\end{displaymath}
 The above equation completes the proof.
 \end{proof}
 \subsection{Construction of H\"older continuous functions} In this section, using the measures introduced in the previous section, we construct H\"older continuous functions $h_{x,\lambda,\eta}^\pm$ and show the propositions related to $h_{x,\lambda,\eta}^\pm$. The propositions in this section will be also used in the Section \ref{subsec:5.4}. We recall that $\mathcal{G}_x\mathcal{T}$ is the space of geodesic lines such that $g(0)=x$.
 
 Let $w\in \cG_x^\pm \mathcal{T}$. Denote $$\cG_{x,\eta} \mathcal{T}:=\underset{s\in (-\eta,\eta)}\bigcup \phi_s\cG_x\mathcal{T}\text{ and }W_{x,\eta}^{0\pm}(w):=\underset{s\in (-\eta,\eta)}\bigcup \phi_s(W^{\pm}(w)\cap \cG_{x}\mathcal{T}).$$ 
 \begin{lem} \label{lem:5.9}For any $x\in \mathcal{T}$, $w\in \cG_x^\pm\mathcal{T}$ $\lambda\in[0,\lambda_0]$ and $t>0$,
 $$\nu^\lambda_{W^{0\pm}(w)}(W_{x,\eta}^{0\pm}(w))=2\eta\int_{\cO_{w(\pm t)}(x)}\theta^\lambda_x(\xi,w_\pm)^2\mu_{x,\lambda}(\xi).$$
 \end{lem}
 \begin{proof} Since $w(\pm t) = g'(\pm t)$ for any $g'\in W^{\pm}(w)\cap \mathcal{G}_x\mathcal{T}$ and $t>0$, $g_\mp\in \mathcal{O}_{w(\pm t)}(x)$ for any $t>\eta$ and $g\in W_{x,\eta}^{\pm}(w)$. By definition of  $\nu^\lambda_{W^{0\pm}(w)}$ and $W_{x,\eta}^{0\pm}(w)$,
 \begin{equation}
\begin{split}
\nonumber \nu^\lambda_{W^{0\pm}(w)}(W_{x,\eta}^{0\pm}(w))&=\int_{-\eta}^\eta\int_{g'\in W^{\pm}(w)\cap \cG_x\mathcal{T}}\theta^\lambda_x(g'_\mp,w_\pm)^2\mu_{W^{\pm}(w)}^\lambda(g')ds\\
&=2\eta\int_{g'\in W^{\pm}(w)\cap \cG_x\mathcal{T}}\theta^\lambda_x(g'_\mp,w_\pm)^2\mu_{W^{\pm}(w)}^\lambda(g')\\
&=2\eta\int_{\cO_{w(\pm t)}(x)}\theta^\lambda_x(\xi,w_\pm)^2\mu_{x,\lambda}(\xi).
\end{split}
\end{equation}
The above equation completes the proof.
 \end{proof}
A map $f_x^\pm:\cG\mathcal{T}\rightarrow\cG^{\pm}_x\mathcal{T}$ is defined by $f_x^\pm(g):=w,$ where $w(0)=x$ and $w_\pm=g_\pm$. 
Then we have a function $h_{x,\lambda,\eta}^\pm:\cG\mathcal{T}\rightarrow \mathbb{R}$ defined by
\begin{displaymath}
h_{x,\lambda,\eta}^\pm(g):=\begin{cases}\displaystyle\frac{2\eta-2d(x,\pi (g)))}{\eta\nu^\lambda_{W^{0\pm}(w)}(W_{x,\eta}^{0\pm}(w))}& \text{ if }\,\,g\in \cG_{x,\eta}\mathcal{T}\\
0&\text{otherwise}
\end{cases}
\end{displaymath}
for any $g\in \cG\mathcal{T}$, where $f_x^\pm(g)=w.$ 
\begin{prop}\label{5.9}For any $\lambda\in [0,\lambda_0]$ and $\eta\in (0,1),$ the function $h_{x,\lambda,\eta}^\pm$ is H\"older continuous. The norm $\|h_{x,\lambda,\eta}^\pm\|_\alpha$ is bounded above by a constant which only depends on $\eta$. 
\end{prop}
\begin{proof}
Let $f^{\pm}_{x}(g)=w$. By Lemma \ref{lem:5.9}, \eqref{eq:3.9} and \eqref{eq:3.16}, there exists a constant $C$ such that for any $\lambda\in [0,\lambda_0]$, $x\in \mathcal{T}$ and $g\in \cG_x\mathcal{T},$
\begin{equation}
\nonumber\nu^\lambda_{W^{0\pm}(w)}(W_{x,\eta}^{0\pm}(w))\overset{\substack{\text{Lem}\\\ref{lem:5.9}}}=2\eta\int_{\cO_{w(\pm 2\eta)}(x)}\theta^\lambda_x(\xi,w_\pm)^2\mu_{x,\lambda}(\xi)\overset{{\eqref{eq:3.9} + \eqref{eq:3.16}}}>2\eta C.
\end{equation}
 Since $f_x^\pm(\phi_s g)=f_x^\pm(g)=w$ for any $s\in\mathbb{R}$ and $g\in \cG_x\mathcal{T}$, we have for any $s_1,s_2 \in[-\eta,\eta],$
 \begin{equation} \label{A}
 |h_{x,\lambda,\eta}^\pm(\phi_{s_1} g)-h_{x,\lambda,\eta}^\pm(\phi_{s_2} g)|\leq \frac{2|s_1-s_2|}{\eta\nu^\lambda_{W^{0\pm}(w)}(W_{x,\eta}^{0\pm}(w))}<\frac{|s_1-s_2|}{\eta^2C}.
 \end{equation}
 
Let $f^{\pm}_{x}(g_i)=w_i$ for $i=1,2$. 
By Corollary \ref{coro:2.7}, there exist $C'>0$ and $\rho\in (0,1)$ such that for any $\lambda\in [0,\lambda_0]$, $g_1,g_2\in \cG_{x,\eta}\mathcal{T}$ with $(w_{1\pm}|w_{2\pm})_{\pi(g_1)}=t>1$, and $\xi\in \cO_{g_1(\pm 1)}(x),$ $$|\theta_x^\lambda(w_{1\pm},\xi)^2-\theta_x^\lambda(w_{2\pm},\xi)^2|\leq C'\rho^t.$$ 
 By Lemma \ref{lem:5.9}, the above inequality and \eqref{eq:3.9}, there exists a constant $C''$ satisfying
\begin{equation}\begin{split}
\nonumber&|\nu^\lambda_{W^{0\pm}(w_1)}(W_{x,\eta}^{0\pm}(w_1))-\nu^\lambda_{W^{0\pm}(w_2)}(W_{x,\eta}^{0\pm}(w_2))|\\
&\leq 2\eta\int_{\cO_{g_1(\pm 2\eta)}(x)}|\theta^\lambda_x(\xi,w_{1\pm})^2-\theta^\lambda_x(\xi,w_{2\pm})^2|\mu_{x,\lambda}(\xi)\leq 2\eta C''\rho^t.
\end{split}\end{equation}

For any $g_1,g_2\in \cG_x\mathcal{T}$ with $(w_{1\pm}|w_{2\pm})_{\pi(g_1)}=t>1$,
\begin{equation}\label{B}
\begin{split}
 |h_{x,\lambda,\eta}^\pm(g_1)-h_{x,\lambda,\eta}^\pm(g_2)|&=\frac{2|\nu^\lambda_{W^{0\pm}(w_1)}(W_{x,\eta}^{0\pm}(w_1))-\nu^\lambda_{W^{0\pm}(w_2)}(W_{x,\eta}^{0\pm}(w_2))|}{\eta^2|\nu^\lambda_{W^{0\pm}(w_1)}(W_{x,\eta}^{0\pm}(w_1))\nu^\lambda_{W^{0\pm}(w_2)}(W_{x,\eta}^{0\pm}(w_2))|}\\
&<\frac{|\nu^\lambda_{W^{0\pm}(w_1)}(W_{x,\eta}^{0\pm}(w_1))-\nu^\lambda_{W^{0\pm}(w_2)}(W_{x,\eta}^{0\pm}(w_2))|}{2\eta^4 C^2}\leq\frac{C''\rho^t}{\eta^3C^2}.
\end{split}
\end{equation}
 Let $\alpha:=\min\{\frac{1}{2},-\frac{1}{2}\log\rho\}.$ For any $g_1$, $g_2\in W^{0\pm}_{x,\eta}(w)$ with $g_1(0)=g_2(s)$ and $(g_{1+}|g_{2+})_{\pi(g_1)}=t$ for some sufficiently large $t$, by \eqref{A}, \eqref{B} and Lemma \ref{lem:2.1}, 
\begin{equation}\nonumber
\begin{split}
|h_{x,\lambda,\eta}^\pm(g_1)-h_{x,\lambda,\eta}^\pm(g_2)|\leq& |h_{x,\lambda,\eta}^\pm(g_1)-h_{x,\lambda,\eta}^\pm(\phi_s g_1)|+|h_{x,\lambda,\eta}^\pm(\phi_s g_1)-h_{x,\lambda,\eta}^\pm(g_2)|\\
\overset{\eqref{A}+\eqref{B}}\leq& \frac{|s|}{2\eta^2C}+\frac{C''\rho^t}{\eta^2C^3}=\frac{d(\pi(g_1),\pi(g_2))}{2\eta^2C}+\frac{C''}{\eta^3C^2}d_{\pi(g)}(g_{1+},g_{2+})^{-\log\rho}\\
\overset{\text{Lem }\ref{lem:2.1}}\leq& C'''(\frac{1}{2\eta^2C}+\frac{C''}{\eta^3C^2})d_{\mathcal{G}\mathcal{T}}(g_1,g_2)^{\alpha}.
\end{split}
\end{equation}
The constant $C'''$ is from Lemma \ref{lem:2.1}. The function $h_{x,\lambda,\eta}^\pm$ is $\alpha$-H\"older continuous and $\|h_{x,\lambda,\eta}^\pm\|_\alpha<C'''(\frac{1}{2\eta^2C}+\frac{C''}{\eta^3C^2})$. 
\end{proof}
Since Lemma \ref{lem:5.7} shows that for all $x\in \mathcal{T}$, $\g\in \G$ and $\lambda\in[0,\lambda_0]$,
 \begin{equation}\label{eq:5.5}
 h_{x,\lambda,\eta}^\pm(g)=h_{\g x,\lambda,\eta}^\pm(\g g),
 \end{equation} 
 we obtain a $\G$-invariant function $\widetilde{H}_{x,\lambda,\eta}$ on $\cG\mathcal{T}$ defined by 
  \begin{equation}\label{eq:5.6}
  \widetilde{H}_{x,\lambda,\eta}^\pm(g):=\sum_{\g\in \G}h_{\g x,\lambda,\eta}^\pm(g).
  \end{equation}
The equality \eqref{eq:5.5} also implies $\widetilde{H}_{x,\lambda,\eta}^\pm(g)=\widetilde{H}_{\g x,\lambda,\eta}^\pm(g)$ for any $g\in \cG\mathcal{T}.$
\begin{lem}\label{lem:5.8} For any point $x\in\mathcal{T},$ $w\in \cG_x^\pm\mathcal{T}$ and $\lambda\in [0,\lambda_0]$,
\begin{equation}\label{eq:5.7}
\int_{W^{0\pm}(w)}h_{x,\lambda,\eta}^\pm d\nu_{W^{0\pm}(w)}^\lambda=1.
\end{equation}\end{lem}
\begin{proof}
Every $g\in W^{0\pm}_{x,\eta}(w)$ is of the form $g=\phi_s g'$ where $g'\in W^\pm(w)\cap \cG_x\mathcal{T}$ and $s\in [-\eta,\eta]$. By the definition of $\nu_{W^{0\pm}(w)}^\lambda$ and $h_{x,\lambda,\eta}^\pm$, we have 
\begin{equation}
\begin{split}
\nonumber\int_{W^{0\pm}(w)}h_{x,\lambda,\eta}^\pm d\nu_{W^{0\pm}(w)}^\lambda&=
\int_{W^{0\pm}_{x,\eta}(w)}\frac{2\eta-2d(x,\pi (g))}{\eta\nu^\lambda_{W^{0\pm}(w)}(W_{x,\eta}^{0\pm}(w))}d\nu_{W^{0\pm}(w)}^\lambda(g)\\
&=\frac{\int_{g'\in W^{\pm}(w)\cap \cG_x\mathcal{T}}\int_{-\eta}^\eta2(\eta-|s|)dsd\mu_{W^{0\pm}(w)}^\lambda(g')}{\eta\nu^\lambda_{W^{0\pm}(w)}(W_{x,\eta}^{0\pm}(w))}\\
&=\frac{2\eta\int_{\xi\in\mathcal{O}_{w(\pm 2\eta)}(x)}\theta^\lambda_x(\xi,w_\pm)^2d\mu_{x,\lambda}(\xi)}{\nu^\lambda_{W^{0\pm}(w)}(W_{x,\eta}^{0\pm}(w))}\overset{\substack{\text{Lem}\\\ref{lem:5.9}}}=1.
\end{split}
\end{equation}
Lemma \ref{lem:5.9} shows the last equation above. 
\end{proof}
\begin{lem}\label{lem:5.10} Let $f$ be a bounded non-negative H\"older continuous function on $\cG^\pm\mathcal{T}$. For any $\epsilon>0$, there exists a constant $\eta'$ such that for any $x\in \mathcal{T}$, $\eta<\eta'$ and $\lambda\in [0,\lambda_0]$,
\begin{equation}
\int_{\cG\mathcal{T}} fh_{x,\lambda,\eta}^\pm dm_\lambda\asymp_{(1+\epsilon)}\int_{\cG_x^\pm \mathcal{T}}f(w)d\mu_{\cG_x^\pm\mathcal{T}}^\lambda(w).
\end{equation}
\end{lem}
\begin{proof}The function $f$ is considered as a bounded non-negative H\"older function on $\cG\mathcal{T}$ when we define $f(g):=f(g|_{[0,\infty)})$ ($f(g):=f(g|_{(-\infty,0]})$, resp.). There exists a constant $\eta'$ such that $f(w)\asymp_{(1+\epsilon)}f(g)$ for all $g\in W^{0\pm}_{x,\eta'}(w)$. Using Lemma \ref{decomposition}, we obtain the first equality of \eqref{eq::5.8}. By H\"older continuity of $f$, the second line of \eqref{eq::5.8} holds. Lemma \ref{lem:5.8} shows the last equality of \eqref{eq::5.8}. 
 \begin{equation}\label{eq::5.8}
\begin{split}
\displaystyle\int_{\cG\mathcal{T}} fh_{x,\lambda,\eta}^\pm dm_\lambda&\overset{\substack{\text{Lem}\\\ref{decomposition}}}=\displaystyle\int_{w\in\cG_x^\pm\mathcal{T}}\int_{g\in W^{\pm}_{x,\eta}(w)}f(g)h^\pm_{x,\lambda,\eta}(g)d\nu_{W^{0\pm}(w)}^\lambda(g)\mu_{\cG^{\pm}_x\mathcal{T}}^\lambda(w)\\
&\asymp_{(1+\epsilon)}\displaystyle\int_{w\in\cG_x^\pm\mathcal{T}}f(w)\int_{g\in W^{\pm}_{x,\eta}(w)}h^\pm_{x,\lambda,\eta}(g)d\nu_{W^{0\pm}(w)}^\lambda(g)\mu_{\cG^{\pm}_x\mathcal{T}}^\lambda(w)\\
&\overset{\substack{\text{Lem}\\ \ref{lem:5.8}}}=\displaystyle\int_{\cG_x^\pm \mathcal{T}}fd\mu_{\cG_x^\pm\mathcal{T}}^\lambda.
\end{split}
\end{equation}
Lemma \ref{lem:5.10} follows from \eqref{eq::5.8}.
\end{proof}
Let $(\mathcal{G}\mathcal{T})_0$ be a fundametal domain of $\Gamma$ in $\mathcal{G}\mathcal{T}$.
\begin{lem}\label{H} For any $\Gamma$-invariant bounded non-negative H\"older continuous function $f$, 
\begin{equation}
\int_{(\mathcal{G}\mathcal{T})_0}f\widetilde{H}_{x,\lambda,\eta'}^\pm dm_\lambda=\int_{\cG\mathcal{T}} fh_{x,\lambda,\eta'}^\pm dm_\lambda
\end{equation}
\end{lem}
\begin{proof}The second line follow from \eqref{eq:5.5}. 
\begin{equation}
\begin{split}
\nonumber\int_{(\mathcal{G}\mathcal{T})_0}f\widetilde{H}_{x,\lambda,\eta'}^\pm dm_\lambda&=\int_{(\mathcal{G}\mathcal{T})_0} f(g)\sum_{\g\in \G}h_{\gamma x,\lambda,\eta'}^\pm(g) dm_\lambda(g)\\
&\overset{\eqref{eq:5.5}}=\int_{(\mathcal{G}\mathcal{T})_0} f(g)\sum_{\g\in \G}h_{ x,\lambda,\eta'}^\pm(\gamma^{-1}g) dm_\lambda(g)\\
&=\sum_{\g\in \G}\int_{(\mathcal{G}\mathcal{T})_0} f(\gamma^{-1} g)h_{ x,\lambda,\eta'}^\pm(\gamma^{-1}g) dm_\lambda(g)=\int_{\cG\mathcal{T}} fh_{x,\lambda,\eta'}^\pm dm_\lambda.
\end{split}
\end{equation}
The above equation completes the proof.
\end{proof}
\subsection{Application of rapid mixing}\label{subsec:5.4}
Let $\cG\mathcal{T}_{x,y}$ be the set of geodesic lines $g$ satisfying $\pi (g)= x$ and $\pi(\phi_{d(x,y)}g)=y$. Denote $\cG\mathcal{T}_{x,y,\eta}:=\bigcup_{s\in (-\eta,\eta)}\phi_s \cG\mathcal{T}_{x,y}.$ Define a function $a_{\eta}$ on $\mathbb{R}$ as follows:
\begin{displaymath}
a_{\eta}(t):=\begin{cases} \frac{1}{6} (2\eta-|t|)^3 & \text{if } \eta\leq|t|\leq 2\eta\\
\frac{1}{2}|t|^3-\eta t^2+\frac{2}{3}\eta^3 & \text{if } |t|\leq \eta\\ 
0 & \text {otherwise}.
\end{cases}
\end{displaymath} 
The function $a_{\eta}(t)$ satisfies that $\int_{-2\eta}^{2\eta}a_{\eta}dt=\eta^4.$ Then we have the following lemma.
\begin{lem}\label{lem:5.11} Let $f$ be a $\G$-invariant non-negative H\"older continuous function on $\cG^+\mathcal{T}$. For any $\epsilon>0$, there exist constants $R_1$, $\eta_1$ and $\delta_1$ such that for any $\lambda\in[\lambda_0-\delta_1,\lambda_0]$, $r\geq R_1$, $x,y$ with $|d(x,y)-r|\leq2\eta\leq 2\eta_1$ and $g_x^y\in \cG\mathcal{T}_{x,y}$,  
$$\int_{\cG\mathcal{T}} fh_{x,\lambda,\eta}^+ h_{y,\lambda,\eta}^-\circ \phi_{r}dm_\lambda  \asymp_{(1+\epsilon)^3}\frac{a_{\eta}(r-d(x,y))}{\eta^4}G_\lambda^2(x,y)e^{-P_{\lambda}d(x,y)}f(g_x^y).$$
\end{lem}
\begin{proof}By H\"older continuity of $f$ and Lemma \ref{lem:2.1}, there exist $\eta_1$ and $R$ such that for any $x,y$ with $d(x,y)>R$ and $g,g'\in \cG\mathcal{T}_{x,y,\eta_1}$, $f(g)\asymp_{1+\epsilon}f(g').$ Note that $|r-d(x,y)|\leq 2\eta$ if and only if there exists $g\in \cG\mathcal{T}_{x,y,\eta}$ such that $$h_{x,\lambda,\eta}^+(g) h_{y,\lambda,\eta}^-\circ \phi_{r}(g)\neq 0$$ 
(see Figure \ref{figure6}).
For any $g_x^y\in \cG\mathcal{T}_{x,y}$ and $\eta\in (0,\eta_1)$, H\"older continuity of $f$ implies that
\begin{eqnarray}
\nonumber\int_{\cG\mathcal{T}} fh_{x,\lambda,\eta}^+ h_{y,\lambda,\eta}^-\circ \phi_{r}dm_\lambda&=&\int_{\cG\mathcal{T}_{x,y,\eta}} fh_{x,\lambda,\eta}^+ h_{y,\lambda,\eta}^-\circ \phi_{r}dm_\lambda\\ \nonumber &\asymp_{1+\epsilon}&f(g_x^y)\int_{\cG\mathcal{T}_{x,y,\eta}} h_{x,\lambda,\eta}^+ h_{y,\lambda,\eta}^-\circ \phi_{r}dm_\lambda.\nonumber
\end{eqnarray}
 Hence the remaining part of the proof is to show that
$$\int_{\cG\mathcal{T}_{x,y,\eta}} h_{x,\lambda,\eta}^+ h_{y,\lambda,\eta}^-\circ \phi_{r}dm_\lambda  \asymp_{(1+\epsilon)^2}\frac{a_{\eta}(r-d(x,y))}{\eta^4}G_\lambda^2(x,y)e^{-P_{\lambda}d(x,y)}.$$

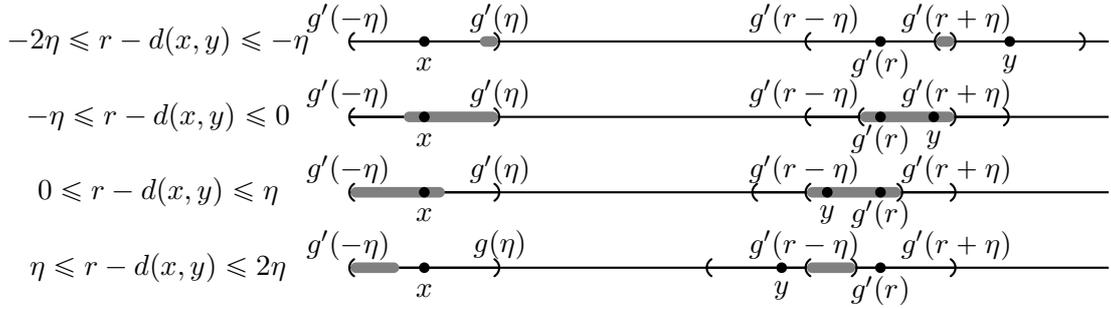
\begin{figure}[h]
\begin{center}
\begin{tikzpicture}[scale=1]
\draw [thick](-4,4)--(6,4);\draw [thick](-4,3)--(6,3);\draw [thick](-4,2)--(6,2);\draw [thick](-4,1)--(6,1);
\draw [thick,(-)](-4,4)--(-2,4);\draw [thick,(-)](-4,3)--(-2,3);\draw [thick,(-)](-4,2)--(-2,2);\draw [thick,(-)](-4,1)--(-2,1);
\draw [thick,(-)](2,4)--(4,4);\draw [thick,(-)](2,3)--(4,3);\draw [thick,(-)](2,2)--(4,2);\draw [thick,(-)](2,1)--(4,1);
\draw[thick,(-)](3.7,4)--(5.7,4);\draw[thick,(-)](2.7,3)--(4.7,3);\draw[thick,(-)](1.3,2)--(3.3,2);\draw[thick,(-)](0.7,1)--(2.7,1);
\draw[line cap=round,line width=4pt, color=gray](-2.2,4)--(-2.1,4);\draw[line cap=round,line width=4pt, color=gray](-3.9,2)--(-2.8,2);\draw[line cap=round,line width=4pt, color=gray](-3.2,3)--(-2.1,3);\draw[line cap=round,line width=4pt, color=gray](-3.4,1)--(-3.9,1);
\draw[line cap=round,line width=4pt, color=gray](3.8,4)--(3.9,4);\draw[line cap=round,line width=4pt, color=gray](2.8,3)--(3.9,3);\draw[line cap=round,line width=4pt, color=gray](2.1,2)--(3.2,2);\draw[line cap=round,line width=4pt, color=gray](2.1,1)--(2.6,1);
\fill (-3,4)  circle (2pt);\fill (-3,3)    circle (2pt);\fill (-3,2)    circle (2pt);\fill (-3,1)    circle (2pt);
\fill (3,4)    circle (2pt);\fill (3,3)    circle (2pt);\fill (3,2)    circle (2pt);\fill (3,1)    circle (2pt);
\fill (4.7,4)    circle (2pt);\fill (3.7,3)    circle (2pt);\fill (2.3,2)    circle (2pt);\fill (1.7,1)    circle (2pt);
\node at (-6.5,4) {$-2\eta \leq r-d(x,y) \leq -\eta$};\node at (-3,3.7) {$x$};\node at (3,3.7) {$g'(r)$};\node at (4.7,3.7) {$y$};\node at (-4,4.3) {$g'(-\eta)$};\node at (-2,4.3) {$g'(\eta)$};\node at (2,4.3) {$g'(r-\eta)$};\node at (4,4.3) {$g'(r+\eta)$};
\node at (-6.5,3) {$-\eta \leq r-d(x,y) \leq 0$};\node at (-3,2.7) {$x$};\node at (3,2.7) {$g'(r)$};\node at (3.7,2.7) {$y$};\node at (-4,3.3) {$g'(-\eta)$};\node at (-2,3.3) {$g'(\eta)$};\node at (2,3.3) {$g'(r-\eta)$};\node at (4,3.3) {$g'(r+\eta)$};
\node at (-6.5,2) {$0 \leq r-d(x,y) \leq \eta$};\node at (-3,1.7) {$x$};\node at (3,1.7) {$g'(r)$};\node at (2.3,1.7) {$y$};\node at (-4,2.3) {$g'(-\eta)$};\node at (-2,2.3) {$g'(\eta)$};\node at (2,2.3) {$g'(r-\eta)$};\node at (4,2.3) {$g'(r+\eta)$};
\node at (-6.5,1) {$\eta \leq r-d(x,y) \leq 2\eta$};\node at (-3,0.7) {$x$};\node at (3,0.7) {$g'(r)$};\node at (1.7,0.7) {$y$};\node at (-4,1.3) {$g'(-\eta)$};\node at (-2,1.3) {$g(\eta)$};\node at (2,1.3) {$g'(r-\eta)$};\node at (4,1.3) {$g'(r+\eta)$};
\end{tikzpicture}
\end{center}
\caption{The intervals satisfying $h_{x,\lambda,\eta}^+(\phi_s g') h_{y,\lambda,\eta}^-( \phi_{r+s}g')\neq 0$}\label{figure6}
\end{figure}
 
When $0\leq r-d(x,y) \leq 2\eta$ ($-2\eta\leq r-d(x,y)<0$, resp.), $g'\in\cG\mathcal{T}_{x,y}$ and $$-\eta\leq s\leq\eta-r+d(x,y) \text{(}-\eta-r+d(x,y)\leq s\leq \eta\text{, resp.)}$$ if and only if $h_{x,\lambda,\eta}^+(\phi_s g') h_{y,\lambda,\eta}^-( \phi_{r+s}g')\neq 0$ (see Figure \ref{figure6}). 

For any $y$ with $0\leq r-d(x,y)\leq 2\eta,$
\begin{equation}\label{eq:5.11}
\int_{-\eta}^{\eta-r+d(x,y)}(\eta-|s|)(\eta-|s+r-d(x,y)|)ds=a_{\eta}(r-d(x,y))
\end{equation}
and for any $y$ with $-2\eta\leq r-d(x,y)\leq 0,$
\begin{equation}\label{eq:5.12}
\int_{-\eta-r+d(x,y)}^{\eta}(\eta-|s|)(\eta-|s+r-d(x,y)|)ds=a_{\eta}(r-d(x,y)).
\end{equation}

For sufficiently small $\eta$,  we have $$\cG\cT_{x,y,\eta}=\{g\in\mathcal{G}\mathcal{T}:(g_-,g_+,t)\in\mathcal{O}_y(x)\times\mathcal{O}_x(y)\times(-\eta,\eta)\}$$ (see Figure \ref{figure6}). Denote by $w_{\xi}$ in $\cG_x^+\cT$ ($w'_{\zeta}$ in $\cG_y^-\cT$, resp.) the geodesic ray satisfying $w_{\xi+}=\xi$ ($w'_{\zeta-}=\zeta$, resp.). Since $\eta-d(y,\pi (\phi_{r+s} g'))=\eta-|s+r-d(x,y)|$ for any $g'\in \cG\mathcal{T}_{x,y}$ and $s\in(-\eta,\eta)$ with $s+r-d(x,y)\in(-\eta,\eta)$, the equations \eqref{eq:5.11} and \eqref{eq:5.12} show the first equality of \eqref{eq:221113}. The second equality of \eqref{eq:221113} follows from Lemma \ref{lem:5.9}. 
\begin{equation}\label{eq:221113}
\begin{split}
&\int_{\cG\mathcal{T}}h_{x,\lambda,\eta}^+ h_{y,\lambda,\eta}^-\circ \phi_{r}dm_\lambda=\int_{\cG\mathcal{T}_{x,y,\eta}}h_{x,\lambda,\eta}^+ h_{y,\lambda,\eta}^-\circ \phi_{r}dm_\lambda\\ 
&\overset{\eqref{eq:5.11}+\eqref{eq:5.12}}=\int_{\mathcal{O}_x(y)}\int_{\mathcal{O}_y(x)}\frac{4a_{\eta}(r-d(x,y))\theta_x^\lambda(\zeta,\xi)^2}{\eta^2\nu_{W^{0+}(w_{\xi})}^\lambda(W^{0+}_{x,\eta}(w_{\xi}))\nu_{W^{0-}(w_{\zeta}')}^\lambda(W^{0-}_{y,\eta}(w_{\zeta}'))}d\mu_{x,\lambda}(\zeta)d\mu_{x,\lambda}(\xi)\\
&\overset{\substack{\text{Lem}\\\ref{lem:5.9}}}=\int_{\mathcal{O}_x(y)}\int_{\mathcal{O}_y(x)}\frac{a_{\eta}(r-d(x,y))\theta_x^\lambda(\zeta,\xi)^2}{\eta^4\int_{\mathcal{O}_y(x)}\theta^\lambda_x(\zeta',\xi)^2d\mu_{x,\lambda}(\zeta')\int_{\mathcal{O}_x(y)}\theta^\lambda_y(\zeta,\xi')^2d\mu_{y,\lambda}(\xi')}d\mu_{x,\lambda}(\zeta)d\mu_{x,\lambda}(\xi).
\end{split}
\end{equation}
By Corollary \ref{coro:2.7}, there exist $R'$ and $\delta_1$ such that for any  $x,y\in \mathcal{T}$ with $d(x,y)\geq R'$, $\lambda\in [\lambda_0-\delta_1,\lambda_0]$, $\zeta\in \mathcal{O}_y(x)$ and $\xi \in \mathcal{O}_x(y)$,
$$G_\lambda(x,y)k_\lambda(x,y,\xi)\theta_y^\lambda(\zeta,\xi)\asymp_{1+\epsilon}1.$$
The equation \eqref{eq:3.4} shows that for any $x,y\in \mathcal{T}$ with $d(x,y)\geq R'$, 
\begin{equation}\label{032501}
\begin{split}
\int_{\mathcal{O}_x(y)}\theta^\lambda_y(\zeta,\xi')^2d\mu_{y,\lambda}(\xi')&\asymp_{(1+\epsilon)^2}\int_{\mathcal{O}_x(y)}\frac{k_\lambda(y,x,\xi)^2}{G_\lambda(x,y)^2}d\mu_{y,\lambda}(\xi)\overset{\eqref{eq:3.4}}=\frac{\mu_{x,\lambda}(\mathcal{O}_x(y))e^{P_\lambda d(x,y)}}{G_\lambda(x,y)^2}.
\end{split}
\end{equation}
Using \eqref{032501}, we have
\begin{equation}
\begin{split}
&\int_{\cG\mathcal{T}}h_{x,\lambda,\eta}^+ h_{y,\lambda,\eta}^-\circ \phi_{r}dm_\lambda\\
&\overset{\eqref{032501}}\asymp_{(1+\epsilon)^2}\int_{\mathcal{O}_x(y)}\frac{a_{\eta}(r-d(x,y))G_\lambda^2(x,y)e^{-P_{\lambda}d(x,y)}\int_{\mathcal{O}_y(x)}\theta_x^\lambda(\zeta,\xi)^2d\mu_{x,\lambda}(\zeta)}{\eta^4\mu_{x,\lambda}(\mathcal{O}_x(y))\int_{\mathcal{O}_y(x)}\theta^\lambda_x(\zeta',\xi)^2d\mu_{x,\lambda}(\zeta')}d\mu_{x,\lambda}(\xi)\\
&=\int_{\mathcal{O}_x(y)}\frac{a_{\eta}(r-d(x,y))G_\lambda^2(x,y)e^{-P_{\lambda}d(x,y)}}{\eta^4\mu_{x,\lambda}(\mathcal{O}_x(y))}d\mu_{x,\lambda}(\xi)=\frac{a_{\eta}(r-d(x,y))G_\lambda^2(x,y)e^{-P_{\lambda}d(x,y)}}{\eta^4}.
\end{split}
\end{equation} 
Let $R_1:=\max\{R,R'\}.$ Then Lemma \ref{lem:5.11} follows.
\end{proof}
\begin{coro}\label{coro:5.13} Let $f_1$ be a $\G$-invariant non-negative H\"older continuous function on $\cG^+\mathcal{T}$ and let $f_2$ be a $\G$-invariant non-negtative H\"older continuous function on $\cG\mathcal{T}$. For any $\epsilon>0$, there exist constants $R_1'$, $\eta_1'$ and $\delta_1'$ saisfyint the following property. For any $x,y$ with $d(x,y)> 2R_1'$, $\eta\in(0,\eta_1')$, $g_x^{y}\in \cG\mathcal{T}_{x,y,\eta_1'}$, and $\lambda\in[\lambda_0-\delta'_1,\lambda_0]$,
\begin{equation}
\begin{split}
\nonumber&\int_{\cG\mathcal{T}}f_1h_{x,\lambda,\eta}^+ \biggl(\int_{R_1'}^{d(x,y)-R_1'}f_2\circ\phi_sds\biggr) h_{y,\lambda,\eta}^-\circ \phi_{r}dm_\lambda \\ 
&\asymp_{(1+\epsilon)^4}\frac{a_{\eta}(r-d(x,y))}{\eta^4}{G_\lambda^2(x,y)e^{-P_{\lambda}d(x,y)}f_1(g_x^{y})}\int_{R_1'}^{d(x,y)-R_1'}f_2\circ\phi_s(g_x^y)ds\nonumber
\end{split}
\end{equation}
where $ds$ is Lebesgue measure on $\mathbb{R}$.
\end{coro}
\begin{proof}There exists $R_2$ and $\eta''$ such that for any $g,g'\in \cG\mathcal{T}$ with $\min\{|t|:g(t)\neq g'(t)\}>R_2$ and $s\in (-\eta'',\eta'')$, $f_2(g)\asymp_{1+\epsilon}f_2(\phi_s g')$. Put $R_1'=\max\{R_1,R_2\}$, $\eta_1'=\min\{\eta'',\eta_1\}$ and $\delta_1'=\delta_1$ where $R_1$, $\eta_1$ and $\delta_1$ are constants in Lemma \ref{lem:5.11}. Since for any $g_1,g_2\in \cG\mathcal{T}_{x,y,\eta'},$
$$\int_{R_1'}^{d(x,y)-R_1'}f_2\circ\phi_s(g_1)ds\asymp_{1+\epsilon}\int_{R_1'}^{d(x,y)-R_1'}f_2\circ\phi_s(g_2)ds,$$
following the proof of Lemma \ref{lem:5.11}, we have Corollary \ref{coro:5.13}.
\end{proof}
Fix a connected fundamental domain $T_0$ of $\G$ in $\mathcal{T}$ and a fundamental domain $(\cG\mathcal{T})_0$ of $\G$ in $\mathcal{T}$ which satisfies $\pi (g)\in T_0$ for all $g\in (\cG\mathcal{T})_0.$ Assume that $\int_{T_0}\mu_{y,\lambda}(\partial\mathcal{T})d\mu(y)=1$ for any $\lambda\in [0,\lambda_0]$. In the rest part of this article, if $y$ and $y'$ are in the same open edge, the geodesic $g_{x}^y\in \mathcal{G}\mathcal{T}_{x,y}$ arising in the integrals will be chosen to be the same as $g_x^{y'}\in \mathcal{G}\mathcal{T}_{x,y'}$. Since  Lemma \ref{lem:5.11} and Corollary \ref{coro:5.13} hold for any $g_x^y\in \mathcal{G}\mathcal{T}_{x,y}$ satisfying asssumptions, the choice of $g_x^y$ makes sense.
\begin{thm}\label{thm:5.14} Let $f$ be a $\G$-invariant bounded non-negative H\"older function on $\cG^+\mathcal{T}$. Suppose that for any ${\lambda\in [0,\lambda_0]}$, $\int_{\cG_x^+\mathcal{T}}fd\mu_{\cG_x^+\mathcal{T}}^\lambda>0.$ For any $\epsilon>0$, there exist constants $R_0$ and $\delta_0$ such that for all $x\in \mathcal{T}$, $\lambda\in [\lambda_0-\delta_0,\lambda]$ and $r\geq R_0$,
\begin{equation}\label{eq:5.13.0} 
-P_\lambda\int_{B(x,r)^c} G_\lambda^2(x,y)f(g_x^y)d\mu(y)\asymp_{(1+\epsilon)^{9}}\mathcal{C}e^{P_{\lambda_0} r}{\int_{\cG_x^+\mathcal{T}}f(w)d\mu_{\cG_x^+\mathcal{T}}^{\lambda_0}(w)},
\end{equation}
where $g_x^y\in \cG\mathcal{T}_{x,y}$ and $\mathcal{C}:=\frac{1}{m_{\lambda_0}((\cG\mathcal{T})_0)}.$ 
\end{thm}
\begin{proof}Let $R_1$, $\eta_1$ and $\delta_1$ be the constants in Lemma \ref{lem:5.11}. Fix $\eta<\min\{\eta_1,\eta'\}$ satisfying $e^{-2P_\lambda \eta}\leq 1+\epsilon$ for any $\lambda\in [0,\lambda_0]$, where $\eta'$ is the constant in Lemma \ref{lem:5.10}. We claim that for any $r\geq R_1$, $\lambda\in [\lambda_0-\delta_1,\lambda_0]$ and $x\in \mathcal{T}$,
 \begin{equation}\label{eq:5.13}
 \begin{split}
&-P_\lambda(1+\epsilon)^{-4}\int_{B(x,r+2\eta)^c} G_\lambda^2(x,y)f(g_x^y)d\mu(y)\\
&\leq (*):=-P_\lambda\int_r^\infty e^{P_\lambda s}\int_{\mathcal{T}}\int_{\cG\mathcal{T}} fh_{x,\lambda,\eta}^+ h_{y,\lambda,\eta}^-\circ \phi_{s}dm_\lambda d\mu(y)ds\\
&\leq -P_\lambda(1+\epsilon)^4\int_{B(x,r-2\eta)^c} G_\lambda^2(x,y)f(g_x^y)d\mu(y).
\end{split}
\end{equation}
Assume that the claim holds. The second line of \eqref{eq:5.13} is replaced by the first line of \eqref{eq:5.13.1}. The second line of \eqref{eq:5.13.1} follows from the definition of $\widetilde{H}_{y',\lambda,\eta}^-$ (see \eqref{eq:5.6}). Since $f$ and $\widetilde{H}_{y',\lambda,\eta}^-$ are $\G$-invariant functions and $m_\lambda$ is a $\G$-invariant measure, the third line of \eqref{eq:5.13.1} holds. Using \eqref{eq:5.5} and Lemma \ref{H}, we have the last line of \eqref{eq:5.13.1}. 
\begin{equation}\label{eq:5.13.1}
\begin{split}
(*)&=\int_r^\infty -P_\lambda e^{P_\lambda s}\int_{T_0}\sum_{\g \in \G}\int_{\cG\mathcal{T}} fh_{x,\lambda,\eta}^+(g) h_{\g y,\lambda,\eta}^-\circ \phi_{s}(g)dm_\lambda(g)d\mu(y)ds\\ 
&=\int_r^\infty -P_\lambda e^{P_\lambda s}\int_{T_0}\int_{\cG\mathcal{T}} fh_{x,\lambda,\eta}^+(g) \widetilde{H}_{y,\lambda,\eta}^-\circ \phi_{s}(g)dm_\lambda(g)d\mu(y)ds\\
&= \int_r^\infty -P_\lambda e^{P_\lambda s}\int_{T_0}\sum_{\g\in \G}\int_{(\cG\mathcal{T})_0} fh_{x,\lambda,\eta}^+(\g g) \widetilde{H}_{y,\lambda,\eta}^-\circ \phi_{s}(\g g)dm_\lambda(g)d\mu(y)ds\\
&=\int_r^\infty -P_\lambda e^{P_\lambda s}\int_{T_0}\int_{(\cG\mathcal{T})_0} f\widetilde{H}_{x,\lambda,\eta}^+ \widetilde{H}_{y,\lambda,\eta}^-\circ \phi_{s}dm_\lambda d\mu ds.
\end{split}
\end{equation}

 The constant $T_1$ in Corollary \ref{coro:5.5} depends on $\|f\widetilde{H}_{x,\lambda,\eta}^+\|_\alpha$, $\|\widetilde{H}_{y,\lambda,\eta}^-\|_\alpha,$ $\inf \int f\widetilde{H}_{x,\lambda,\eta}^+ d\overline{m}_\lambda,$ and $\inf \int \widetilde{H}_{y,\lambda,\eta}^- d\overline{m}_\lambda$.  By Lemma \ref{lem:5.10} and Lemma \ref{H}, $\int f\widetilde{H}_{x,\lambda,\eta}^+ d\overline{m}_\lambda\asymp_{1+\epsilon}\int_{\cG_x^+ \mathcal{T}}f(w)d\mu_{\cG_x^+\mathcal{T}}^{\lambda}(w)$ and $\int \widetilde{H}_{x,\lambda,\eta}^- d\overline{m}_\lambda=\int_{\cG_y^- \mathcal{T}}d\mu_{\cG_y^-\mathcal{T}}^{\lambda}(w).$ Since the map $\lambda\mapsto \mu_{x,\lambda}$ is continuous in $[0,\lambda_0]$ for any $x\in \mathcal{T}$ and the action of $\Gamma$ is cocompact, it is possible to choose $T_1$, $a_1$ such that for any $r>T_1$ and $\lambda\in [\lambda_0-a_1,\lambda_0]$, the second line of \eqref{eq:5.14} holds. The third line follows from Lemma \ref{H}. Since we assumed that $\int_{T_0}\mu_{y,\lambda}(\partial \mathcal{T})d\mu(y)=1$, Lemma \ref{lem:5.10} is used to obtain the fourth line of \eqref{eq:5.14}. 
\begin{equation}\label{eq:5.14}
\begin{split}
&\int_r^\infty -P_\lambda e^{P_\lambda s}\int_{T_0}\int_{(\cG\mathcal{T})_0} f\widetilde{H}_{x,\lambda,\eta}^+ \widetilde{H}_{y,\lambda,\eta}^-\circ \phi_sdm_\lambda d\mu ds\\
&\overset{\substack{\text{Coro}\\\ref{coro:5.5}}}\asymp_{1+\epsilon} \frac{1}{m_\lambda((\cG\mathcal{T})_0)}\int_r^\infty -P_\lambda e^{P_\lambda s}ds \int_{T_0}\int_{(\cG\mathcal{T})_0} f\widetilde{H}_{x,\lambda,\eta}^+dm_\lambda \int_{(\cG\mathcal{T})_0}\widetilde{H}_{y,\lambda,\eta}^-dm_\lambda d\mu(y) \\ 
&\overset{\substack{\text{Lem}\\ \ref{H}}}=\frac{e^{P_\lambda r}}{m_\lambda((\cG\mathcal{T})_0)}\int_{\cG\mathcal{T}} f(g)h_{x,\lambda,\eta}^+(g)dm_\lambda(g)\int_{T_0}\mu_{y,\lambda}(\partial \mathcal{T})d\mu(y) \\
&\overset{\substack{\text{Lem}\\\ref{lem:5.10}}}\asymp_{(1+\epsilon)^{2}}\frac{e^{P_{\lambda} r}}{m_{\lambda}((\cG\mathcal{T})_0)}\int_{\cG_x^+ \mathcal{T}}f(w)d\mu_{\cG_x^+\mathcal{T}}^{\lambda}(w)\\
&\asymp_{1+\epsilon}\frac{e^{P_{\lambda_0} r}}{m_{\lambda_0}((\cG\mathcal{T})_0)}\int_{\cG_x^+ \mathcal{T}}f(w)d\mu_{\cG_x^+\mathcal{T}}^{\lambda_0}(w).
\end{split}
\end{equation}
By Lemma \ref{2203311}, Lemma \ref{2203312} and Corollary \ref{2203313}, there exists $\delta'$ such that the last line of \eqref{eq:5.14} holds  for any $\lambda\in [\lambda_0-\delta',\lambda_0]$.
 By \eqref{eq:5.13}, \eqref{eq:5.13.1}, \eqref{eq:5.14} and the choice of $\delta_0$ and $\eta$, 
 \begin{equation}\label{eq:2201151}
 \frac{e^{P_{\lambda_0} (r+2\eta)}}{m_{\lambda_0}((\cG\mathcal{T})_0)}\int_{\cG_x^+ \mathcal{T}}f(w)d\mu_{\cG_x^+\mathcal{T}}^{\lambda_0}(w)\leq -P_\lambda(1+\epsilon)^{8}\int_{B(x,r)^c} G_\lambda^2(x,y)f(g_x^y)d\mu(y)
  \end{equation}
  and 
   \begin{equation}\label{eq:2201152}
   -P_\lambda(1+\epsilon)^{-8}\int_{B(x,r)^c} G_\lambda^2(x,y)f(g_x^y)d\mu(y)\leq \frac{e^{P_{\lambda_0} (r-2\eta)}}{m_{\lambda_0}((\cG\mathcal{T})_0)}\int_{\cG_x^+ \mathcal{T}}f(w)d\mu_{\cG_x^+\mathcal{T}}^{\lambda_0}(w). 
  \end{equation}
The equation \eqref{eq:5.13.0} follows from \eqref{eq:2201151}, \eqref{eq:2201152} and the choice of $\eta$. Let $R_0=\max\{R_1,T_1\}$ and $\delta_0=\min\{\delta_1,a_1,\delta'\}$. Then we obtain Theorem \ref{thm:5.14}.
 
 The remaining part of the proof is to show the claim. Since $|s-d(x,y)|<2\eta$ if and only if $h_{x,\lambda,\eta}^+(g) h_{y,\lambda,\eta}^-\circ \phi_{s}(g)\neq 0$ for some $g\in \mathcal{GT}_{x,y,\eta},$ Lemma \ref{lem:5.11} shows that 
\begin{equation}\label{eq:20222}
\begin{split}
& -P_\lambda \int_r^\infty e^{P_\lambda s}\int_{\mathcal{T}}\int_{\cG\mathcal{T}} fh_{x,\lambda,\eta}^+ h_{y,\lambda,\eta}^-\circ \phi_{s}dm_\lambda d\mu(y)ds\\
&\asymp_{(1+\epsilon)^3} -P_\lambda \int_{r}^\infty \int_{B(x,r-2\eta)^c}\frac{a_{\eta}(s-d(x,y))}{\eta^4}G_\lambda^2(x,y)e^{P_{\lambda}(s-d(x,y))}f(g_x^y)d\mu(y)ds\\
\end{split}
\end{equation}
 Since $$\int_r^\infty \frac{a_\eta(s-d(x,y))}{\eta^4}ds=\int_r^{d(x,y)+2\eta} \frac{a_\eta(s-d(x,y))}{\eta^4}ds\leq 1$$ for any $y$ with $d(x,y)+2\eta\geq r$, we have the first inequality of \eqref{eq:20221}. By the choice of $\eta$, $e^{P_\lambda(s-d(x,y))}<1+\epsilon$ for any $x,y$ with $|d(x,y)-s|<2\eta$. Thus we have the second inequality of \eqref{eq:20221}.
 \begin{equation}\label{eq:20221}
\begin{split}
&\int_{r}^\infty\int_{B(x,r-2\eta)^c} \frac{a_{\eta}(s-d(x,y))}{\eta^4}G_\lambda^2(x,y)e^{P_{\lambda}(s-d(x,y))}f(g_x^y)d\mu(y)ds\\
&\leq \int_{B(x,r-2\eta)^c}\int_{d(x,y)-2\eta}^{\infty} \frac{a_{\eta}(s-d(x,y))}{\eta^4}G_\lambda^2(x,y)e^{P_{\lambda}(s-d(x,y))}f(g_x^y)dsd\mu(y)\\
& \leq (1+\epsilon)\int_{B(x,r-2\eta)^c} G_\lambda^2(x,y)f(g_x^y)d\mu(y).
\end{split}
\end{equation}
For any $y$ with $d(x,y)-2\eta>r$, $\int_r^\infty\frac{a_\eta(s-d(x,y))}{\eta^4}ds=1.$ Similar to \eqref{eq:20221}, we have the following inequality
 \begin{equation}\label{eq:20223}
\begin{split}
&\int_{r}^\infty\int_{B(x,r-2\eta)^c} \frac{a_{\eta}(s-d(x,y))}{\eta^4}G_\lambda^2(x,y)e^{P_{\lambda}(s-d(x,y))}f(g_x^y)ds\\
&\geq \int_{B(x,r+2\eta)^c}\int_{d(x,y)-2\eta}^{\infty} \frac{a_{\eta}(s-d(x,y))}{\eta^4}G_\lambda^2(x,y)e^{P_{\lambda}(s-d(x,y))}f(g_x^y)dsd\mu(y)\\
&\geq (1+\epsilon)^{-1}\int_{B(x,r+2\eta)^c} G_\lambda^2(x,y)f(g_x^y)d\mu(y).
\end{split}
\end{equation}
Combining \eqref{eq:20222}, \eqref{eq:20221} and \eqref{eq:20223}, we have \eqref{eq:5.13}.
\end{proof}
\begin{coro}\label{coro:5.16} Let $f_1$ be a $\G$-invariant non-negative H\"older continuous function on $\cG^+\mathcal{T}$ and let $f_2$ be a $\G$-invariant positive H\"older contnuous function on $\cG\mathcal{T}$. For any sufficiently small $\epsilon>0$, there exist constants $R_0'$ and $\delta_0'$ such that for all $R>R_0',$ $r$ with $(1+\epsilon)^{-1}<\frac{r-2R}{r}$ and $\lambda\in [\lambda_0-\delta_0',\lambda_0]$, 
\begin{equation}
\begin{split}
&\int_{B(x,r)^c}{P_\lambda^2}G_\lambda^2(x,y)f_1(g_{x}^y)\int_{R}^{d(x,y)-R}f_2\circ\phi_s(g_x^y)dsd\mu(y)\nonumber\\
\nonumber&\asymp_{(1+\epsilon)^{12}}\mathcal{C}^2(1-P_{\lambda_0} r)e^{P_{\lambda_0} r}\int_{\cG_x^+\mathcal{T}}f_1(w)d\mu_{\cG_x^+\mathcal{T}}^{\lambda_0}(w)\int_{\Gamma\backslash\cG\mathcal{T}}f_2d\overline{m}_{\lambda_0},\nonumber
\end{split}
\end{equation}
where $g_x^y\in \mathcal{GT}_{x,y}$ and $ds$ is Lebesgue measure on $\mathbb{R}$. 
\end{coro}
\begin{proof}Let $R_1'$, $\eta_1'$ and $\delta_1'$ be the constants in Corollary \ref{coro:5.13} and let $T_2$ and $a_2$ be the constants in Corollary \ref{coro:5.6}. Fix $\eta<\eta_1'$. Denote $R_0':=\max\{R_1', T_2\}$ and $\delta_0':=\min\{\delta_0, \delta_1', a_2\},$ where $\delta_0$ is the constant in the proof of Theorem \ref{thm:5.14}. Similar to the proof of Theorem \ref{thm:5.14}, applying Corollary \ref{coro:5.13}, for any $R>R_0',$ $r>2R$ and $\lambda\in [\lambda_0-\delta_0',\lambda_0]$, we have 
\begin{equation}\label{eq:2202061}
\begin{split}
&(1+\epsilon)^{-5}\int_{B(x,r+2\eta)^c}{P_\lambda^2}G_\lambda^2(x,y)f_1(g_{x}^y)\int_{R}^{d(x,y)-R}f_2\circ\phi_s(g_x^y)dsd\mu(y)\\
&\leq \int_r^\infty P_\lambda^2e^{P_\lambda t}\int_{T}\int_{\cG\mathcal{T}}f_1h_{x,\lambda,\eta}^+ \biggl(\int_{R}^{t-R}f_2\circ\phi_sds\biggr) h_{y,\lambda,\eta}^-\circ \phi_{r}dm_\lambda d\mu dt\\
&\leq(1+\epsilon)^{5}\int_{B(x,r-2\eta)^c}{P_\lambda^2}G_\lambda^2(x,y)f_1(g_{x}^y)\int_{R}^{d(x,y)-R}f_2\circ\phi_s(g_x^y)dsd\mu(y),
\end{split}
\end{equation}
where $dt$ is Lebesgue measure on $\mathbb{R}$.  Similar to the proof of Theorem \ref{thm:5.14}, Corollary \ref{coro:5.6} shows that for any $r$ with $(1+\epsilon)^{-1}<\frac{r-2R}{r},$ 
\begin{equation}\label{eq:2202062}
\begin{split}
& \int_r^\infty P_\lambda^2e^{P_\lambda t}\int_{\mathcal{T}}\int_{\cG\mathcal{T}}f_1h_{x,\lambda,\eta}^+ \biggl(\int_{R}^{t-R}f_2\circ\phi_sds\biggr) h_{y,\lambda,\eta}^-\circ \phi_{r}dm_\lambda d\mu dt\\
&\asymp_{1+\epsilon} \int_r^\infty \frac{tP_\lambda^2e^{P_\lambda t}}{t-2R}\int_{\mathcal{T}}\int_{\cG\mathcal{T}}f_1h_{x,\lambda,\eta}^+ \biggl(\int_{R}^{t-R}f_2\circ\phi_sds\biggr) h_{y,\lambda,\eta}^-\circ \phi_{r}dm_\lambda d\mu dt\\
&\overset{\substack{\text{Coro}\\\ref{coro:5.6}}}\asymp_{(1+\epsilon)^3}  \mathcal{C}^2\int_r^\infty {tP_\lambda^2e^{P_\lambda t}}\int_{\cG_x^+\mathcal{T}}f_1d\mu_{\cG_x^+\mathcal{T}}^{\lambda} \int_{\Gamma\backslash\cG\mathcal{T}}f_2d\overline{m}_{\lambda}dt\\
&\asymp_{(1+\epsilon)^2}  \mathcal{C}^2(1-P_{\lambda_0} r)e^{P_{\lambda_0} r}\int_{\cG_x^+\mathcal{T}}f_1d\mu_{\cG_x^+\mathcal{T}}^{\lambda_0} \int_{\Gamma\backslash\cG\mathcal{T}}f_2d\overline{m}_{\lambda_0}.
\end{split}
\end{equation}
Using \eqref{eq:2202061} and \eqref{eq:2202062}, Corollary \ref{coro:5.16} follows.
\end{proof}
\section{Proof of local limit theorem}\label{Sec:6}
In this section, using the formulae in the previous section and the fact that $\lambda\mapsto P_\lambda$ is continuous on $[0,\lambda_0]$, we find functions which are asymptotic to the derivatives of $G_\lambda(x,y)$ as $\lambda\rightarrow\lambda_0-$ . The local limit theorem will be proved by applying the asymptotic functions and Hardy-Littlewood Tauberian theorem. 

\subsection{The pressure of $f_{\lambda_0}$}
\begin{prop}\label{prop:6.1} $P_{\lambda_0}=0.$
\end{prop}
\begin{proof} Proposition 4.10 in \cite{HL} shows that $P_\lambda\leq0$ for any $\lambda\in[0,\lambda_0]$. As in the proof of Proposition 5.1 in \cite{LL}, we claim that if $P_{\lambda_0}<0$, the function $\lambda\mapsto G_\lambda(x,y)$ has a real analytic extension on $\varepsilon$-neighborhood of $\lambda_0$ for any distinct points $x,y\in \mathcal{T}$. This implies that for any $\lambda\in(\lambda_0,\lambda_0+\epsilon)$, $G_\lambda(x,y)$ is finite. This contradicts that the set of $\lambda\in \mathbb{R}$ for which $G_\lambda(x,y)$ converges for any $x,y$ is $(-\infty,\lambda_0]$ (\cite{HL} Corollary 3.12). 
Following the proof of Proposition 2.2 in \cite{LL}, we have
\begin{equation}\label{eq:2202041}
\frac{\partial^k}{\partial \lambda^k}G_\lambda(x,y)=k!\int_{\mathcal{T}\times\cdots\times \mathcal{T}}G_\lambda(x,x_1)\cdots G_\lambda(x_{k-1},x_k)G_\lambda(x_k,y)d\mu(x_1)\cdots d\mu(x_k).
\end{equation}
To prove the claim, we show that if $P_{\lambda_0}<0,$ there exists a constant $B_0$ such that for any $k\in \mathbb{N}$, distinct points $x,y\in \mathcal{T}$ and $\delta\in(0,-\frac{P_{\lambda_0}}{2})$,
\begin{equation}\label{eq:6.1}
\int_{\mathcal{T}\times\cdots \times \mathcal{T}}G_{\lambda_0}(x,x_1)G_{\lambda_0}(x_1,x_2)\cdots G_{\lambda_0}(x_k,y)e^{\delta d(x,x_k)}d\mu(x_1)\cdots d\mu(x_k)\leq G_{\lambda_0}(x,y)e^{\delta d(x,y)}B_0^k.
\end{equation}
If there exists $B_0$, then the analytic extension of $G_{\lambda}(x,y)$ is bounded above by $\frac{G_{\lambda_0}(x,y)e^{\delta d(x,y)}}{1-(\lambda-\lambda_0) B_0}$ for any $\lambda$ with $0<\lambda-\lambda_0<1/B_0.$  The inequality \eqref{eq:6.1} follows from the next lemma and induction on $k$.
\end{proof}
\begin{lem}\label{lem:6.2}
Suppose that $P_{\lambda_0}<0$. There exists $B_0$ such that for any distinct points $x,y \in\mathcal{T}$ and $\delta\in(0,-\frac{P_{\lambda_0}}{2})$,
$$\int_{\mathcal{T}}G_{\lambda_0}(x,z)G_{\lambda_0}(z,y)e^{\delta d(x,z)}d\mu(z)\leq B_0G_{\lambda_0}(x,y)e^{\delta d(x,y)}.$$
\end{lem}
\begin{proof}First, we verify that for any $x$,
$$\int_{B(x,r)^c}G_{\lambda_0}^2(x,z)e^{\delta d(x,z)}d\mu(z)\leq B$$
for some $B.$
Following the proof of Theorem \ref{thm:5.14}, for sufficiently large $R$, we have
\begin{equation}\label{eq:6.2}
\int_{B(x,R)^c} G_{\lambda_0}^2(x,y)e^{\delta d(x,y)}d\mu(y)\leq -\frac{\mathcal{C}(1+\epsilon)^{9}e^{(P_{\lambda_0}+\delta) R}}{(P_{\lambda_0}+\delta) r}\mu_{\cG_x^+\mathcal{T}}^{\lambda_0}(\cG_x^+\mathcal{T})=:A_1.
\end{equation}
Proposition 4.10 in \cite{HL} shows that there exists a constant $A_2$ for any $r>1$,
\begin{equation}\label{eq:2201153}
\sum_{y\in S(x,r)} G_{\lambda_0}^2(x,y)\leq A_2.
\end{equation}
The inequalities \eqref{eq:6.2} and \eqref {eq:2201153} show that there exists $B$ such that for any positive number $\delta<-\frac{P_{\lambda_0}}{2}$ and for any $x\in \mathcal{T},$
\begin{equation}\label{eq:6.3}
\int_{B(x,1)^c}G_{\lambda_0}^2(x,z)e^{\delta d(x,z)}d\mu(z)\leq A_1e^{\delta R}+\int_1^R\sum_{y\in S(x,r)}G_{\lambda_0}^2(x,y)e^{\delta r}dr\leq A_1e^{\delta R}+A_2e^{\delta R}R = B.\end{equation}
where $dr$ is Lebesgue measure on $\mathbb{R}$.
\begin{figure}[h]
\begin{center}
\begin{tikzpicture}[scale=1]
\node at (-2.7,-0.3) {$x$};\node at (2.7,-0.3) {$y$};
 \fill (-2.7,0) circle (2pt); \fill (2.7,0) circle (2pt);
 \draw[(-),thick] (-3.7,0)--(-1.7,0); \draw[-),thick] (-2,0)--(-2,0.3);\draw[-),thick] (-3,0)--(-3,-0.7);
 \draw[(-),thick] (3.7,0)--(1.7,0); \draw[-),thick] (2,0)--(2,0.3);\draw[-),thick] (3,0)--(3,-0.7);
 \draw[very thick,dashed](-4.5,-{sqrt(3)}/2)--(-4,0);\draw[very thick,dashed](-4,0)--(-3.7,0);\draw[very thick,dashed](-4.5,+{sqrt(3)}/2)--(-4,0);
  \draw[very thick,dashed](-4.5,-{sqrt(3)}/2)--(-5.2,-1.3);  \draw[very thick,dashed](-4.5,-{sqrt(3)}/2)--(-4.7,-1.6);
    \draw[very thick,dashed](-4.5,+{sqrt(3)}/2)--(-5.2,1.3);  \draw[very thick,dashed](-4.5,+{sqrt(3)}/2)--(-4.7,1.6);
 \draw[very thick,dashed](4.5,-{sqrt(3)}/2)--(4,0);\draw[very thick,dashed](4,0)--(3.7,0);\draw[very thick,dashed](4.5,+{sqrt(3)}/2)--(4,0);
   \draw[very thick,dashed](4.5,-{sqrt(3)}/2)--(5.2,-1.3);  \draw[very thick,dashed](4.5,-{sqrt(3)}/2)--(4.7,-1.6);
    \draw[very thick,dashed](4.5,+{sqrt(3)}/2)--(5.2,1.3);  \draw[very thick,dashed](4.5,+{sqrt(3)}/2)--(4.7,1.6);
 \draw[,very thick,dashed](-3,-0.7)--(-3,-1)--(-3.5,-{sqrt(3)}/2-1);\draw[very thick,dashed](-3,-1)--(-2.5,-{sqrt(3)}/2-1);
  \draw[-,very thick,dashed](-3.5,-{sqrt(3)}/2-1)--(-3.7,-2.4); \draw[,very thick,dashed](-3.5,-{sqrt(3)}/2-1)--(-3.3,-2.4);
    \draw[-,very thick,dashed](-2.5,-{sqrt(3)}/2-1)--(-2.7,-2.4); \draw[,very thick,dashed](-2.5,-{sqrt(3)}/2-1)--(-2.3,-2.4);
 \draw[,very thick,dashed](3,-0.7)--(3,-1)--(3.5,-{sqrt(3)}/2-1);\draw[very thick,dashed](3,-1)--(2.5,-{sqrt(3)}/2-1);
   \draw[,very thick,dashed](3.5,-{sqrt(3)}/2-1)--(3.7,-2.4); \draw[,very thick,dashed](3.5,-{sqrt(3)}/2-1)--(3.3,-2.4);
       \draw[,very thick,dashed](2.5,-{sqrt(3)}/2-1)--(2.7,-2.4); \draw[,very thick,dashed](2.5,-{sqrt(3)}/2-1)--(2.3,-2.4);
 \draw[,very thick,dashed](-2,0.3)--(-2,1)--(-2.5,+{sqrt(3)}/2+1);\draw[very thick,dashed](-2,1)--(-1.5,+{sqrt(3)}/2+1);
  \draw[very thick,dashed](-2.5,+{sqrt(3)}/2+1)--(-2.7,2.4);  \draw[very thick,dashed](-2.5,+{sqrt(3)}/2+1)--(-2.3,2.4);
  \draw[very thick,dashed](-1.5,+{sqrt(3)}/2+1)--(-1.7,2.4);  \draw[very thick,dashed](-1.5,+{sqrt(3)}/2+1)--(-1.3,2.4);
  \draw[,very thick,dashed](2,0.3)--(2,1)--(2.5,+{sqrt(3)}/2+1);\draw[very thick,dashed](2,1)--(1.5,+{sqrt(3)}/2+1); 
    \draw[very thick,dashed](2.5,+{sqrt(3)}/2+1)--(2.7,2.4);  \draw[very thick,dashed](2.5,+{sqrt(3)}/2+1)--(2.3,2.4);
  \draw[very thick,dashed](1.5,+{sqrt(3)}/2+1)--(1.7,2.4);  \draw[very thick,dashed](1.5,+{sqrt(3)}/2+1)--(1.3,2.4);
 \draw[,very thick,dotted] (-1.7,0)--(1.7,0); \draw[-),very thick,dotted](-1,0)--(-1,-1);\draw[-),very thick,dotted](0,0)--(0,1);\draw[-),very thick,dotted](1,0)--(1,-1);
  \draw[line width=2.2pt, color=gray](-1,-1)--(-1.5,-{sqrt(3)}/2-1); \draw[line width=2.2pt, color=gray](-1,-1)--(-0.5,-{sqrt(3)}/2-1);
      \draw[line width=2.2pt, color=gray](-0.5,-{sqrt(3)}/2-1)--(-0.7,-2.4);\draw[line width=2.2pt, color=gray](-0.5,-{sqrt(3)}/2-1)--(-0.3,-2.4);
    \draw[line width=2.2pt,color=gray](-1.5,-{sqrt(3)}/2-1)--(-1.7,-2.4);\draw[line width=2.2pt,color=gray](-1.5,-{sqrt(3)}/2-1)--(-1.3,-2.4); 
    \draw[line width=2.2pt,color=gray](1,-1)--(1.5,-{sqrt(3)}/2-1); \draw[line width=2.2pt, color=gray](1,-1)--(0.5,-{sqrt(3)}/2-1);
          \draw[line width=2.2pt,color=gray](0.5,-{sqrt(3)}/2-1)--(0.7,-2.4);\draw[line width=2.2pt,color=gray](0.5,-{sqrt(3)}/2-1)--(0.3,-2.4);
    \draw[line width=2.2pt,color=gray](1.5,-{sqrt(3)}/2-1)--(1.7,-2.4);\draw[line width=2.2pt,color=gray](1.5,-{sqrt(3)}/2-1)--(1.3,-2.4); 
     \draw[line width=2.2pt, color=gray](0,1)--(0.5,+{sqrt(3)}/2+1); \draw[line width=2.2pt,color=gray](0,1)--(-0.5,+{sqrt(3)}/2+1);
      \draw[line width=2.2pt, color=gray](0.5,+{sqrt(3)}/2+1)--(0.7,2.4); \draw[line width=2.2pt, color=gray](0.5,+{sqrt(3)}/2+1)--(0.3,2.4); 
       \draw[line width=2.2pt,color=gray](-0.5,+{sqrt(3)}/2+1)--(-0.7,2.4);    \draw[line width=2.2pt,color=gray](-0.5,+{sqrt(3)}/2+1)--(-0.3,2.4);
\node at (-5,-3) {${T}_1$}; \draw[thick] (-4.7,-3)--(-3.3,-3);
\node at (-3,-3) {${T}_2$};\draw[very thick,dashed](-2.7,-3)--(-1.3,-3);
\node at (-1,-3) {${T}_3$};\draw[very thick,dotted]  (-0.7,-3)--(0.7,-3);
\node at (1,-3) {${T}_4$};\draw[line width=2.2pt,color=gray] (1.3,-3)--(2.7,-3);
  \end{tikzpicture}
\end{center}
\caption{Decomposition of 3-regular tree of which every edge has length 1}\label{figure7}
\end{figure}
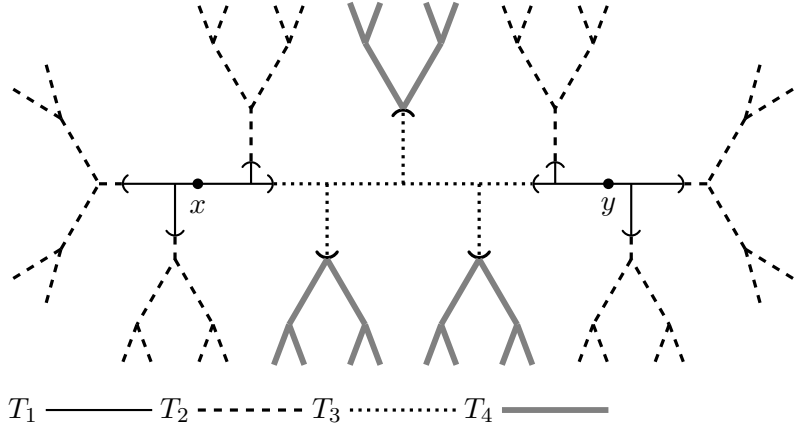\\
Since the constant $D_{r,l}$ in Harnack inequality \eqref{eq:2.2} depends on $r$ and $l$, we prove by separating into cases $d(x,y)\geq 2$ and $d(x,y)<2.$\\
\underline{\textit{Step 1, $d(x,y)\geq 2$}:} First, we will consider the case when $d(x,y)\geq2.$ For given $z\in \mathcal{T}$, denote by $z'$ the closest point to $z$ on $[x,y]$. Decompose $\mathcal{T}$ as follows: 
 \begin{equation}
 \begin{split}\nonumber
 &{T}_1:=B(x,1)\cup B(y,1)\\
 &{T}_2:=\{z\in T_1^c:z'\in \mathcal{T}_1\}\\
 &{T}_3:=\{z\in (\mathcal{T}_1\cup\mathcal{T}_2)^c:d(z,[x,y])<1\} \text{ and }\\
 & {T}_4:=(\mathcal{T}_1\cup \mathcal{T}_2\cup\mathcal{T}_3)^c.
\end{split}
\end{equation}
Harnack inequality \eqref{eq:2.2} and Lemma \ref{lem:2.4} show that 
\begin{equation}
\begin{split}
\nonumber&\int_{B(x,1)} G_{\lambda_0}(x,z)G_{\lambda_0}(z,y) e^{\delta d(x,z)}d\mu(z)\\
\nonumber&\overset{\eqref{eq:2.2}}\leq  e^{D_{1,1/2}}e^{2\delta }\int_{B(x,1)} G_{\lambda_0}(x,z)G_{\lambda_0}(x,y) d\mu \overset{\substack{\text{Lem}\\ \ref{lem:2.4}}}\leq  e^{D_{1,1/2}}e^{2\delta}C_{1}G_{\lambda_0}(x,y)e^{\delta d(x,y)},
\end{split}
\end{equation}
where $C_{1}$ is the constant in Lemma \ref{lem:2.4}. Similarly, we also have the inequality for $B(y,1)$ and we have
\begin{equation}\label{eq:6.4}
\int_{T_1} G_{\lambda_0}(x,z)G_{\lambda_0}(z,y) e^{\delta d(x,z)}d\mu(z)\leq  2e^{D_{1,1/2}}e^{2\delta}C_{1}G_{\lambda_0}(x,y)e^{\delta d(x,y)}.
\end{equation} 
For any point $z\in T_2$ with $d(z,z')<1$ and $z'\in B(x,1)$, by Harnack inequality \eqref{eq:2.2},
\begin{equation}\label{6a}
G_{\lambda_0}(x,z)G_{\lambda_0}(z,y)\leq e^{D_{1,1/2}} G_{\lambda_0}(x,z)G_{\lambda_0}(z',y)\leq e^{2D_{1,1/2}} G_{\lambda_0}(x,z)G_{\lambda_0}(x,y).
\end{equation}

For any point $z\in T_2$ with $d(z,z')\geq 1$ and $z'\in B(x,1),$ Harnack inequality \eqref{eq:2.2} and Ancona inequality \eqref{eq:2.4} show that
\begin{equation}\label{6b}
G_{\lambda_0}(x,z)G_{\lambda_0}(z,y)\overset{\eqref{eq:2.4}}\leq C G_{\lambda_0}(x,z)G_{\lambda_0}(z,z')G_{\lambda_0}(z',y)\overset{\eqref{eq:2.2}}\leq Ce^{2D_{1,1/2}} G_{\lambda_0}^2(x,z)G_{\lambda_0}(x,y).
\end{equation}
where $C$ is the constant in Theorem \ref{thm:2.6}. The reason we consider two cases is the assumption of Ancona inequality. Obviously, 
\begin{equation}
\begin{split}
&\{z\in T_2:z'\in B(x,1),d(z,z')<1\}\subset B(x,2)\text{ and }\\ 
&\{z\in T_2:z'\in B(x,1),d(z,z')\geq 1\}\subset B(x,1)^c.\nonumber
\end{split}
\end{equation}
By \eqref{6a}, \eqref{6b}, Lemma \ref{lem:2.4} and \eqref{eq:6.3},
\begin{equation}
\begin{split}
\nonumber&\int_{\{z\in T_2:z'\in B(x,1)\}}G_{\lambda_0}(x,z)G_{\lambda_0}(z,y) e^{\delta d(x,z)}d\mu(z)\\
&\overset{\eqref{6a}+\eqref{6b}}\leq e^{2D_{1,1/2}}e^{\delta(d(x,y)+2)}G_{\lambda_0}(x,y)\int_{\{z\in T_2:z'\in B(x,1),d(z,z')<1\}}G_{\lambda_0}(x,z)d\mu(z)\\
&+Ce^{2D_{1,1/2}}e^{\delta d(x,y)}G_{\lambda_0}(x,y)\int_{\{z\in T_2:z'\in B(x,1),d(z,z')\geq 1\}}G_{\lambda_0}^2(x,z) e^{\delta d(x,z)}d\mu(z)\\
&\overset{\substack{\text{Lem} \ref{lem:2.4}\\+\eqref{eq:6.3}}}\leq e^{2D_{1,1/2}}e^{\delta(d(x,y)+2)}C_2G_{\lambda_0}(x,y)+CBe^{2D_{1,1/2}}e^{\delta d(x,y)}G_{\lambda_0}(x,y)\\
&=(e^{2\delta}C_2+CB)e^{2D_{1,1/2}}G_{\lambda_0}(x,y)e^{\delta d(x,y)},
\end{split}
\end{equation}
where $C_2$ is the constant in Lemma \ref{lem:2.4}. Analogously, it is possible to compute the case when $z\in T_2$ and $z'\in B(y,1)$. Then we have
\begin{equation}\label{6ccc}
\begin{split}
\int_{T_2}G_{\lambda_0}(x,z)G_{\lambda_0}(z,y) e^{\delta d(x,z)}d\mu(z)\leq2(e^{2\delta}C_2+CB)e^{2D_{1,1/2}}G_{\lambda_0}(x,y)e^{\delta d(x,y)}.
\end{split}
\end{equation}
For any point $z\in T_3$, Harnack inequality and Ancona inequality show that
$$G_{\lambda_0}(x,z)G_{\lambda_0}(z,y)\leq e^{2D_{1,1/2}}G_{\lambda_0}(x,z')G_{\lambda_0}(z',y) \leq e^{2D_{1,1/2}}CG_{\lambda_0}(x,y).$$
Since the action of $\Gamma$ is cocompact, for any $x,y\in \mathcal{T}$ and $t\in [1,d(x,y)-1]$, the number of the points $z\in T_3$ with $d(x,z)=t$ is bounded above by some constant $M$. Using this and the above inequality, we have
\begin{equation}\label{eq:6.5} 
\begin{split}
&\int_{T_3}G_{\lambda_0}(x,z)G_{\lambda_0}(z,y) e^{\delta d(x,z)}d\mu(z)\leq e^{2D_{1,1/2}}CG_{\lambda_0}(x,y)\int_{T_2}e^{\delta d(x,z)}d\mu(z)\\
&\leq Me^{2D_{1,1/2}}CG_{\lambda_0}(x,y)\int_{1}^{d(x,y)-1}e^{\delta t} dt
\leq Me^{2D_{1,1/2}}\delta^{-1} e^{-\delta}CG_{\lambda_0}(x,y)e^{\delta d(x,y)}.
\end{split}
\end{equation}
By Ancona inequality, for any $z\in T_4$, 
\begin{equation}\label{eq:6.6}
G_{\lambda_0}(x,z)G_{\lambda_0}(z,y)\leq C^2 G_{\lambda_0}(x,z')G_{\lambda_0}(z',z)G_{\lambda_0}(z,z')G_{\lambda_0}(z',y)\leq C^3G_{\lambda_0}(x,y)G_{\lambda_0}^2(z,z').
\end{equation}
Let $\{v_i\}_{1\leq i\leq m}$ be the sequences of vertices in $[x,y]\cap T_1^c$. The cocompact action of $\Gamma$ guarantees that for any $x,y\in \mathcal{T}$ and $k>0$, the number of vertices $v_i$ on $[x,y]$ with $k-1 < d(x,v_i)\leq k$ is bounded above by some constant $D$. For any $z\in T_4$, the closest point $z'$ on $[x,y]$ is a vertex. Using the inequalities \eqref{eq:6.6} and \eqref{eq:6.3} in order, we have
\begin{equation}\label{eq:6.7}
\begin{split}
&\int_{T_4}G_{\lambda_0}(x,z)G_{\lambda_0}(z,y)e^{\delta d(x,z)}\mu(z)\\
&\overset{\eqref{eq:6.6}}\leq C^3G_{\lambda_0}(x,y)\sum_{i=1}^me^{\delta d(x,v_i)}\int_{\{z\in T_4:z'=v_i\}}G_{\lambda_0}^2(v_i,z)e^{\delta d(v_i,z)}d\mu(z)\\
 &\overset{\eqref{eq:6.3}}\leq C^3BG_{\lambda_0}(x,y)D\sum_{k=1}^{\llcorner d(x,y)-1\lrcorner}e^{\delta k}\leq C^3BDG_{\lambda_0}(x,y)\frac{e^{\delta d(x,y)+\delta}}{e^{\delta }-1},
 \end{split}
\end{equation}
where $\llcorner d(x,y)\lrcorner$ is the largest integer smaller than $d(x,y)$. Denote 
$$B_1:=4\max\left\{2e^{D_{1,1/2}}e^{\delta }C_{1},2(e^{2\delta}C_2+CB)e^{2D_{1,1/2}},Me^{2D_{1,1/2}}\delta^{-1} e^{-\delta }C,C^3BD\frac{e^{\delta }}{e^{\delta }-1}\right\}.$$
By \eqref{eq:6.4}, \eqref{6ccc}, \eqref{eq:6.5}, and \eqref{eq:6.7}, for any $x$ and $y$ with $d(x,y)\geq 2,$ we have
\begin{equation}\label{eq:6.8}
\int_{\mathcal{T}}G_{\lambda_0}(x,z)G_{\lambda_0}(z,y)e^{\delta d(x,z)}d\mu(z)\leq B_1G_{\lambda_0}(x,y)e^{\delta d(x,y)}.
\end{equation}
\begin{figure}[h]
\begin{center}
\begin{tikzpicture}[scale=1]
\node at (-0.35,-0.4) {$x$};\node at (0.65,-0.6) {$y$};\node at (0.2,0.1) {$v$};\node at (1.35,-1.05) {$y'$};\node at (2,-1.2) {$y''$};
\node at (-3.3,-2) {$T_1'$};\node at (-1.3,-2) {$T_2'$};\node at (-2.5,2.2) {$d(x,y)=3/5$};\node at (-2.63,1.7) {$d(x,y')=2$};\node at (-2.6,1.2) {$d(v,y'')=2$};
\draw  [-),line width=1.1pt, dashed](-3,-1)--(-{sqrt(3)},-1)--(-{sqrt(3)}/2,-1/2);\draw(-{sqrt(3)}/2,-1/2)--(0,0)--(0,1/5);\draw[(-,line width=1.1pt, dashed](0,1/5)--(0,2);\draw[line width=1.1pt, dashed] (0,2)--(-{sqrt(3)}/2,2.5);\draw [line width=1.1pt, dashed](-{sqrt(3)},-1)--(-{sqrt(3)},-1.7);\draw[line width=1.1pt, dashed](0,2)--(+{sqrt(3)}/2,2.5);\draw(0,0)--(+{sqrt(147)/10},-7/10);\draw[-),line width=1.1pt, dashed](+{sqrt(3)},-1)--(+{sqrt(147)/10},-7/10);
\draw [line width=1.1pt, dashed](+{sqrt(3)},-1)--(3,-1);\draw[line width=1.1pt, dashed] (+{sqrt(3)},-1)--(+{sqrt(3)},-1.7);
\fill (-{sqrt(3)}/5,-1/5) circle (2pt);\fill (+{sqrt(12)}/5,-2/5) circle (2pt);\fill (0,0) circle (2pt); \fill (+{sqrt(3)},-1) circle (2pt); \fill (+{sqrt(48)}/5,-4/5) circle (2pt);
\draw (-2.9,-2)--(-2,-2);\draw[line width=1.1pt, dashed](-1,-2)--(0,-2); 
  \end{tikzpicture}
\end{center}
\caption{Decomposition of 3-regular tree of which every edge has length 2}\label{figure8}
\end{figure}
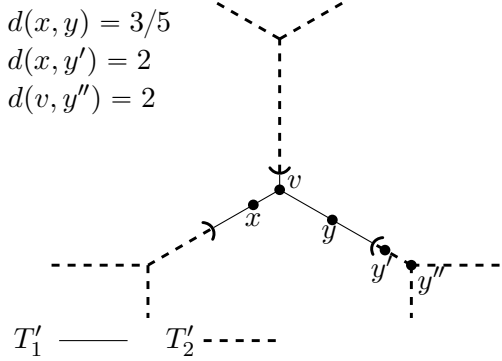
\underline{\textit{Step 2, $d(x,y)<2$}:}  Let us consider the case when $d(x,y)<2.$ 

Denote $T_1':=\overline{B(x,d(x,y)/2)\cup B(y,d(x,y)/2)}\text{ and }T_2':=T_1^c.$  Using Lemma \ref{lem:2.2}, we have for any $z\in B(x,d(x,y)/2)$, 
$$G_{\lambda_0}(z,y)\leq e^{D_{d(x,y)/2,d(x,y)/4}}G_{\lambda_0}(x,y).$$
In the proof of \eqref{eq:2.2} in \cite{HL}, we obtain $$D_{r,l}=\max\{\sqrt{4|S(x,r+s)|\mu(B(x,r))/l}:x\in \mathcal{T}, 0<s\leq l\},$$ where $|S(x,r+l)|$ is the cardinality of $S(x,r+l)$. Since $$|S(x,r+l)|<|S(x,1)|\text{ and }\mu(B(x,r))\leq r|S(x,1)|$$ for any $r,2l<1/2$,  $D_{d(x,y)/2,d(x,y)/4}\leq 4|S(x,1)|.$
Similar to \eqref{eq:6.4},  
\begin{equation}\label{eq:6.9}
\int_{T_1'} G_{\lambda_0}(x,z)G_{\lambda_0}(z,y) e^{\delta d(x,z)}d\mu(z)\leq  2e^{4|S(x,1)|}e^{\delta }C_{1}G_{\lambda_0}(x,y)e^{\delta d(x,y)}.
\end{equation}
For any $z\in T_2'$, $z'$ is either $x$, $y$ or if exists, a vertex on $[x,y]$ different to $x,y$ (see Figure \ref{figure8}). The number of vertices on $[x,y]$ is bounded above by $2D$. Consider the case when $z'=x$. Choose a point $y'$ with $y\in [x,y']$ and $d(x,y')=2.$ Using \eqref{eq:2.8} twice,
\begin{equation}\label{eq:6.10}
\begin{split}
 G_{\lambda_0}(x,z)G_{\lambda_0}(z,y)&=G_{\lambda_0}(x,z)G_{\lambda_0}(x,y)\mathbb{E}_z[1_{t_x(\omega)<\infty}(\omega)e^{\lambda_0 t_x(\omega)}]\\
&=G_{\lambda_0}(x,z)G_{\lambda_0}(x,y)\frac{G_{\lambda_0}(z,y')}{G_{\lambda_0}(x,y')}.
\end{split}
\end{equation}
Appying \eqref{eq:6.10} and \eqref{eq:6.8} in order, we have
\begin{equation}\label{eq:6.11}
\begin{split}
&\int_{\{z\in T_2':z'=x\}} G_{\lambda_0}(x,z)G_{\lambda_0}(z,y)e^{\delta d(x,z)}\mu(z)\\
&\overset{\eqref{eq:6.10}}=\frac{G_{\lambda_0}(x,y)}{G_{\lambda_0}(x,y')}\int_{\{z\in T_2:z'=x\}}G_{\lambda_0}(x,z)G_{\lambda_0}(z,y')e^{\delta d(x,z)}d\mu(z)\\
&\overset{\eqref{eq:6.8}}\leq B_1\frac{G_{\lambda_0}(x,y)}{G_{\lambda_0}(x,y')}G_\lambda(x,y')e^{\delta d(x,y')}\leq B_1e^{2\delta}G_{\lambda_0}(x,y)e^{\delta d(x,y)}.
\end{split}
\end{equation}
Suppose that there exists a vertex $v$ on $[x,y]$. Choose a point $y''$ with $y\in [v,y'']$ and $d(v,y'')=2.$
Using \eqref{eq:2.8} twice, we have for any $z\in T_2$ with $z'=v,$
\begin{equation}\label{eq:6.12}
\begin{split}
 G_{\lambda_0}(x,z)G_{\lambda_0}(z,y)&=G_{\lambda_0}(x,v)\mathbb{E}_z[1_{t_v(\omega)<\infty}(\omega)e^{\lambda_0 t_v(\omega)}]G_{\lambda_0}(z,v)\mathbb{E}_y[1_{t_v(\omega)<\infty}(\omega)e^{\lambda_0 t_v(\omega)}]\\
&=G_{\lambda_0}(x,y)G_{\lambda_0}(z,v)\frac{G_{\lambda_0}(z,y'')}{G_{\lambda_0}(v,y'')}.
\end{split}
\end{equation}
For any points $z\in T_2$ satisfying $z'=v$ or $z'=y$, $d(x,z)\leq 2+d(z',z).$ Similar to \eqref{eq:6.11}, applying \eqref{eq:6.12}, we obtain
\begin{equation}\label{eq:6.13}
\int_{T_2}G_\lambda(x,z)G_\lambda(z,y)e^{\delta d(x,z)}\leq (2D+2)B_1e^{4\delta }G_{\lambda_0}(x,y)e^{\delta d(x,y)}.
\end{equation}
Denote $B_0:=2\max\{2e^{4|\partial B(x,1)|}e^{\delta }C_{1},(2D+2)B_1e^{4\delta} \}.$ Using the inequalities \eqref{eq:6.9} and \eqref{eq:6.13}, Lemma \ref{lem:6.2} follows.
\end{proof}
\subsection{Derivatives of Green function} In this section, we will show the following theorem
\begin{thm}\label{thm:6.3} There exists a constant $\mathcal{C}_1$ such that for any distinct points $x,y\in \mathcal{T}$, 
\begin{equation}\label{aaa}
\lim_{\lambda\rightarrow\lambda_0-}\sqrt{\lambda_0-\lambda}\frac{\partial}{\partial \lambda}G_\lambda(x,y)=\frac{1}{2\mathcal{C}_1^{1/2}}c(x,y)\text{ and } \lim_{\lambda\rightarrow\lambda_0-}-\frac{P_\lambda}{\sqrt{\lambda_0-\lambda}}=2\mathcal{C}\mathcal{C}_1^{1/2},
\end{equation}
where $c(x,y):=\int_{\partial \cT}k_{\lambda_0}(x,y,\xi)\mu_{x,\lambda_0}(\xi).$
\end{thm}
As in the proof Theorem $6.1$ in \cite{LL}, we will prove  Proposition \ref{prop:6.4} and Proposition \ref{prop:6.6}. Then we obtain the proof of Theorem \ref{thm:6.3}.
\begin{lem}\label{lem:6.3} For any $\epsilon>0$ and $d>0$, there exist constants $R_d$ and $\delta_d$ such that for all $x,y\in \mathcal{T}$ with $d(x,y)<d$, $r>R_d$ and $\lambda\in [\lambda_0-\delta_d,\lambda)$, 
\begin{equation}\label{eq:6.14}
-P_\lambda\int_{B(x,r)^c} G_\lambda(x,z)G_\lambda(z,y)d\mu(z)\asymp_{(1+\epsilon)^{11}} \mathcal{C} c(x,y), 
\end{equation}
where $\cC$ is the constant in Theorem \ref{thm:5.14}.
\end{lem}
\begin{proof} By defintion of Martin kernel, we have
$$\int_{B(x,r)^c} G_\lambda(x,z)G_\lambda(z,y)d\mu(z)=\int_{B(x,r)^c} G_\lambda^2(x,z)k_\lambda(x,y,z)d\mu(z).$$
By Corollary \ref{coro:2.7} and Harnack inequality, for any $\epsilon>0$, there exist $R$ and $\delta$ such that for any $y$ with $d(x,y)<d$, $z$ with $d(x,z)>R$, $\xi\in \cO_x(z)$ and $\lambda\in [\lambda_0-\delta,\lambda_0]$,
\begin{equation}\label{eq:6.15.1}
k_\lambda(x,y,z)\asymp_{1+\epsilon} k_{\lambda_0}(x,y,z)\text{ and } k_{\lambda_0}(x,y,z)\asymp_{1+\epsilon}k_{\lambda_0}(x,y,\xi).
\end{equation}
Choose a smooth function $h:\mathbb{R}_+\rightarrow[0,1]$ satisfying the following properties: 
\begin{itemize}
\item If $h(d(x,\g x))\neq 0$, then $\g\in\G_x$. 
\item The support of $h$ is $[0,2\eta]$ and $h|_{[0,\eta]}\equiv 1$ for some sufficiently small $\eta$. 
\item $|h'|<2/\eta.$ 
\end{itemize}
The function $h$ allows us to define a function $f$ as follows: For any $g\in \mathcal{G}\mathcal{T}$,
$$f(g):=\frac{1}{|\G_x|}\sum_{\g\in \G} h(d(\g x,\pi (g)))k_{\lambda_0}(x, y, \g^{-1} g_+).$$
Since $f(g)$ depends on $\pi (g)$ and $g_+$, $f$ is considered as a function on $\cG^+\mathcal{T}.$
By the choice of $h$ and the bounded H\"older continuity of $\xi\mapsto k_{\lambda_0}(x,y,\xi)$, $f$ is H\"older continuous. For any $\g'\in \G$ and $g\in \cG \mathcal{T}$,
\begin{equation}
\begin{split}
f(\g' g)&=\frac{1}{|\G_x|}\sum_{\g\in \G} h(d(\g x,\pi (\g'g)))k_{\lambda_0}(x, y, \g^{-1} \g' g_+)\\
\nonumber&=\frac{1}{|\G_x|}\sum_{\g\in \G} h(d(\g'^{-1}\g x,\pi (g)))k_{\lambda_0}(x, y, (\g'^{-1}\g)^{-1}g_+)=f(g).
\end{split}
\end{equation}
This implies that $f$ is $\G$-invariant.
Let $R_0$ and $\delta_0$ be the constants in Theorem \ref{thm:5.14} which depend on $\epsilon$. Denote $R_d:=\max\{R,R_0\}$ and $\delta_d:=\min\{\delta,\delta_0\}$. By \eqref{eq:6.15.1}, Theorem \ref{thm:5.14} and Proposition \ref{prop:6.1}, for any $y$ with $d(x,y)<d$, $r>R_d$ and $\lambda\in [\lambda_0-\delta_d,\lambda_0]$,
\begin{equation}\label{eq:6.15.2}
\begin{split}
&-P_\lambda \int_{B(x,r)^c}G_\lambda^2(x,z)k_\lambda(x,y,z)d\mu(z)=-P_\lambda\int_{B(x,r)^c} G_\lambda^2(x,z)\frac{1}{|\G_x|}\sum_{\g\in \G_x}k_\lambda(x,y,\g^{-1}z)\mu(z)\\
&\overset{\eqref{eq:6.15.1}}\asymp_{(1+\epsilon)^2}-P_\lambda\int_{B(x,r)^c} G_\lambda^2(x,z)f(g_x^z)d\mu(z)
\overset{\substack{\text{Thm }\ref{thm:5.14}\\+\text{Prop } \ref{prop:6.1}}}\asymp_{(1+\epsilon)^{9}}\mathcal{C}\int_{\cG_x^+\mathcal{T}}f(w)d\mu_{\cG_x^+\mathcal{T}}^{\lambda_0}(w)
\end{split}
\end{equation}
where $g_x^z\in \cG\mathcal{T}_{x,z}.$

For any $w\in \cG_x\mathcal{T}$, if $h(d(\g x, \pi (w)))\neq 0$, then $\g \in \G_x$. We conclude that
\begin{equation}\label{eq:6.15.3}
\begin{split}
&\int_{\partial \mathcal{T}}f(w)d\mu_{\cG_x^+\mathcal{T}}^{\lambda_0}(w)\\
&=\frac{1}{|\G_x|}\sum_{\g\in \G_x} \int_{\partial \mathcal{T}} k_{\lambda_0}(x,y,\g^{-1} \xi)d\mu_{x,\lambda_0}(\xi)=\frac{1}{|\G_x|}\sum_{\g\in \G_x} \int_{\partial \mathcal{T}} k_{\lambda_0}(x,y,\xi)d\g^{-1}_*\mu_{x,\lambda_0}(\xi)\\
&\overset{\eqref{eq:3.3}}=\frac{1}{|\G_x|}\sum_{\g\in \G_x} \int_{\partial \mathcal{T}} k_{\lambda_0}(x,y,\xi)d\mu_{\g^{-1} x,\lambda_0}(\xi)= \int_{\partial \mathcal{T}}k_{\lambda_0}(x,y,\xi)d\mu_{x,\lambda_0}(\xi),
\end{split}
\end{equation}
The relation \eqref{eq:6.14} follows from \eqref{eq:6.15.2} and \eqref{eq:6.15.3}.
\end{proof}
\begin{prop}\label{prop:6.4}For any distinct points $x,y\in \mathcal{T}$, 
\begin{equation}\label{eq:6.16}
\lim_{\lambda\rightarrow\lambda_0-}-P_\lambda\frac{\partial}{\partial \lambda} G_\lambda(x,y)=\mathcal{C}c(x,y), 
\end{equation}
where $\cC$ is in Lemma \ref{lem:6.3}.
For any compact set $K$ with $x\in K$ and $Diam(K)\geq3Diam(T_0),$ there exists $\lambda'<\lambda_0$ such that the function defined by $y\mapsto \sup_{\lambda\in[\lambda',\lambda_0]} -P_{\lambda}\frac{\partial}{\partial \lambda}G_\lambda(x,y)$ is integrable on $K$.
\end{prop}
\begin{proof}As in the proof of Proposition \ref{prop:6.1}, we have
$$-P_{\lambda}\frac{\partial}{\partial\lambda}G_\lambda(x,y)=-P_\lambda\int_{\mathcal{T}}G_\lambda(x,z)G_\lambda(z,y)d\mu(z).$$
Following the proof of Lemma \ref{lem:6.2}, for any compact set $K$, it is possible to choose $C_K$ such that for any $y\in K\backslash\{x\}$ and $\lambda\in[0,\lambda_0)$, 
\begin{equation}\label{eq:6.17}
\int_{B(x,Diam(K)+1)}G_\lambda(x,z)G_\lambda(z,y)d\mu(z)\leq C_KG_{\lambda}(x,y).
\end{equation}
By Harnack inequality, we have for any $y\in K\backslash\{x\}$ and $z$ with $d(x,z)>Diam(K)+1$,
\begin{eqnarray}\label{eq:6.18}
G_\lambda(x,z)G_\lambda(z,y)\leq e^{D_{Diam(K),1}}G_\lambda^2(x,z)
\end{eqnarray}
By Proposition 4.10 in \cite{HL}, there exists constant $C$ such that for any $r>Diam(K)+1$ and $\lambda\in [0,\lambda_0]$
\begin{equation}\label{eq:6.19}
\sum_{z\in S(x,r)} G_{\lambda}^2(x,z)\leq C.
\end{equation}
Using \eqref{eq:6.18} and \eqref{eq:6.19} in order, we have that for any $r>Diam(K)+1$ and $\lambda\in[0,\lambda_0]$,
\begin{eqnarray}\label{eq:6.20}
\nonumber &&\int_{B(x,r)\backslash\overline{B(x,Diam(K)+1)}}G_\lambda(x,z)G_\lambda(z,y)d\mu(z)\\ &&\leq e^{D_{Diam(K),1}}\int_{Diam(K)+1}^{r}\sum_{z\in S(x,r)} G_\lambda^2(x,z)dr\leq  e^{D_{Diam(K),1}}C(r-Diam(K)-1),
\end{eqnarray}
where $dr$ is Lebesgue measure on $\mathbb{R}.$
Since $P_{\lambda_0}=0$, using the inequalities \eqref{eq:6.17} and \eqref{eq:6.20}, for any $r>Diam(K)+1,$ we have
\begin{eqnarray}\label{eq:6.21}
\nonumber&&\lim_{\lambda\rightarrow\lambda_0-}-P_\lambda\int_{B(x,r)}G_\lambda(x,z)G_\lambda(z,y)d\mu(z)\\
&&\leq\lim_{\lambda\rightarrow\lambda_0-}-P_\lambda \left\{C_KG_\lambda(x,y)+ e^{D_{Diam(K),1}}C(r-Diam(K)-1)\right\}=0.
\end{eqnarray}
By Lemma \ref{lem:6.3}, for any $\epsilon$, there exist constants $R_{Diam(K)+1}$ and $\delta_{Diam(K)+1}$ such that for any $r>R_{Diam(K)+1}$ and $\lambda\in[\lambda_0-\delta_{Diam(K)+1},\lambda_0],$
 \begin{eqnarray}
\nonumber-P_\lambda\int_{B(x,r)^c}G_\lambda(x,z)G_\lambda(z,y)d\mu(z)\asymp_{(1+\epsilon)^{11}}\mathcal{C}\int_{\partial\mathcal{T}} k_{\lambda_0}(x,y,\xi)d\mu_{x,\lambda_0}(\xi).
\end{eqnarray}
The first equality of \eqref{eq:6.24} follows from \eqref{eq:6.21}. The above relation shows that the last line of \eqref{eq:6.24} holds. 
\begin{equation}\label{eq:6.24}
\begin{split}
&\lim_{\lambda\rightarrow\lambda_0-}-P_\lambda\int_{\mathcal{T}} G_\lambda(x,z)G_\lambda(z,y)d\mu(y)\overset{\eqref{eq:6.21}}=\lim_{\lambda\rightarrow\lambda_0-}-P_\lambda\int_{{B(x,r)}^c}G_\lambda(x,z)G_\lambda(z,y)d\mu(y)\\
&\asymp_{(1+\epsilon)^{11}}\cC \int_{\partial\mathcal{T}}k_{\lambda_0}(x,y,\xi)d\mu_{x,\lambda_0}(\xi).
\end{split}
\end{equation}
Since \eqref{eq:6.24} holds for any $\epsilon$, we obtain \eqref{eq:6.16}. The relation \eqref{eq:6.24} also implies that the last statement of this proposition.
\end{proof}
\begin{remark}\label{rem:6.5}
Following the proof of Proposition \ref{prop:6.4}, we also have
\begin{equation}\label{eq:6.25}
\lim_{\lambda\rightarrow\lambda_0-}-P_\lambda\int_{{B(x,l_m)}^c} G_\lambda^2(x,z)d\mu(y)=\cC \mu_{x,\lambda_0}(\partial \mathcal{T}).
\end{equation}
\end{remark}
For any $g\in \mathcal{G}\mathcal{T}$ and $x\in \overline{\mathcal{T}}$, denote $t_{g,x}=(g_+| x)_{\pi(g)}$. For any $x\in\overline{\mathcal{T}}$ with $t_{g,x}<1$ and $d(x,g(0))>1$, the point $g(t_{g,x})$ is the closest to $x$ on $g$. Let $u$ be a function defined for any $g\in\mathcal{G}\mathcal{T}$, by
$$u(g):=\mathcal{C}\int_{\partial \mathcal{T}}h_1(t_{g,\xi})\frac{ \theta^{\lambda_0}_{g(t_{g,\xi})}(g_-,\xi)\theta^{\lambda_0}_{g(t_{g,\xi})}(\xi,g_+)}{\theta_{g(t_{g,\xi})}^{\lambda_0}(g_-,g_+)}\mu_{g(t_{g,\xi}),\lambda_0}(\xi),$$
where  $h_1(t)=\max\{1-|t|,0\}$. Ancona inequality \eqref{eq:2.4} and Corollary \ref{coro:2.7} show that $u$ is a bound H\"older continuous function. Obviously, $u$ is $\G$-invariant. 
The function $u$ satisfies the following proposition.
\begin{prop}\label{prop:6.6}For any $x,y$, 
\begin{equation}\label{eq:6.261}
\lim_{\lambda\rightarrow\lambda_0-}-P_\lambda^3\frac{\partial^2}{\partial \lambda^2}G_\lambda(x,y)=2\mathcal{C}^3 c(x,y)\int_{\G\backslash \cG\mathcal{T}} u(g)d\overline{m}_{\lambda_0}(g),
\end{equation}
where $\mathcal{C}$ is the constant in Theorem \ref{thm:5.14}. For any compact set $K$ with $x\in K$ and $Diam(K)\geq3Diam(T_0),$ there exists $\lambda'<\lambda_0$ such that the function defined by $y\mapsto \sup_{\lambda\in[\lambda',\lambda_0]} -P_{\lambda}^3\frac{\partial^2}{\partial \lambda^2}G_\lambda(x,y)$ is integrable on $K$.
\end{prop}
\begin{proof}By \eqref{eq:2202041}, 
$$\frac{\partial^2}{\partial \lambda^2}G_\lambda(x,y)=2\int_{\mathcal{T}\times \mathcal{T}}G_\lambda(x,z_1)G_\lambda(z_1,z_2)G_\lambda(z_2,y)d\mu(z_2)d\mu(z_1).$$
\underline{\textit{Step 1. The integral over $B(x,r)\times \mathcal{T}$}:} Following the proof of Proposition 6.2 in \cite{LL}, we will verify 
 $$\lim_{\lambda\rightarrow\lambda_0-}-2P_\lambda^3\int_{B(x,r)\times \mathcal{T}}G_\lambda(x,z_1)G_\lambda(z_1,z_2)G_\lambda(z_2,y)d\mu(z_2)d\mu(z_1)=0.$$
Using Corollary \ref{coro:5.16}, we will show that the limit of the remaining part coincides with the right hand side of \eqref{eq:6.261}.

 Let $R_0'$ be the constant in Corollary \ref{coro:5.16}. By Corollary \ref{coro:3.6}, for any $\epsilon>0,$ there exist constants $R>R_0'$ and $\delta$ such that for any $x_0$, $x_1,$ $x_2$ and $x_3$ with $d(x_0,x_i)>R$, $\xi_i\in \mathcal{O}_{x_0}(x_i)$ and $\lambda\in [\lambda_0-\delta,\lambda_0],$
$$\theta_{x_0}^\lambda(x_i,x_j)\asymp_{(1+\epsilon)}\theta_{x_0}^{\lambda_0}(\xi_i,\xi_j).$$

Fix $r>2R+2d(x,y)+2$ with $\frac{r-2R-2d(x,y)-2}{r}\asymp_{1+\epsilon}1$. The first equality of \eqref{eq::6.33} follows from Proposition \ref{prop:6.4}. Harnack inequality \eqref{eq:2.2} shows the first inequality of \eqref{eq::6.33}. By Harnack inequality \eqref{eq:2.2}, $\mu_{x,\lambda_0}(\partial\mathcal{T})\leq e^{2D_{Diam(T_0),1}}\mu_{x_0\lambda_0}(\partial \mathcal{T})$ for any $x\in \mathcal{T}$ and $x_0\in T_0$. Using this and Lemma \ref{lem:2.4}, we have the last equality of \eqref{eq::6.33}.
\begin{equation}\label{eq::6.33}
\begin{split}
&\lim_{\lambda\rightarrow\lambda_0-}-P_\lambda^3\int_{B(x,r)}G_\lambda(x,z_1)\int_{\mathcal{T}}G_\lambda(z_1,z_2)G_\lambda(z_2,y)d\mu(z_2)d\mu(z_1)\\
&\overset{\text{Prop }\ref{prop:6.4}}=\lim_{\lambda\rightarrow\lambda_0-}\mathcal{C}P_\lambda^2\int_{B(x,r)}G_\lambda(x,z_1)\int_{\partial \mathcal{T}}k_{\lambda_0}(z_1,y,\xi)d\mu_{z_1,\lambda_0}(\xi)d\mu(z_1)\\
&\overset{\eqref{eq:2.2}}\leq\lim_{\lambda\rightarrow\lambda_0-}\mathcal{C}P_\lambda^2\int_{B(x,r)}G_\lambda(x,z_1)e^{D_{r,1}}\mu_{z_1,\lambda_0}(\partial \mathcal{T})d\mu(z_1)\\
&\overset{\text{Lem }\ref{lem:2.4}} \leq\lim_{\lambda\rightarrow\lambda_0-}\mathcal{C}P_\lambda^2C_re^{D_{r,1}+2D_{Diam(T_0),1}}\mu_{x_0,\lambda_0}(\partial \mathcal{T})=0.
\end{split}
\end{equation}
\underline{\textit{Step 2. The integral over $B(x,r)^c\times \mathcal{T}$}:} The remaining part of the integral satisfies
 \begin{equation}\label{t}
\begin{split}
 &-P_\lambda^3\int_{{B(x,r)}^c}G_\lambda(x,z_1)\int_{\mathcal{T}}G_\lambda(z_1,z_2)G_\lambda(z_2,y)d\mu(z_2)d\mu(z_1)\\
&=-P_\lambda^3\int_{{B(x,r)}^c}G_\lambda^2(x,z_1)\frac{G_\lambda(y,z_1)}{G_\lambda(x,z_1)}\int_{\mathcal{T}}\frac{G_\lambda(z_1,z_2)G_\lambda(z_2,y)}{G_\lambda(y,z_1)}d\mu(z_2)d\mu(z_1)\\
&=-P_\lambda^3\int_{{B(x,r)}^c}G_\lambda^2(x,z_1)k_\lambda(x,y,z_1)\int_{\mathcal{T}}\frac{G_\lambda(z_1,z_2)G_\lambda(z_2,y)}{G_\lambda(y,z_1)}d\mu(z_2)d\mu(z_1).
\end{split}
\end{equation}
To obtain \eqref{eq:6.261} from \eqref{t}, we need to show that for any $\epsilon>0$, there exists a constant $r$ such that for any $z_1$ with $d(x,z_1)>r$ and for any $g_x^{z_1}\in \mathcal{G}\mathcal{T}_{x,z_1}$, 
\begin{equation}\label{l}
-P_\lambda \int_{\mathcal{T}}\frac{G_\lambda(z_1,z_2)G_\lambda(z_2,y)}{G_\lambda(y,z_1)}d\mu(z_2)\asymp_{(1+\epsilon)^{15}}\int_{R-d(x,y)-1}^{r-R-d(x,y)+1}u(\phi_t g_x^{z_1})dt.
\end{equation}
Let $z'$ be the closest point to $z$ on $[x,z_1]$. Decompose $\mathcal{T}$ into $\mathcal{T}_1(z_1)$, $\mathcal{T}_2(z_1)$ and $\mathcal{T}_3(z_1),$ where 
\begin{equation}
\begin{split}
\nonumber&\mathcal{T}_1(z_1):=B(x,R+d(x,y))\cup B(z_1,R+d(x,y))\\
&\mathcal{T}_2(z_1):=\{z_2\in (\mathcal{T}_1(z_1))^c:z_2'\in B(x,R+d(x,y)-1)\cup B(z_1,R+d(x,y)-1)\}\\
&\mathcal{T}_3(z_1):=(\mathcal{T}_1(z_1)\cup \mathcal{T}_2(z_1))^c.
\end{split}
\end{equation}
The decomposition is a complicated version of the one in the proof of Lemma \ref{lem:6.2} (see Figure \ref{figure7}). First, we will verify that the integral value over $\mathcal{T}$ in \eqref{l} is approximated by the integral over the set $\mathcal{T}_3(z_1).$
\begin{lem}\label{lem:a}For any $i\in\{1,2\},$ 
\begin{equation}\label{e-}
\lim_{\lambda\rightarrow\lambda_0-}-P_\lambda^3\int_{{B(x,r)}^c}G_\lambda^2(x,z_1)k_\lambda(x,y,z_1)\int_{\mathcal{T}_i(z_1)}\frac{G_\lambda(z_1,z_2)G_\lambda(z_2,y)}{G_\lambda(y,z_1)}d\mu(z_2)d\mu(z_1)=0.
\end{equation}
\end{lem}
\begin{proof}
 Since $\lim_{\lambda\rightarrow\lambda_0-}P_\lambda=0$ and $$\lim_{\lambda\rightarrow\lambda_0-}-P_\lambda\int_{{B(x,r)}^c}G_\lambda^2(x,z_1)k_\lambda(x,y,z_1)d\mu(z_1)=\mathcal{C}c(x,y)$$ by Proposition \ref{prop:6.4},
it suffices to show that the limit
$$\lim_{\lambda\rightarrow\lambda_0-}-P_\lambda\int_{\mathcal{T}_i(z_1)}\frac{G_\lambda(z_1,z_2)G_\lambda(z_2,y)}{G_\lambda(y,z_1)}d\mu(z_2)$$
is bounded.\\
\underline{\textit{Step 1: i=1}} By Harnack inequality \eqref{eq:2.2}, for any for $z_2\in B(x,R+d(x,y))$
$$\frac{G_\lambda(z_1,z_2)G_\lambda(z_2,y)}{G_\lambda(y,z_1)} \overset{\eqref{eq:2.2}}\leq e^{D_{R+d(x,y),1}}\frac{G_\lambda(y,z_1)G_\lambda(z_2,y)}{G_\lambda(y,z_1)}\leq e^{D_{R+d(x,y),1}}G_\lambda(z_2,y). $$

Using above inequality, we have the first inequality below. Lemma \ref{lem:2.4} shows the second inequality below. 
  \begin{equation}
 \begin{split}
 \nonumber \int_{B(x,R+d(x,y))}\frac{G_\lambda(z_1,z_2)G_\lambda(z_2,y)}{G_\lambda(y,z_1)}d\mu(z_2) &\leq e^{D_{R+d(x,y),1}}\int_{B(x,R+d(x,y))}G_\lambda(z_2,y)d\mu(z_2)\\
\nonumber &\overset{\substack{\text{Lem }\\ \ref{lem:2.4}}}\leq  e^{D_{R+d(x,y),1}}C_{R+2d(x,y)}.
\end{split}
 \end{equation}
 Following the above inequality for $B(z_1,R+d(x,y))$, we have 
 \begin{equation}
 \begin{split}\nonumber
\lim_{\lambda\rightarrow\lambda_0-}-P_\lambda\int_{\mathcal{T}_1(z_1)}\frac{G_\lambda(z_1,z_2)G_\lambda(z_2,y)}{G_\lambda(y,z_1)}d\mu(z_2)d\mu(z_1)\leq \lim_{\lambda\rightarrow\lambda_0-}-2e^{D_{R+d(x,y),1}}C_{R+2d(x,y)}P_\lambda=0.
 \end{split}
 \end{equation}
 \underline{\textit{Step 2: i=2}}  Let $z_2''$ be the closest point to $z_2$ on $[y,z_1]$. For any $z_2\in \mathcal{T}_2(z_1),$ $z_2''$ is contained in $B(x,R+d(x,y)-1)\cup B(z_1,R+d(x,y)-1)$. 
  
By Ancona inequality \eqref{eq:2.4} and Harnack inequality \eqref{eq:2.2}, for any $z_1\in B(x,r)^c$, $z_2\in \mathcal{T}_2(z)$ with $z_2''\in B(x,R+d(x,y)-1)$ and $d(y,z_2'')<1,$ and for any $\lambda\in [0,\lambda_0]$,
 \begin{equation}\label{e}
 \begin{split}
 &\frac{G_\lambda(z_1,z_2)G_\lambda(z_2,y)}{G_\lambda(z_1,y)}\overset{\eqref{eq:2.4}}\leq C\frac{G_\lambda(z_1,z_2'')G_\lambda(z_2'',z_2)G_\lambda(z_2,y)}{G_\lambda(z_1,y)}\\
  &\overset{\eqref{eq:2.2}}\leq Ce^{3D_{R+d(x,y)-1,1}} \frac{G_\lambda(z_1,y)G_\lambda(x,z_2)G_\lambda(z_2,x)}{G_\lambda(z_1,y)} \leq C^3e^{3D_{R+d(x,y)-1,1}} G_\lambda(x,z_2)^2.
 \end{split}
 \end{equation}
 Using Ancona inequality \eqref{eq:2.4} and Harnack inequality \eqref{eq:2.2}, we have for any $z_1\in B(x,r)^c$, $z_2\in \mathcal{T}_2(z_1)$ with $z_2''\in B(x,R+d(x,y)-1)$ and $d(y,z_2'')\geq 1$, and for any $\lambda\in [0,\lambda_0]$, 
 \begin{equation}\label{f}
 \begin{split}
 \frac{G_\lambda(z_1,z_2)G_\lambda(z_2,y)}{G_\lambda(z_1,y)}&\overset{\eqref{eq:2.4}}\leq C^2\frac{G_\lambda(z_1,z_2'')G_\lambda(z_2'',z_2)G_\lambda(z_2,z_2'')G_\lambda(z_2'',y)}{G_\lambda(z_1,y)}\\
  &\overset{\eqref{eq:2.2}}\leq C^2e^{2D_{R+d(x,y)-1,1}} \frac{G_\lambda(x,z_2)^2 G_\lambda(z_1,z_2'')G_\lambda(z_2'',y)}{G_\lambda(z_1,y)}\\
   &\overset{\eqref{eq:2.4}}\leq C^3e^{2D_{R+d(x,y)-1,1}} \frac{G_\lambda(x,z_2)^2 G_\lambda(z_1,y)}{G_\lambda(z_1,y)}\leq C^3e^{3D_{R+d(x,y)-1,1}} G_\lambda(x,z_2)^2.
 \end{split}
 \end{equation}
 Similar to \eqref{e} and \eqref{f}, we have for any $z_2\in \mathcal{T}_2(z_1)$ with $z_2'\in B(z_1,R+d(x,y)-1)$ and for any $\lambda\in [0,\lambda_0]$, 
\begin{equation}\label{g}
\frac{G_\lambda(z_1,z_2)G_\lambda(z_2,y)}{G_\lambda(z_1,y)}\leq C^3e^{3D_{R+d(x,y)-1,1}} G_\lambda(z_1,z_2)^2.
\end{equation}
 Using \eqref{e}, \eqref{f} and \eqref{g}, and Remark \ref{rem:6.5}, we have 
\begin{equation}\label{eq:6.26}
\begin{split}
\lim_{\lambda\rightarrow\lambda_0-}&-P_\lambda\int_{\mathcal{T}_2(z_1)}\frac{G_\lambda(z_1,z_2)G_\lambda(z_2,y)}{G_\lambda(y,z_1)}d\mu(w)\\
\leq& \lim_{\lambda\rightarrow\lambda_0-}-P_\lambda C^3e^{3D_{R+d(x,y)-1,1}}\\
&\times \left\{\int_{B(x,R+d(x,y))^c} G_\lambda(x,z_2)^2d\mu(z_2)+\int_{B(z_1,R+d(x,y))^c} G_\lambda(z_1,z_2)^2d\mu(z_2)\right\}\\
\overset{\substack{\text{Remark}\\\ref{rem:6.5}}}=& C^3e^{3D_{R+d(x,y)-1,1}} \biggl(\mu_{x,\lambda_0}(\partial \mathcal{T})+\mu_{z_1,\lambda_0}(\partial \mathcal{T})\biggr).\\
 \end{split}
\end{equation}
Since $\mu_{z,\lambda_0}(\partial \mathcal{T})$ is bounded for any $z\in \mathcal{T}$, we have \eqref{e-}.
\end{proof}

To check that $u$ satisfies \eqref{l}, we use a function $u_\lambda$ defined for any $g\in \mathcal{G}\mathcal{T}$, by
 $$u_{\lambda,z_1}(g):=\int_{\mathcal{T}} h_1(t_{g,z_2})\frac{G_\lambda(z_1,z_2)G_\lambda(z_2,y)}{G_\lambda(y,z_1)}d\mu(z_2).$$

\begin{lem}\label{lem:6.7}For any $\epsilon>0$, there exists $\delta_3$ such that for any $\lambda\in[\lambda_0-\delta_3,\lambda_0]$, $g\in \cG\mathcal{T}_{x,z_1}$ and $s$ with $R+d(x,y)-1\leq s\leq r-R-d(x,y)+1,$
\begin{equation}\label{eq:6.33}
-P_\lambda{u}_{\lambda,z_1}(\phi_sg)\asymp_{(1+\epsilon)^{13}}\mathcal{C}u(\phi_s g).
\end{equation}
\end{lem}
\begin{proof} For any $z_2\in \mathcal{T}_3(z_1)$, $z_2'=z_2''$ and $d(x,z_2')\in [R+d(x,y)-1,r-R-d(x,y)+1]$. For any $z_2\in \mathcal{T}_3(z_1)$ and $g\in \mathcal{G}\mathcal{T}_{x,z_1}$, the closest point to $z_2$ on $g$ coincides with $z_2'$.
By Harnack inequality \eqref{eq:2.2} and Ancona inequality \eqref{eq:2.4}, for any $\lambda\in[0,\lambda_0]$, $g$ in $\cG\mathcal{T}_{x,z_1}$ and $s$ with $R+d(x,y)-1\leq s\leq r-R-d(x,y)+1,$
\begin{equation}\label{m}
\begin{split}
&-P_\lambda\int_{B(g(s),R)}h_1(t_{\phi_sg,z_2})\frac{G_\lambda(z_1,z_2)G_\lambda(z_2,y)}{G_\lambda(y,z_1)}d\mu(z_2)\\
&\leq -Ce^{2D_{R,1}}P_\lambda\int_{B(g(s),R)}h_1(t_{\phi_sg,z_2})d\mu(z_2)\leq -P_\lambda Ce^{2D_{R,1}}\mu(B(g(s),R)).
\end{split}
\end{equation}
Let $\delta_0$ be the constant in Theorem \ref{thm:5.14}. The inequality \eqref{m} allows us to choose $\delta_3<\delta_0$ satisfying that for any $\lambda \in [\lambda_0-\delta_3,\lambda_0],$  $e^{P_\lambda r}\asymp_{1+\epsilon}1$ and  $g\in \mathcal{G}\mathcal{T}_{x,z_1},$
\begin{equation}\label{nn}
\begin{split}
-P_\lambda{u}_{\lambda,z_1}(\phi_s g)\asymp_{1+\epsilon}-P_\lambda\int_{B(g(s),R)^c} h_1(t_{\phi_sg,z_2})\frac{G_\lambda(z_1,z_2)G_\lambda(z_2,y)}{G_\lambda(y,z_1)}d\mu(z_2).
\end{split}
\end{equation}
For any $z_2\in \mathcal{T}_3(z_1)$ with $d(z_2,z_2')>R$,
\begin{equation}\label{jj}
\begin{split}
&\frac{G_\lambda(z_1,z_2)G_\lambda(z_2,y)}{G_\lambda(y,z_1)}=G_\lambda^2(z_2,z_2')\frac{ \theta^\lambda_{z_2'}(z_1,z_2)\theta^\lambda_{z_2'}(z_2,y)}{\theta_{z_2'}^\lambda(y,z_1)}.\end{split}\end{equation}
By the choice of $R$, for any $z_2\in \mathcal{T}_3(z_1)$ and $g_{s,z_2}\in \mathcal{G}\mathcal{T}_{g(s),z_2}$ with $d(z_2,z_2')>R$ and $h_1(t_{\phi_sg,z_2})<1$, 
\begin{equation}\label{k}
\begin{split}
G_\lambda^2(z_2,z_2')\frac{ \theta^\lambda_{z_2'}(z_1,z_2)\theta^\lambda_{z_2'}(z_2,y)}{\theta_{z_2'}^\lambda(y,z_1)}\asymp_{(1+\epsilon)^3}G_\lambda^2(z_2,z_2')\frac{ \theta^\lambda_{z_2'}(g_-,g_{s,z_2+})\theta^\lambda_{z_2'}(g_{s,z_2+},g_+)}{\theta_{z_2'}^\lambda(g_-,g_+)}.
\end{split}
\end{equation}
Let $\overline{u}_g$ be a $\Gamma$-invariant bounded function $\overline {u}_{g}$ defined by for any $g'\in \mathcal{G}\mathcal{T}$, by
$$\overline{u}_{g}(g')=\frac{1}{|\G_{\pi(g)}|}\sum_{\g\in \G}h_1(t_{g, \gamma g'_+})h(d(\pi(g),\pi(\g g'))\frac{ \theta^{\lambda_0}_{g(t_{g, \gamma g'_+})}(g_-,\g g'_+)\theta^{\lambda_0}_{g(t_{g, \gamma g'_+})}(\g g'_+,g_+)}{\theta_{g(t_{g, \gamma g'_+})}^{\lambda_0}(g_-,g_+)},$$
where $h$ is the function defined in the proof of Lemma \ref{lem:6.3}. Similar to the proof of Lemma \ref{lem:6.3}, H\"older continuity of $\overline{u}_g$ follows from Corollary \ref{coro:2.7}. Combining \eqref{nn}, \eqref{jj} and \eqref{k}, for any $g\in \mathcal{G}\mathcal{T}_{x,z_1}$, we have 
\begin{equation}\label{n}
\begin{split}
&-P_\lambda{u}_{\lambda,z_1}(\phi_s g)\asymp_{(1+\epsilon)^4}-P_\lambda\int_{B(g(s),R)^c} G_\lambda^2(g(s),z_2)\overline u_{\phi_sg}(g_{s,z_2})d\mu(z_2)
\end{split}
\end{equation}
where $g_{s,z_2}\in \mathcal{G}\mathcal{T}_{g(s),z_2}.$
 By Theorem \ref{thm:5.14}, for any $\lambda\in [\lambda_0-\delta_3,\lambda_0],$
\begin{equation}\label{eq:6.35}
\begin{split}
\nonumber&-P_\lambda{u}_{\lambda,z_1}(\phi_sg)\asymp_{(1+\epsilon)^4}-P_\lambda\int_{B(g(s),R)^c} G_\lambda^2(g(s),z_2)\overline u_{\phi_sg}(g_{s,z_2})d\mu(z_2)\\
&\overset{\substack{\text{Thm}\\\ref{thm:5.14}}}\asymp_{(1+\epsilon)^{9}} \mathcal{C} \int_{\cG_{g(s)}^+\mathcal{T}}\overline{u}_{\phi_s g}(w)d\mu_{\cG_{g(s)}^+\mathcal{T}}^{\lambda_0}(\omega)\\
&=\mathcal{C} \int_{\partial \mathcal{T}}\frac{1}{|\G_{g(s)}|}\sum_{\g\in \G_{g(s)}}h_1(t_{\phi_sg,\g \xi})\frac{ \theta^{\lambda_0}_{g(t_{\phi_sg,\g \xi})}(g_-,\g \xi)\theta^{\lambda_0}_{g(t_{\phi_sg,\xi})}(\g \xi,g_+)}{\theta_{g(t_{\phi_sg,\g \xi})}^{\lambda_0}(g_-,g_+)}d\mu_{g(t_{\phi_sg,\g \xi}),\lambda_0}(\xi)\\
&=\mathcal{C} \int_{\partial \mathcal{T}}h_1(t_{\phi_sg,\xi})\frac{ \theta^{\lambda_0}_{g(t_{\phi_sg,\xi})}(g_-,\xi)\theta^{\lambda_0}_{g(t_{\phi_sg,\xi})}(\xi,g_+)}{\theta_{g(t_{\phi_sg,\xi})}^{\lambda_0}(g_-,g_+)}d\mu_{g(t_{\phi_sg,\xi}),\lambda_0}(\xi)= \mathcal{C}u(\phi_s g).\\
\end{split}
\end{equation}
The above equation completes the proof.
\end{proof}
Since $\int_{R+d(x,y)-1}^{r-R-d(x,y)+1}h_1(d(\pi(\phi_tg_x^{z_1}),z_2'))dt\leq1$ for any $z_2\in \mathcal{T}_3(z_1)$ and $g_{x}^{z_1} \in \mathcal{G}\mathcal{T}_{x,z_1}$, the function $u_\lambda$ satisfies
\begin{equation}\label{eq:6.32}
\int_{R+d(x,y)}^{r-R-d(x,y)}{u}_{\lambda}(\phi_t g_{x}^{z_1}) dt\leq\int_{\mathcal{T}_3(z_1)}\frac{G_\lambda(z,w)G_\lambda(w,y)}{G_\lambda(y,z)}d\mu(w)\leq 
\int_{R+d(x,y)-1}^{r-R-d(x,y)+1}{u}_{\lambda}(\phi_t g_{x}^{z_1}) dt.
\end{equation}
Applying Lemma \ref{lem:a} and \eqref{eq:6.32}, for any $\lambda$ sufficiently close to $\lambda_0$, we have the first relation of \eqref{eq:6.40}. The second relation of \eqref{eq:6.40} follows from Lemma \ref{lem:6.7}.
\begin{equation}\label{eq:6.40}
\begin{split}
-{P_\lambda}\int_{\mathcal{T}}\frac{G_\lambda(z_1,z_2)G_\lambda(z_2,y)}{G_\lambda(y,z_1)}d\mu(z_2)&\asymp_{(1+\epsilon)^2}-{P_\lambda}\int_{R+d(x,y)-1}^{r-R-d(x,y)+1}{u}_{\lambda}(\phi_t g_{x}^{z_1}) dt\\
&\asymp_{(1+\epsilon)^{13}}{\mathcal{C}}\int_{R+d(x,y)-1}^{r-R-d(x,y)+1}{u}(\phi_t g_{x}^{z_1}) dt.
\end{split}
\end{equation}
The second line of \eqref{eq:6.3612} follows from \eqref{eq:6.40}. Corollary \ref{coro:5.16} shows the last line of \eqref{eq:6.3612}.
\begin{equation}
\begin{split}\label{eq:6.3612}
&\lim_{\lambda\rightarrow\lambda_0-}-P_\lambda^3\int_{{B(x,r)}^c}G_\lambda^2(x,z_1)k_\lambda(x,y,z_1)\int_{\mathcal{T}_3(z_1)}\frac{G_\lambda(z_1,z_2)G_\lambda(z_2,y)}{G_\lambda(y,z_1)}d\mu(z_2)d\mu(z_1)\\
& \overset{ \eqref{eq:6.40}} \asymp_{(1+\epsilon)^{15}}\lim_{\lambda\rightarrow\lambda_0-}\mathcal{C}P_\lambda^2\int_{B(x,r)^c}G_\lambda^2(x,z_1)k_\lambda(x,y,z_1)\int_{R+d(x,y)-1}^{r-R-d(x,y)+1}{u}(\phi_t g_x^{z_1}) dtd\mu(z_1)\\
& \overset{\substack{\text{Coro}\\\ref{coro:5.16}}}\asymp_{(1+\epsilon)^{12}}{\mathcal{C}}^3\int_{\partial \mathcal{T}} k_{\lambda_0}(x,y,\xi)\mu_{x,\lambda_0}(\xi)\int_{\G\backslash \cG\mathcal{T}} u(g)d\overline{m}_{\lambda_0}(g).
 \end{split}
 \end{equation}
Since $\epsilon$ is arbitrary, we have \eqref{eq:6.25}. The equation \eqref{eq:6.3612} shows the last statement of Proposition \ref{prop:6.6}.
 \end{proof}
\begin{proof}[Proof of Theorem \ref{thm:6.3}] Denote $F(\lambda):=\frac{\partial}{\partial \lambda}G_\lambda(x,y).$ By Proposition \ref{prop:6.4} and Proposition \ref{prop:6.6}, we have
$$\lim_{\lambda\rightarrow \lambda_0}\frac{2F(\lambda)'}{F(\lambda)^3}=4\mathcal{C}_1c(x,y)^{-2},$$
where $\mathcal{C}_1=\int_{\G\backslash \cG\mathcal{T}} u(g)d\overline{m}_{\lambda_0}(g).$
 Since $P_{\lambda_0}=0$, $\frac{\partial}{\partial\lambda}G_\lambda(x,y)$ diverges as $\lambda$ goes to $\lambda_0$. Thus for any $\lambda$ sufficiently close to $\lambda_0$,
$\frac{1}{F(\lambda)^2}$ is asymptotic to $4{\mathcal{C}_1c(x,y)^{-2}({\lambda_0-\lambda})}.$ This implies the first equation of \eqref{aaa}. Obviously, the second equation follows from the first equation. \end{proof}
The second term of \eqref{aaa} and Proposition \ref{prop:6.6} shows the following corollary.
\begin{coro}For any distinct points $x$ and $y$, we have
\begin{equation}
\lim_{\lambda\rightarrow\lambda_0-}(\sqrt{\lambda_0-\lambda})^3\frac{\partial^2}{\partial \lambda^2}G_\lambda(x,y)=\frac{1}{4\mathcal{C}_0^{1/2}}c(x,y).
\end{equation}
\end{coro}
\subsection{Proof of local limit theorem} The following proposition is proved in \cite{LL}, based on Hardy-Littlewood Tauberian Theorem. The proof of the following proposition is the same as in \cite{LL} when we replace the volume form by the measure $\mu$. 
\begin{prop}\label{prop:6.11} Let $f$ be a compactly supported $C^\infty$-function on $\mathcal{T}$. Suppose that all derivatives of $f$ at all vertices in the support of the function $f$ are zero. Then
\begin{equation}
\lim_{t\rightarrow\infty}t^{3/2}\int_{\mathcal{T}\times\mathcal{T}}e^{\lambda_0 t}p(t,x,y)f(x)f(y)d\mu(x)d\mu(y)=\frac{1}{2\sqrt{\pi}\mathcal{C}_0^{1/2}}c(x,y).
\end{equation}
\end{prop}
To prove Proposition \ref{prop:6.11}, the spectral measure on the spectrum of $\Delta-\lambda_0$ is used. The function $f$ need to be in the domain of Laplacian $\Delta$ on $\mathcal{T}$. The function $f$ in Proposition \ref{prop:6.11} is in the domain of Laplacian. 

\begin{proof}[Proof of the Theorem \ref{thm:1.3}] Fix a compact set $K$ and a closed interval $[a,b].$ Lemma \ref{bssw} implies that for any $y\in K$ and $t\in [a,b]$, the function $x\mapsto p(t,x,y)$ is H\"older continuous. The convergence of Green function guarantees that $p(t,x,y)$ goes zero as $t\rightarrow\infty.$  Thus the $\|p(t,x,y)\|_\alpha$ converges to zero as $t\rightarrow \infty.$ Since $p(t,x,y)=p(t,y,x)$, for any sufficiently large $t$ and compact set $K_1$ containing $x$, $p(t,x,x)\asymp_{1+\epsilon}p(t,y_1,y_2)$ for any $y_1,y_2\in K_1.$ Let $f$ be a function satisfying the assumption of Proposition \ref{prop:6.11} and $\int_\mathcal{T}fd\mu=1$. For sufficiently large $t$, we have 
\begin{equation}\label{eq:6.48}
\int_{\mathcal{T}\times\mathcal{T}}p(t,x,y)f(x)f(y)d\mu(x)d\mu(y)\asymp_{1+\epsilon}p(t,x,x).
\end{equation}
Since $\epsilon$ is arbitrary, by \eqref{eq:6.48} and Proposition \ref{prop:6.11},
\begin{equation}
\lim_{t\rightarrow\infty}t^{3/2}e^{\lambda_0 t}p(t,x,x)=\frac{1}{2\sqrt{\pi}\mathcal{C}_0^{1/2}}c(x,x).
\end{equation}
Consider the case when $x\neq y$. Then choose a function $f_1$ and $f_2$ supported on compact sets containing $x$ and $y$. Proposition \ref{prop:6.11} show the rest part of proof when we use $f_1+f_2$.
\end{proof}
 

\begin{thebibliography}{xxx}
\bibitem[Ab]{Ab} L. M. Abramov. \emph{On the entropy of a flow,} Doklady Akad. Nauk. SSSR \textbf{128} (1959) 873--875.
 \bibitem[Ac]{Ac} A. Ancona, \emph{Negatively curved manifolds, elliptic operators and the Martin boundary}, Ann. of Math. \textbf{(2)} 125 (1987), no,3, 495--536. 
\bibitem[AK]{AK} W. Ambrose and S. Kakutani, \emph{Structure and continuity of measure flows,} Duke Math. J. \textbf{9} (1942), 25--42.
 \bibitem[Bo]{Bo}P. Bougerol, \emph{Th\'eor\`eme central limite local sur certains groupes de Lie,} Ann. Sci. \'Ecole Norm. Sup. (4) \textbf{14} (1981), 403--432.
\bibitem[BH]{BH}  M. R. Bridson and A. Haefleger, \emph{Metric spaces of non-positive curvature,} Grund. math. Wiss. \textbf{319}, Springer Verlag, 1999.
\bibitem[BP]{BP} A. Broise-Alamichel and F. Paulin, \emph{Sur le codage du flot géodésique dans un arbre}, Ann. Fac. Sci. Toulouse \textbf{16} (2007) 477--527
 \bibitem[BPP]{BPP} A. Broise-Alamichel, J. Parkkonen, and F. Paulin \emph{ Equidistribution and counting under equilibrium states in negatively curved spaces and graphs of groups. Applications to non-Archmedean Diophantine approximation,} Progress in Mathematics 329, Birkhäuser.
 \bibitem[C]{C}  M. Coornaert, \emph{Mesures de Patterson-Sullivan sur le bord d'un espace hyperbolicque au sens de Gromov}, Pacific Journal of Mathematics, \textbf{159} no.2, (1993) 241--270
 \bibitem[D]{D} D. Dolgopyat, \emph{Prevalence of rapid mixing in hyperbolic flows,} Ergo. Th. and Dynam. sys. \textbf{18} (1998), 1097--1114.
  \bibitem[BSSW]{BSSW} A. Bendikov, L. Saloff-Coste, M. Salvatori and W. Woess, \emph{The heat semigroup and Brownian motion on strip complexes}, Adv. in Math. \textbf{226} (2011), no.1, 992--1055.
  \bibitem[EW]{EW} M. Einsiedler and T. Ward, \emph{Functional Analysis, Spectral Theory, and Applications}, Graduate Texts in Mathematics, \textbf{276}, Springer 
   \bibitem[FOT]{FOT}M. Fukushima, Y. Oshima and M. Takeda, \emph{Dirichlet forms and symmetric Markov process}, De Gruyter Studies in Mathematics, \textbf{19} Walter de Gruyter, Berlin, 1994.
   \bibitem[G]{G}S. Gou\"ezel, \emph{Local limit theorem for symmetric random walks in Gromov-hyperbolic groups},  J. Amer. Math. Soc. \textbf{27} (2014), no.3, 893--928.
\bibitem[GL]{GL} S. Gou\"ezel and P. Lalley \emph{Random walks on co-compact Fuchsian groups} Ann. Sci. \'Ecole Norm. Sup. (4) \textbf{46} (2013), no. 1, 129--173.
 \bibitem[HL]{HL} S. Hong and S. Lim, \emph{Martin boundary of Brownian motion on hyperbolic graphs}, Discrete and Continuous Dynam. Sys., \textbf{41} no.8, (2021) 3725-3757. 
 \bibitem[IJT]{IJT} G. Immoi and T. Jotdan and M. Todd. \emph{Recurrence and transience for suspension flows,} Isr. J. Math. \textbf{209} (2015) no. 2.  547--592
 \bibitem[K]{K} V. A. Kaimanovich, \emph{Invariant measures of the geodesic flow and measures at infinity on negative curved manifolds}, Ann. Inst. H. Poincar\'e, A, Phys. Th\'eor., \textbf{53} (1990), 361--393. 
 \bibitem[La]{La} S. Lalley, \emph{Finite range random walks on free groups and homogeneous trees}, Ann. Prob. \textbf{21} (1993), 2078--2130.
 \bibitem[Le]{Le} F. Ledrappier, \emph{Structure au bord des vari\'et\'es \`a courbure n\'egative, in: S\'eminaire de Th\'eorie Spectrale et G\'eom\'etrie}, No. 13, ann\'ee 1994–1995, Univ. Grenoble I, Saint-Martin-d’H\`eres, (1995), pp. 97–122.
 \bibitem[Li]{Li} P. Li, \emph{Geometirc Analysis}, Cambridge University Press (2012).
 \bibitem[Lim]{Lim} S.Lim, \emph{Minimal volume entropy for graphs}, Trans. Amer. Math. Soc. \textbf{360} (2008), 5089--5100.
 \bibitem[LL]{LL} F. Ledrappier and S. Lim \emph{Lcoal limite theorem in negative curvature}, Duke Math. J. (8) \textbf{170} 1585--1681. 
   \bibitem[LSV]{LSV} D. Lenz, P, Stollmann and I. Veseli\'c, \emph{The Allegretto-Piepenbrink theorem for strong local dirichlet forms,} Documenta Math. \textbf{14} (2009)  
    \bibitem[LW]{LW} F. Ledrappier and X. Wang, \emph{An integral formula for the volume entropy with application to rigidity}, J. Differential Geom., \textbf{85} 3(2010),461--478.
\bibitem[Me]{Me} I. Melbourne, \emph{Rapid decay of correlations for nonuniformly hyperbolic flows}, Trans. AMS, \textbf{359}, (2007), 2421--2441. 
 \bibitem[Mo]{M} O. Mohsen, \emph{Le bas du spectre d'une vari\'et\'e hyperbolique est un point selle}, Ann. Sci. \'Ecole Norm. Sup. \textbf{40} (2007) 191--207.
 \bibitem[Mu]{Mun} J. R. Munkres, \emph{Topology}, Second edition, Prentice Hall, Inc., Upper Saddle River, NJ, 2000.
 \bibitem[MLS]{MLS} A.B. de Monvel, D. Lenz and P. Stollmann, \emph{Sch'nol's theorem for strongly local forms}, Israel J. Math.,\textbf{173} (2009), 189--211.
 \bibitem[PP]{PP} W. Parry and M. Pollicott, \emph{Zeta functions and the periodic orbit structure of hyperbolic dynamics}, Ast\'erisque \textbf{187–188} (1990).
 \bibitem[P]{P} S. Patterson, \emph{The limit set of a Fuchsian group}, Acta Math. \textbf{136} (1976), no. 3-4, 241--273.
 \bibitem[PPS]{PPS} F. Paulin, M. Pollicott, B. Schapira, \emph{Equilibrium states in negative curvature},  Ast\'erisque, \textbf{373} 1807 (2015).
 \bibitem[Q]{Q} JF. Quint  \emph{An overview of Patterson-Sullivan theory}, Workshop The barycenter method, FIM, Zurich, May 2006.
 \bibitem[R]{R} T. Roblin. \emph{Ergodicit\'e et \'equidistribution en courbure n\'egative. M\'emoire} Soc. Math. France, 95 (2003).
 \bibitem[S]{S} D. Sullivan, \emph{The density at infinity of a discrete group of hyperbolic motions}, Publications Math\'ematiques de L'Institut des Hautes Scientifiques \textbf{50} (1979) 171--202.
  \bibitem[St1]{St1} K-T Sturm, 
\emph{Analysis on local Dirichlet spaces-I . Recurrence, conservativeness and $L^p$-Liouville properties.} {J. Reine Angew. Math.}, \textbf{456} (1994) 173--196.
 \bibitem[St2]{St2} K-T Sturm, 
 \emph{Analysis on local Dirichlet spaces-II. Upper Gaussian estimates for the fundamental solutions of parabolic equations,} {Osaka J. Math.}, \textbf{32} (1995) 275--312.
\bibitem[St3]{St3} K-T Sturm, 
\emph{Analysis on local Dirichlet spaces-III. The parabolic Harnack inequality,} {J. Math. Pures Appl.}, \textbf{75} (1996), no. 3,  273--297.
 \bibitem[SW]{SW} L. Saloff-Coste, and W. Woess, \emph{Transition operators on co-compact G-spaces}, Rev. Mat. Iberoam. \textbf{22} (2006) 747--799.
 \end{thebibliography}
\end{document}